\ifdefined\pdfoutput\pdfoutput=1\fi
\documentclass[11pt]{amsart}
\usepackage{amsmath,amssymb,amsthm}
\usepackage[all]{xy}
\usepackage{iftex}
\ifpdf\else
  
\fi
\usepackage{tikz}
\usetikzlibrary{arrows.meta}

\newtheorem{theorem}{Theorem}[section]
\newtheorem{lemma}[theorem]{Lemma}
\newtheorem{proposition}[theorem]{Proposition}
\newtheorem{corollary}[theorem]{Corollary}
\newtheorem{maintheorem}{Theorem}

\theoremstyle{definition}
\newtheorem{definition}[theorem]{Definition}
\newtheorem{remark}[theorem]{Remark}
\newtheorem{example}[theorem]{Example}
\newtheorem{setting}[theorem]{Setting}

\numberwithin{equation}{section}

\renewcommand{\phi}{\varphi}

\newcommand{\ep}{\varepsilon}

\newcommand{\Coker}{\operatorname{Coker}}
\newcommand{\Fix}{\operatorname{Fix}}

\newcommand{\Homeo}{\operatorname{Homeo}}
\newcommand{\id}{\operatorname{id}}
\newcommand{\Ima}{\operatorname{Im}}
\newcommand{\Iso}{\operatorname{Iso}}
\newcommand{\Ker}{\operatorname{Ker}}
\newcommand{\Nor}{\operatorname{N}}
\newcommand{\St}{\operatorname{St}}
\newcommand{\supp}{\operatorname{supp}}

\newcommand{\N}{\mathbb{N}}
\newcommand{\Z}{\mathbb{Z}}
\newcommand{\Q}{\mathbb{Q}}

\newcommand{\cCO}{\mathcal{CO}}
\newcommand{\cG}{\mathcal{G}}
\newcommand{\cH}{\mathcal{H}}
\newcommand{\cK}{\mathcal{K}}

\newcommand{\sD}{\mathsf{D}}
\newcommand{\sF}{\mathsf{F}}
\newcommand{\sA}{\mathsf{A}}

\title[Maximal subgroups of topological full groups]
{Maximal subgroups of topological full groups \\
arising from factor maps and subgroupoids}
\author{Hiroki Matui}
\thanks{The author was supported by JSPS KAKENHI Grant Number JP23K22397.}
\address{Graduate School of Mathematical Sciences, 
The University of Tokyo, Japan}
\email{hiroki@ms.u-tokyo.ac.jp}
\date{}


\begin{document}

\begin{abstract}
We construct maximal subgroups of topological full groups 
and their commutator subgroups using factor maps 
and wide open subgroupoids of minimal effective ample groupoids 
with Cantor unit spaces, assuming almost finiteness or pure infiniteness. 
For two-to-one coverings and prime factor maps with principal target 
and fibers of cardinality at most two, 
surjectivity and injectivity of the same induced map 
on full-group abelianizations characterize maximality 
of the factor full group and factor commutator subgroup, respectively. 
For subgroupoids arising from prime-order cyclic actions 
free on the unit space, the normalizer of the smaller commutator subgroup 
in the ambient full group is a semidirect product, 
maximal exactly when the inclusion induces a surjection 
on full-group abelianizations. 
In both constructions, the embedded commutator subgroup 
has a unique maximal overgroup in the ambient commutator subgroup, 
namely its normalizer. 
For Cantor minimal systems, the factor construction yields 
partition stabilizers of the type in the Grigorchuk--Vorobets conjecture. 
The subgroupoid normalizers act topologically primitively. 
Groupoid homology makes the criteria computable 
for these systems and shifts of finite type. 
Examples include finitely generated nonsimple and locally finite simple 
maximal subgroups of infinite index 
in finitely generated infinite simple amenable groups. 
\end{abstract}

\maketitle

\section{Introduction}

Topological full groups encode dynamical information 
in the algebraic structure of groups. 
For Cantor minimal systems, Giordano, Putnam and Skau showed that 
the abstract isomorphism class of the topological full group 
determines the system up to flip conjugacy \cite{GPS99Israel}. 
Reconstruction results for ample groupoids extend this principle 
to a wider class of dynamical systems \cite{Ma15crelle}. 
The study of these groups has also revealed connections 
between dynamical and group-theoretic properties. 
For example, the simplicity and finite generation results 
for commutator subgroups in \cite{Ma06IJM}, 
together with the amenability theorem of Juschenko and Monod 
\cite{JM13Annals}, yield finitely generated infinite simple amenable groups. 
Nekrashevych's alternating full groups relate minimality and expansiveness 
to simplicity and finite generation in the groupoid setting 
\cite{Ne19ETDS}. 
Groupoid homology provides another connection: 
the index map and the description of abelianizations 
relate algebraic invariants of full groups 
to homological invariants of the underlying groupoids 
\cite{Ma12PLMS,Ma15crelle,Ma16Adv,Li25ForumPi}. 

These developments lead to the question of how natural relations 
between dynamical systems are reflected in the subgroup structure 
of their topological full groups. 
Factor maps and inclusions of subgroupoids both induce 
inclusions of full groups. 
When do these constructions give maximal subgroups? 
More specifically, how is the absence of intermediate dynamical objects 
related to the absence of intermediate subgroups? 
The present work was motivated by the study of maximal subgroups 
of ample groups by Grigorchuk and Vorobets \cite{GV26ETDS}. 
We investigate the questions above using ample groupoids, 
with particular emphasis on the role of abelianization 
and groupoid homology. 

For an effective ample groupoid $\cG$ with Cantor unit space, 
we write $\sF(\cG)$ for its topological full group 
and $\sD(\cG)$ for the commutator subgroup of $\sF(\cG)$. 
All groupoids in this paper are second countable and Hausdorff. 
Our main results concern minimal groupoids 
that are either almost finite or purely infinite. 
These classes include groupoids of Cantor minimal $\Z$ actions 
and groupoids of irreducible shifts of finite type (SFTs) 
with non-permutation adjacency matrices, respectively. 
In both classes, $\sD(\cG)$ is simple. 

We first consider factor maps. 
A factor map $\pi:\cG\to\cH$ is a proper surjective 
continuous homomorphism that is bijective on each range fiber 
(Definition~\ref{def:factor}). 
Pullback of compact open bisections gives an embedding 
\[
\pi^*:\sF(\cH)\longrightarrow\sF(\cG). 
\]
We call $\pi$ prime if it is not an isomorphism 
and has no nontrivial intermediate ample groupoid. 
Primeness is a natural necessary condition for maximality, 
but it does not by itself imply maximality of the induced subgroups. 
For the factor maps below, the remaining obstruction 
is precisely a condition on the induced map on abelianizations. 
Here $H_1(\Gamma)$ denotes the abelianization of a group $\Gamma$. 

\begin{maintheorem}\label{intro:factor}
Let $\pi:\cG\to\cH$ be a factor map between 
minimal effective ample groupoids with Cantor unit spaces, 
each of which is either almost finite or purely infinite. 
Suppose that either 
\begin{itemize}
\item $\pi$ is a covering map of degree two, or 
\item $\pi$ is prime, at most two-to-one, and $\cH$ is principal. 
\end{itemize}
Then the following hold. 
\begin{enumerate}
\item $\pi^*(\sF(\cH))$ is maximal in $\sF(\cG)$ 
if and only if 
\[
H_1(\pi^*):H_1(\sF(\cH))\longrightarrow H_1(\sF(\cG))
\]
is surjective. 
\item $\pi^*(\sD(\cH))$ is maximal in $\sD(\cG)$ 
if and only if $H_1(\pi^*)$ is injective. 
\end{enumerate}
\end{maintheorem}

This is proved in Theorems~\ref{thm1:factor} and \ref{thm2:factor}. 
The same argument identifies the maximal subgroup 
containing $\pi^*(\sD(\cH))$ even when this subgroup is not maximal. 
Writing $\Nor(\Gamma,S)$ for the normalizer of $S$ in $\Gamma$, 
we show that 
\[
\Nor(\sF(\cG),\pi^*(\sD(\cH)))=\pi^*(\sF(\cH)). 
\]
Moreover, 
\[
M:=\Nor(\sD(\cG),\pi^*(\sD(\cH)))
=\pi^*(\sF(\cH))\cap\sD(\cG)
\]
is the unique maximal subgroup of $\sD(\cG)$ 
containing $\pi^*(\sD(\cH))$, and 
\[
M/\pi^*(\sD(\cH))\cong\Ker H_1(\pi^*)
\]
(Corollary~\ref{cor:factornormalizer}). 
Thus the kernel of the map on abelianizations 
measures the enlargement needed to obtain a maximal subgroup. 

Our second construction starts with a wide open subgroupoid. 
Such an inclusion $\cG\subset\cH$ is called prime 
if it has no proper intermediate wide open subgroupoid. 
We study the case in which $\cH$ is a semidirect product 
of $\cG$ by a cyclic group of prime order. 
Under the hypotheses below, the subgroupoid inclusion is prime, 
but $\sF(\cG)$ is not maximal in $\sF(\cH)$ 
and $\sD(\cG)$ is not maximal in $\sD(\cH)$. 
The normalizer of $\sD(\cG)$ in $\sF(\cH)$ is 
$\sF(\cG)\rtimes_\lambda\Z/p$. 
This normalizer is the natural candidate for a maximal subgroup. 

\begin{maintheorem}\label{intro:subgroupoid}
Let $p$ be a prime number, 
let $\lambda:\Z/p\curvearrowright\cG$ be an action 
on an ample groupoid with Cantor unit space, 
and put $\cH:=\cG\rtimes_\lambda\Z/p$. 
Assume that $\lambda$ is free on $\cG^{(0)}$ 
and that $\cG$ and $\cH$ are effective and minimal, 
each being either almost finite or purely infinite. 
Let $\iota:\cG\to\cH$ and 
$\iota_*:\sF(\cG)\to\sF(\cH)$ be the inclusions. 
Then the following hold. 
\begin{enumerate}
\item $\sF(\cG)\rtimes_\lambda\Z/p$ is maximal in $\sF(\cH)$ 
if and only if $H_1(\iota):H_1(\cG)\longrightarrow H_1(\cH)$ is surjective, 
or equivalently, if and only if $H_1(\iota_*)$ is surjective. 
\item $\Nor(\sD(\cH),\sD(\cG))$ is the unique maximal subgroup 
of $\sD(\cH)$ containing $\sD(\cG)$. 
\item If every $\lambda_a$ belongs to $\sD(\cH)$, 
then $\sD(\cG)\rtimes_\lambda\Z/p$ is maximal in $\sD(\cH)$ 
if and only if $H_1(\iota_*)$ is injective. 
The additional hypothesis holds automatically when $p$ is odd. 
\end{enumerate}
\end{maintheorem}

These assertions follow from Corollary~\ref{cor:normalizermaximal} 
and Theorems~\ref{thm1:subgroupoid} and \ref{thm2:subgroupoid}. 
When $p=2$ and the nontrivial element of the action 
does not belong to $\sD(\cH)$, 
we give a separate description of the normalizer 
in Theorem~\ref{thm3:subgroupoid}. 
Under the injectivity condition, it is an extension of $\sD(\cG)$ 
by $\Z/2$, and examples show that this extension need not split 
(Remark~\ref{rem:nonsplit:subgroupoid}). 

The proofs of Theorems~\ref{intro:factor} and \ref{intro:subgroupoid} 
have a common structure. 
Starting with the embedded commutator subgroup, 
we show that adjoining an element outside its normalizer 
produces a group containing the ambient commutator subgroup. 
The main step is to construct local alternating groups 
using multisections and then propagate them throughout the unit space. 
For factor maps, primeness is used through invariant Boolean algebras 
and the topology of the unit space. 
For subgroupoids, the cyclic action allows local permutations 
to connect the different layers of the semidirect product. 
Once the ambient commutator subgroup has been obtained, 
maximality reduces to the induced map on abelianizations. 

To make these criteria computable, we use the exact sequence 
\[
H_0(\cG,\Z/2)\longrightarrow H_1(\sF(\cG))
\longrightarrow H_1(\cG)\longrightarrow0,
\]
which is part of Li's exact sequence \cite[Corollary E]{Li25ForumPi}. 
The maps induced by factor maps and subgroupoid inclusions 
are compatible with this sequence. 
Consequently, groupoid homology can be used 
to test the surjectivity and injectivity conditions above. 
The distinction between the two conditions is essential: 
maximality for full groups and for their commutator subgroups 
can behave differently even in elementary examples. 

We apply these results to several classes of dynamical systems. 
For almost one-to-one extensions of Cantor minimal $\Z$ systems, 
the realization theorem of Giordano, Putnam and Skau 
\cite{GPS01MathScand} provides extensions 
with prescribed inclusions of dimension groups. 
We show that their construction yields prime factor maps 
to which Theorem~\ref{intro:factor} applies. 
The resulting factor commutator subgroups are always maximal, 
while maximality of the factor full groups 
is equivalent to $Q\otimes\Z/2=0$, 
where $Q$ is the prescribed torsion-free quotient 
of the dimension group of the extension 
by the pullback of the dimension group of the base 
(Theorem~\ref{almost1to1:Z}). 
For two-to-one coverings of Cantor minimal $\Z$ systems, 
both factor subgroups fail to be maximal, 
and the unique maximal subgroup containing the factor commutator subgroup 
is a split extension by $\Z/2$ 
(Proposition~\ref{prop:Zcovering}). 
For the Thue--Morse system, this gives a finitely generated, nonsimple 
maximal subgroup of a finitely generated infinite simple amenable group; 
the maximal subgroup contains a simple subgroup of index two 
(Example~\ref{ex:ThueMorse}). 
Finite order automorphisms of Cantor minimal systems 
provide examples for the subgroupoid construction as well. 
For involutions acting trivially on the dimension group $(D,D^+,u)$, 
the resulting maximal subgroup of $\sD(\cH)$ 
is an extension of $\sD(\cG)$ by $\Z/2$. 
This extension splits exactly when $u\in4D$ 
(Corollary~\ref{cor:finiteorderZ} and 
Remark~\ref{rem:nonsplit:subgroupoid}). 

For SFT groupoids, we express the induced homology maps 
in terms of finite integer matrices. 
This gives explicit maximality criteria for both constructions 
(Propositions~\ref{SFTcovering} and \ref{SFTsemidirect}). 
Already for graphs with one vertex, 
the criteria exhibit different maximality phenomena 
for full groups and commutator subgroups. 
We also consider orbit-breaking subgroupoids of Cantor minimal $\Z$ systems. 
A direct argument shows that the associated locally finite full group 
is maximal in the kernel of the index map, 
and its commutator subgroup is maximal 
in the ambient commutator subgroup 
(Theorem~\ref{thm:orbitbreakingmaximal}). 
For minimal subshifts, this yields locally finite simple maximal subgroups 
of finitely generated infinite simple amenable groups. 
These maximal subgroups and those in the Thue--Morse example 
have infinite index, since an infinite simple group 
has no proper subgroup of finite index. 

Our results can be viewed in terms of stabilizers of partitions.
Grigorchuk and Vorobets \cite[Section 1]{GV26ETDS} develop an analogy
with the intransitive, imprimitive and primitive cases for finite
symmetric groups, the last of which is governed by the
O'Nan--Scott theorem.
In their topological setting, minimality replaces transitivity,
and only partitions into closed sets are considered.
They characterize maximal subgroups arising from
closed sets and clopen partitions in \cite{GV26ETDS}.
Their bold conjecture for Cantor minimal $\Z$ systems 
\cite[Conjecture 1.8]{GV26ETDS} predicts that every maximal subgroup 
is a stabilizer of a nontrivial closed set, 
a stabilizer of a nontrivial partition into closed sets, 
or a normal subgroup of prime index, 
with additional conditions in each case. 
For an equivariant factor map $\pi:X\to Y$ 
between Cantor minimal $\Z$ systems, 
the embedded factor full group is precisely 
the stabilizer of the partition $\{\pi^{-1}(y)\mid y\in Y\}$. 
Indeed, an element preserving this partition 
has an orbit cocycle constant on each fiber, 
since the action on $Y$ is free, and this cocycle descends 
to a continuous function on $Y$. 
Its induced action on the quotient $Y$ is minimal. 
The finite-fiber extensions considered here give infinite partitions 
into finite closed sets, beyond the clopen partitions 
covered by their theorems. 
In particular, Theorem~\ref{almost1to1:Z} gives a maximality criterion 
for stabilizers of this kind in the ambient full group, 
as contemplated in part~(ii) of their conjecture. 

In the subgroupoid setting of Theorem~\ref{intro:subgroupoid}, 
the groups $\sF(\cG)\rtimes_\lambda\Z/p$ and $\Nor(\sD(\cH),\sD(\cG))$ 
are the stabilizers of the partition into $\cG$-orbits 
in $\sF(\cH)$ and $\sD(\cH)$, respectively 
(Lemma~\ref{normalizer=semidirectproduct}). 
They permute these dense, nonclosed orbit classes, 
yet act topologically primitively: they are minimal and preserve 
no nontrivial partition into closed sets 
(Proposition~\ref{prop:topologicalprimitivity}). 
This distinguishes them from the closed-fiber partition stabilizers 
in the factor construction. 

Belk, Bleak, Quick and Skipper \cite{BBQS25TAMS} use type systems 
to construct maximal subgroups of Thompson's group $V$, 
including an uncountable family of pairwise nonisomorphic examples. 
The subgroup $V_2\rtimes_\lambda\Z/3$ 
in Example~\ref{SFTsemidirect3} is both a type-system stabilizer 
from their construction and a normalizer 
arising from our subgroupoid construction. 

In the companion paper \cite{Ma26maximal2}, 
we study maximal stabilizers of closed sets, their isomorphism invariants 
and realization results for general ample groupoids. 
The companion paper \cite{Ma26maximal3} develops subgroupoid graphs, 
a trichotomy for prime subgroupoids and 
maximality criteria for SFT groupoids. 
It also extends the type-system constructions of \cite{BBQS25TAMS} 
to general irreducible SFT groupoids, including uncountable families 
of pairwise nonisomorphic maximal subgroups 
acting minimally on the unit space. 

Suzuki's correspondence \cite{Su20CMP} 
between intermediate dynamical extensions and intermediate $C^*$-algebras 
of reduced crossed products implies that the prime finite-fiber 
extensions of Cantor minimal $\Z$ systems considered here 
give maximal crossed-product inclusions. 
Every intermediate space is again a Cantor set, 
by finite fibers and minimality, 
so primeness excludes nontrivial intermediate extensions. 
The full groups retain an additional obstruction 
in the induced map on abelianizations: 
the two-to-one coverings in Proposition~\ref{prop:Zcovering} 
give maximal crossed-product inclusions although neither 
the factor full group nor the factor commutator subgroup 
is maximal in its respective ambient group. 

The paper is organized as follows. 
Section~\ref{sec:preliminaries} collects the tools 
for groupoids, full groups and homology used in the proofs. 
Section~\ref{sec:factor} establishes the criteria for factor maps, 
and Section~\ref{sec:examplesI} gives their applications 
to Cantor minimal $\Z$ systems and SFT groupoids. 
Section~\ref{sec:subgroupoid} proves the subgroupoid results. 
Section~\ref{sec:examplesII} presents examples from finite cyclic actions 
and concludes with the orbit-breaking construction.

\section{Preliminaries}\label{sec:preliminaries}

We collect the tools used in the proofs of maximality. 
For minimal almost finite and purely infinite groupoids, 
we record simplicity and local extension properties of full groups. 
We then describe the maps induced by factor maps and 
the multisections used to generate commutator subgroups.

\subsection{Ample groupoids}\label{subsec:amplegroupoids}

We fix notation and recall the basic terminology for ample groupoids. 

The cardinality of a set $A$ is written $\#A$ and 
the characteristic function of $A$ is written $1_A$. 
The finite cyclic group of order $p$ is denoted by $\Z/p$. 
We say that a subset of a topological space $X$ is clopen 
if it is both closed and open. 
We let $\cCO(X)$ denote 
the set of all nonempty compact open subsets of $X$. 
A topological space is said to be totally disconnected 
if its connected components are singletons. 
By a Cantor set, 
we mean a nonempty compact, metrizable, totally disconnected space 
with no isolated points. 
It is known that any two such spaces are homeomorphic. 

In this article, by an \'etale groupoid 
we mean a second countable locally compact Hausdorff groupoid 
such that the range map is a local homeomorphism. 
(We emphasize that our \'etale groupoids are always assumed to be Hausdorff, 
while non-Hausdorff groupoids are also studied actively.) 
We refer the reader to \cite{Re_text,Re08Irish} 
for background material on \'etale groupoids. 
For an \'etale groupoid $\cG$, 
we let $\cG^{(0)}$ denote the unit space and 
let $s$ and $r$ denote the source and range maps, 
i.e.\ $s(g)=g^{-1}g$, $r(g)=gg^{-1}$. 
A subset $U\subset\cG$ is called a bisection 
if $r|U$ and $s|U$ are injective. 
The \'etale groupoid $\cG$ has a basis for its topology 
consisting of open bisections. 
For $x\in\cG^{(0)}$, 
the set $r(s^{-1}(x))$ is called the orbit 
(or $\cG$-orbit) of $x$. 
When every orbit is dense in $\cG^{(0)}$, $\cG$ is said to be minimal. 
An \'etale groupoid $\cG$ is called ample 
if its unit space $\cG^{(0)}$ is totally disconnected. 
An \'etale groupoid is ample if and only if 
it has a basis for its topology consisting of compact open bisections. 

For $x\in\cG^{(0)}$, 
we write $\cG_x=r^{-1}(x)\cap s^{-1}(x)$ and 
call it the isotropy group of $x$. 
The isotropy subgroupoid of $\cG$ is 
$\Iso(\cG):=\{g\in\cG\mid r(g)=s(g)\}=\bigsqcup_{x\in\cG^{(0)}}\cG_x$. 
The interior $\Iso(\cG)^\circ$ is an open subgroupoid of $\cG$. 
We say that $\cG$ is principal if $\Iso(\cG)=\cG^{(0)}$. 
We say that $\cG$ is effective 
if the interior of $\Iso(\cG)$ is $\cG^{(0)}$. 
The term essentially principal is also used for this condition.
In our second countable Hausdorff setting, effectiveness implies that
the units with trivial isotropy are dense
(see \cite[Proposition 3.6]{Re08Irish}). 
In this paper, unless specified otherwise, 
an ample groupoid is always assumed to be effective. 
A subgroupoid $\cH\subset\cG$ is said to be wide if $\cH^{(0)}=\cG^{(0)}$. 
A wide open subgroupoid of an ample groupoid is again ample. 
If $\cG$ is effective, so is every wide open subgroupoid $\cH$,
since $\Iso(\cH)^\circ\subset\Iso(\cG)^\circ=\cG^{(0)}$. 

For an ample groupoid $\cG$ and an open subset $Y\subset\cG^{(0)}$, 
we write
\[
\cG|Y:=r^{-1}(Y)\cap s^{-1}(Y), 
\]
which is again an ample groupoid, with unit space $Y$. 
Since $\cG$ is effective, so is $\cG|Y$. 
We say that $Y$ is full if $Y$ meets every $\cG$-orbit. 
Two ample groupoids $\cG_1$ and $\cG_2$ are said to be 
Kakutani equivalent 
if there exist full clopen subsets $Y_i\subset\cG_i^{(0)}$ 
such that $\cG_1|Y_1$ and $\cG_2|Y_2$ are isomorphic. 
In particular, $\cG$ and $\cG|Y$ are Kakutani equivalent 
for every full clopen subset $Y\subset\cG^{(0)}$.

We use the following conventions for semidirect products and skew products 
(see \cite{Re_text}). 
Let $\Gamma$ be a countable discrete group. 
For an action $\lambda:\Gamma\curvearrowright\cG$, 
the semidirect product $\cG\rtimes_\lambda\Gamma$ is $\cG\times\Gamma$ 
with the product topology and unit space identified with $\cG^{(0)}$. 
Its range, source and multiplication are given by 
\[
r(g,\gamma)=r(g),\qquad s(g,\gamma)=\lambda_\gamma^{-1}(s(g)),\qquad
(g,\gamma)(h,\delta)=(g\lambda_\gamma(h),\gamma\delta). 
\]
When $\cG$ is a unit groupoid $X$, this is the transformation groupoid 
$X\rtimes_\lambda\Gamma$. 

For a continuous homomorphism $\xi:\cG\to\Gamma$, 
the skew product $\cG\times_\xi\Gamma$ is $\cG\times\Gamma$ 
with the product topology and unit space $\cG^{(0)}\times\Gamma$, where 
\[
r(g,\gamma)=(r(g),\gamma),\qquad
s(g,\gamma)=(s(g),\gamma\xi(g)),\qquad
(g,\gamma)(h,\gamma\xi(g))=(gh,\gamma). 
\]
The action $\delta\cdot(g,\gamma)=(g,\delta\gamma)$ is by automorphisms. 
When $\Gamma$ is finite, the projection $\cG\times_\xi\Gamma\to\cG$ 
is a covering map, and these automorphisms are called deck transformations.

\subsection{Topological full groups, homology and the index map}
\label{subsec:homology}

The index map relates topological full groups to groupoid homology. 
We introduce these objects together, 
fixing the conventions for the induced maps used below. 

A bisection $U\in\cCO(\cG)$ is said to be full when $r(U)=\cG^{(0)}=s(U)$. 
For a full bisection $U$, we define $\theta_U\in\Homeo(\cG^{(0)})$ 
by $\theta_U:=r\circ(s|U)^{-1}$. 

\begin{definition}[{\cite[Definition 2.3]{Ma12PLMS}}]
\label{def:topologicalfullgroup}
Let $\cG$ be an effective ample groupoid 
whose unit space $\cG^{(0)}$ is compact. 
The set of all $\alpha\in\Homeo(\cG^{(0)})$, 
for which there exists a full bisection $U\subset\cG$ 
such that $\alpha=\theta_U$, 
is called the topological full group of $\cG$ 
and is denoted by $\sF(\cG)$. 
\end{definition}

The full bisection implementing $\alpha\in\sF(\cG)$ is unique
by effectiveness.
The group $\sF(\cG)$ is a countable subgroup of $\Homeo(\cG^{(0)})$,
since $\cG$ has countably many compact open subsets. 
A homeomorphism $\alpha:\cG^{(0)}\to\cG^{(0)}$ belongs to $\sF(\cG)$ 
if and only if for any $x\in\cG^{(0)}$ 
there exists a bisection $V\in\cCO(\cG)$ 
such that $x$ is in $s(V)$ and 
$\alpha$ equals $r\circ(s|V)^{-1}$ on a neighborhood of $x$. 
We let $\sD(\cG)$ denote the derived subgroup (commutator subgroup) 
of $\sF(\cG)$. 

When $\Gamma$ is a group and $S\subset\Gamma$ is a subset, 
we denote the normalizer of $S$ in $\Gamma$ by 
\[
\Nor(\Gamma,S):=\{\alpha\in\Gamma\mid\alpha S\alpha^{-1}=S\}. 
\]

For $\alpha\in\Homeo(\cG^{(0)})$ we write 
\[
\supp(\alpha):=\overline{\{x\in\cG^{(0)}\mid\alpha(x)\neq x\}}. 
\]
When $\alpha=\theta_U$ belongs to $\sF(\cG)$, 
one has $\supp(\alpha)=s(U\setminus\cG^{(0)})$, 
and so $\supp(\alpha)$ is a clopen subset of $\cG^{(0)}$ 
by \cite[Lemma 2.2]{Ma15crelle}. 

For a clopen set $V\subset\cG^{(0)}$, we write 
\[
\sF(\cG)_V:=\{\alpha\in\sF(\cG)\mid\supp(\alpha)\subset V\}. 
\]
Extending homeomorphisms by the identity outside $V$ identifies 
$\sF(\cG|V)$ with $\sF(\cG)_V$.

For groupoid homology, we follow \cite[Section 3]{Ma12PLMS}. 
Let $A$ be a topological abelian group. 
For a locally compact Hausdorff totally disconnected space $X$, 
we let $C_c(X,A)$ denote the abelian group of 
compactly supported continuous functions from $X$ to $A$. 
When $\rho:X\to Y$ is a local homeomorphism between such spaces, 
one obtains a homomorphism $\rho_*:C_c(X,A)\to C_c(Y,A)$ by 
\[
\rho_*f(y):=\sum_{x\in\rho^{-1}(y)}f(x), 
\]
only finitely many terms being nonzero, 
since $\rho^{-1}(y)\cap\supp f$ is compact and discrete. 
For an ample groupoid $\cG$ and $n\geq1$, 
we let $\cG^{(n)}$ denote the space of composable $n$-strings 
\[
\cG^{(n)}:=\left\{(g_1,g_2,\dots,g_n)\in\cG^n\mid
s(g_i)=r(g_{i+1})\ \text{for}\ i=1,2,\dots,n{-}1\right\}, 
\]
which is a closed subset of $\cG^n$, and hence is again 
a second countable locally compact Hausdorff totally disconnected space. 
The face maps $d_i:\cG^{(n)}\to\cG^{(n-1)}$ are given by 
\[
d_i(g_1,\dots,g_n):=
\begin{cases}
(g_2,\dots,g_n) & i=0, \\
(g_1,\dots,g_ig_{i+1},\dots,g_n) & 1\leq i\leq n{-}1, \\
(g_1,\dots,g_{n-1}) & i=n, 
\end{cases}
\]
with the convention $d_0(g)=s(g)$ and $d_1(g)=r(g)$ for $n=1$. 
Each face map is a local homeomorphism. 

\begin{definition}\label{def:groupoidhomology}
The groupoid homology $H_*(\cG,A)$ is the homology of the chain complex 
\[
(C_c(\cG^{(n)},A),\partial_n)_{n\geq0}, 
\]
where $\partial_0:=0$ and
$\partial_n:=\sum_{i=0}^n(-1)^id_{i*}$ for $n\geq1$. 
As usual we write $H_*(\cG):=H_*(\cG,\Z)$. 
\end{definition}

As explained in the discussion preceding \cite[Definition 3.4]{Ma12PLMS}, 
when a homomorphism $\rho:\cG\to\cH$ between ample groupoids 
is a local homeomorphism, 
it gives rise to well-defined homomorphisms 
\[
H_i(\rho):H_i(\cG,A)\longrightarrow H_i(\cH,A),\qquad i\geq0, 
\]
via the pushforwards $C_c(\cG^{(n)},A)\to C_c(\cH^{(n)},A)$. 
This applies to the inclusion $\iota:\cG\to\cH$ 
of an open subgroupoid $\cG\subset\cH$, 
for which the pushforward is the extension by zero.

We can now define the index map \cite[Section 7]{Ma12PLMS}. 
Let $\cG$ be an effective ample groupoid with compact unit space. 
For $\alpha\in\sF(\cG)$, let $U$ be the unique full bisection 
such that $\alpha=\theta_U$.
Since $\partial_1(1_U)=1_{s(U)}-1_{r(U)}=0$, we may define 
\[
I:\sF(\cG)\longrightarrow H_1(\cG),\qquad I(\alpha):=[1_U]. 
\]
For two such bisections $U,V$, 
the boundary of the characteristic function 
of $(U\times V)\cap\cG^{(2)}$ is $1_V-1_{UV}+1_U$. 
Thus $[1_{UV}]=[1_U]+[1_V]$ in $H_1(\cG)$,
so $I$ is a homomorphism, called the index map.

\subsection{SFT groupoids}\label{subsec:graphgroupoid}

Groupoids of one-sided shifts of finite type (SFT) will provide 
the purely infinite examples in this paper. 
Let $(W,F)$ be a finite directed graph without sinks, 
where $W$ is the vertex set and $F$ is the edge set. 
For $f\in F$, let $i(f)$ and $t(f)$ denote its initial and terminal vertices. 
For $w\in W$, put $wF:=\{f\in F\mid i(f)=w\}$; 
the absence of sinks means that $wF\neq\emptyset$ for every $w\in W$. 
A finite path is a word $\mu=\mu_1\cdots\mu_n$ with 
$t(\mu_j)=i(\mu_{j+1})$ for $1\leq j<n$. 
We write $\lvert\mu\rvert=n$, $i(\mu)=i(\mu_1)$ and $t(\mu)=t(\mu_n)$. 
Paths of length zero are vertices, with $i(w)=t(w)=w$. 
Concatenation $\mu\nu$ is defined when $t(\mu)=i(\nu)$. 

Let 
\[
X_{(W,F)}:=\left\{(x_k)_{k\in\N}\in F^\N
\mid t(x_j)=i(x_{j+1})\quad\forall j\in\N\right\}
\]
be the infinite path space with the relative product topology, 
and let $\sigma$ be the shift. 
For a finite path $\mu$, let $Z(\mu)$ be the set of infinite paths 
beginning with $\mu$; for a vertex $w$, 
put $Z(w):=\{x\in X_{(W,F)}\mid i(x_1)=w\}$. 
These nonempty clopen cylinder sets form a basis for the topology of 
the compact, metrizable, totally disconnected space $X_{(W,F)}$. 
A finite set of finite paths is called a complete prefix set 
if every infinite path has exactly one member of the set as a prefix. 

The shift $\sigma$ is a local homeomorphism. 
The associated SFT groupoid is 
\[
\cG_{(W,F)}:=\left\{(x,n,y)\in X_{(W,F)}\times\Z\times X_{(W,F)}\mid 
\exists k,l\geq0,\ n=k{-}l,\ \sigma^k(x)=\sigma^l(y)\right\}. 
\]
Its range and source are $r(x,n,y)=x$ and $s(x,n,y)=y$,
and multiplication is given by $(x,n,y)(y,m,z)=(x,n+m,z)$.
For finite paths $\mu,\nu$ with $t(\mu)=t(\nu)$, the sets 
\[
Z(\nu,\mu):=\left\{(\nu z,\,\lvert\nu\rvert-\lvert\mu\rvert,\,\mu z)
\mid z\in Z(t(\mu))\right\}
\]
form a basis of compact open bisections. 
The bisection $Z(\nu,\mu)$ has range $Z(\nu)$ and source $Z(\mu)$, 
and implements the prefix replacement $\mu z\mapsto\nu z$. 

The adjacency matrix $A=(A(u,v))_{u,v\in W}$ is defined by 
\[
A(u,v):=\#\{f\in F\mid i(f)=u,\ t(f)=v\}. 
\]
It is irreducible if for all $u,v\in W$ there exists $n\in\N$ 
with $A^n(u,v)>0$. 
If $A$ is irreducible and is not a permutation matrix, 
then $X_{(W,F)}$ is a Cantor set and $\cG_{(W,F)}$ is a minimal 
effective ample groupoid, by \cite[Lemma 6.1]{Ma15crelle}. 
See also \cite{Ma17Abel} for background on SFT groupoids. 

For the covering constructions below, we also fix the following terminology.
For finite directed graphs $(W,F)$ and $(W',F')$ without sinks, 
a graph homomorphism $\rho:(W,F)\to(W',F')$ consists of maps 
on vertices and edges preserving $i$ and $t$. 
It is right-covering if $\rho:wF\to\rho(w)F'$ is bijective 
for every $w\in W$.

\subsection{Two classes of ample groupoids}

Almost finite and purely infinite groupoids are the two classes 
to which our maximality results apply. 
Although their definitions are quite different, 
their full groups share the simplicity and 
local extension properties needed in the proofs. 

A subgroupoid $\cK\subset\cG$ is elementary 
if it is compact, open and principal, with $\cK^{(0)}=\cG^{(0)}$. 
We write $\cK x:=\{g\in\cK\mid s(g)=x\}$. 

\begin{definition}[{\cite[Definition 6.2]{Ma12PLMS}}]\label{af/def}
Let $\cG$ be an ample groupoid with $\cG^{(0)}$ compact. 
We say that $\cG$ is almost finite 
if for any compact subset $C\subset\cG$ and $\ep>0$ 
there exists an elementary subgroupoid $\cK\subset\cG$ such that 
\[
\frac{\#(C\cK x\setminus\cK x)}{\#(\cK x)}<\ep
\]
for all $x\in\cG^{(0)}$. 
\end{definition}

Thus the finite fibers $\cK x$ are 
approximately invariant under multiplication by $C$. 

\begin{definition}[{\cite[Definition 4.9]{Ma15crelle}}]\label{pi/def}
Let $\cG$ be an ample groupoid with $\cG^{(0)}$ compact. 
We say that $\cG$ is purely infinite 
if for every clopen set $A\subset\cG^{(0)}$, 
there exist compact open bisections $U,V\subset\cG$ 
such that $s(U)=s(V)=A$, $r(U)\sqcup r(V)\subset A$. 
\end{definition}

By \cite[Lemma 6.1]{Ma15crelle}, 
the SFT groupoid of a finite directed graph 
with irreducible non-permutation adjacency matrix is purely infinite. 

Almost finiteness is invariant under Kakutani equivalence 
for ample groupoids with compact unit spaces by \cite{ABBL23ETDS}. 
In particular, 
if $\cG$ is minimal and either almost finite or purely infinite, 
then so is $\cG|W$ for every $W\in\cCO(\cG^{(0)})$: 
in the almost finite case $W$ is full, and 
in the purely infinite case this follows directly from the definition.

The following theorem collects the simplicity, index surjectivity 
and cancellation results used throughout the paper. 

\begin{theorem}\label{niceproperties}
Let $\cG$ be an ample groupoid with $\cG^{(0)}$ a Cantor set. 
Suppose that $\cG$ is either almost finite or purely infinite. 
\begin{enumerate}
\item If $\cG$ is minimal, then 
any non-trivial subgroup of $\sF(\cG)$ normalized 
by the commutator subgroup $\sD(\cG)$ contains $\sD(\cG)$. 
In particular, $\sD(\cG)$ is simple. 
\item The index map $I:\sF(\cG)\to H_1(\cG)$ is surjective. 
\item If $\cG$ is minimal, then $\cG$ has cancellation, i.e. 
for any $A,B\in\cCO(\cG^{(0)})$ with $[1_A]=[1_B]$ in $H_0(\cG)$, 
there exists a bisection $U\in\cCO(\cG)$ such that $r(U)=A$ and $s(U)=B$. 
\end{enumerate}
\end{theorem}

\begin{proof}
Assertion (1) follows from \cite[Theorems 4.7 and 4.16]{Ma15crelle};
(2) follows from \cite[Theorem 7.5]{Ma12PLMS} and
\cite[Theorem 5.2]{Ma15crelle};
and (3) follows from \cite[Theorem 6.12]{Ma12PLMS} and
\cite[Theorem 4.2 (3)]{Ma16Adv}. 
\end{proof}

The term cancellation in Theorem~\ref{niceproperties} (3) is motivated 
by cancellation of projections in $C^*$-algebra $K$-theory. 
For a $C^*$-algebra $D$, 
this means that projections $p,q$ in matrix algebras over $D$ 
with the same class in $K_0(D)$ are Murray--von Neumann equivalent. 
The word refers to cancelling a common direct summand: 
if $p\oplus r$ and $q\oplus r$ are equivalent, 
then $p$ and $q$ are equivalent. 
In our setting, nonempty clopen sets play the role of projections, 
and a compact open bisection with source $B$ and range $A$ 
implements an equivalence between $A$ and $B$. 
Thus cancellation says that equality of their homology classes is 
realized by such a concrete equivalence.

To control the abelianization of the full group, 
we use Li's exact sequence. 
Its hypothesis is comparison: 
for nonempty clopen sets $A,B\subset\cG^{(0)}$ 
with $\mu(A)<\mu(B)$ for every invariant probability measure $\mu$, 
there is a compact open bisection $U$ with $s(U)=A$ and $r(U)\subset B$. 
Here invariance means that $\mu(s(V))=\mu(r(V))$ 
for every compact open bisection $V$. 
Minimal almost finite and minimal purely infinite groupoids have comparison; 
see \cite[Section 2.1]{Li25ForumPi}. 

\begin{theorem}[{\cite[Corollary E]{Li25ForumPi}}]\label{Li}
Let $\cG$ be an ample groupoid with unit space a Cantor set. 
Assume that $\cG$ is minimal and has comparison. 
Then there is an exact sequence 
\[
\xymatrix@M=8pt{
H_2(\sD(\cG)) \ar[r] & H_2(\cG) \ar[r] & H_0(\cG,\Z/2) \ar[r]^-\zeta & 
H_1(\sF(\cG)) \ar[r]^-\eta & H_1(\cG) \ar[r] & 0. 
}
\]
The map $\zeta$ coincides with the one in \cite[Section 7]{Ne19ETDS} 
and the map $\eta$ is induced by the index map $I:\sF(\cG)\to H_1(\cG)$. 
\end{theorem}

\begin{remark}
The last three nonzero terms form the AH exact sequence conjectured 
in \cite{Ma16Adv}. 
In particular, Li's theorem establishes this conjecture 
for all minimal almost finite and minimal purely infinite groupoids. 
\end{remark}

The map $\zeta$ in Theorem~\ref{Li} is described by transpositions. 
When $U\in\cCO(\cG)$ is a bisection satisfying $r(U)\cap s(U)=\emptyset$, 
one can define $\tau_U\in\sF(\cG)$ by $\tau_U:=\theta_W$, 
where 
\[
W:=U\sqcup U^{-1}\sqcup\left(\cG^{(0)}\setminus(r(U)\cup s(U))\right). 
\]
The element $\tau_U$ is a transposition, i.e. $\tau_U^2=1$. 
Under the hypotheses of Theorem~\ref{Li}, 
$\zeta$ sends $[1_{r(U)}]\in H_0(\cG,\Z/2)$ to 
the class of $\tau_U$ in $H_1(\sF(\cG))$. 
In particular, if $[1_{r(U)}]=0$ in $H_0(\cG,\Z/2)$, 
then $\tau_U\in\sD(\cG)$.

The following elementary observation will be used repeatedly. 
It provides elements of commutator subgroups 
with prescribed local action and small support. 

\begin{lemma}\label{threecycle}
Let $\cH$ be an ample groupoid with compact unit space 
and let $U,V\in\cCO(\cH)$ be bisections with $s(U)=s(V)$ 
such that $s(U)$, $r(U)$ and $r(V)$ are mutually disjoint. 
\begin{enumerate}
\item The element $\tau_V\tau_U$ equals the commutator $[\tau_U,\tau_V]$ 
and has order three, with 
\[
\supp(\tau_V\tau_U)=s(U)\cup r(U)\cup r(V). 
\]
In particular $\tau_V\tau_U\in\sD(\cH)$. 
\item The element $\tau_V\tau_U$ agrees with $\theta_U$ on $s(U)$. 
\end{enumerate}
\end{lemma}

\begin{proof}
Both assertions follow directly from the definitions. 
\end{proof}

The next lemma says that the abelianization of $\sF(\cG)$ is already 
realized by elements supported in an arbitrarily small clopen set. 
This allows extensions of partial homeomorphisms 
to be adjusted to lie in the commutator subgroup. 

\begin{lemma}\label{localabelianization}
Let $\cG$ be a minimal ample groupoid with Cantor unit space 
which is either almost finite or purely infinite. 
For any $W\in\cCO(\cG^{(0)})$, 
the canonical embedding $\sF(\cG|W)\to\sF(\cG)$ 
induces a surjection $H_1(\sF(\cG|W))\to H_1(\sF(\cG))$. 
Equivalently, $\sF(\cG)=\sF(\cG)_W\cdot\sD(\cG)$. 
\end{lemma}

\begin{proof}
Let $\pi:\sF(\cG)\to H_1(\sF(\cG))$ be the quotient map and 
put $\Lambda:=\pi(\sF(\cG)_W)$.
Let $\zeta$ and $\eta$ be as in Theorem~\ref{Li}. 

Since $W$ is full, the inclusion $\cG|W\to\cG$ induces 
an isomorphism $H_1(\cG|W)\to H_1(\cG)$ 
by \cite[Proposition 3.5 and Theorem 3.6]{Ma12PLMS}. 
The groupoid $\cG|W$ is 
again minimal and either almost finite or purely infinite, 
so its index map is surjective by Theorem~\ref{niceproperties} (2). 
Naturality of the index map therefore gives $\eta(\Lambda)=H_1(\cG)$. 

Choose a bisection $U\in\cCO(\cG|W)$ with $s(U)\cap r(U)=\emptyset$. 
By \cite[Lemma 3.5]{Ma16Adv} in the almost finite case, 
or \cite[Lemma 4.3]{Ma16Adv} in the purely infinite case, 
every class in $H_0(\cG,\Z/2)$ is represented by $1_B$ 
for some clopen subset $B\subset s(U)$. 
Thus $\Ima\zeta$ is generated by the classes of transpositions 
$\tau_{UB}\in\sF(\cG)_W$, and hence $\Ima\zeta\subset\Lambda$. 
Exactness in Theorem~\ref{Li} now gives $\Lambda=H_1(\sF(\cG))$. 
\end{proof}

We now obtain two forms of local control: 
transpositions near a given arrow and extensions of bisections, 
both lying in the commutator subgroup. 

\begin{lemma}\label{elementofDG}
Let $\cG$ be a minimal ample groupoid with Cantor unit space 
which is either almost finite or purely infinite. 
\begin{enumerate}
\item If $g\in\cG$ satisfies $r(g)\neq s(g)$ and 
$A\in\cCO(\cG^{(0)})$ is a neighborhood of $r(g)$, then 
there exists a bisection $U\in\cCO(\cG)$ such that 
$g\in U$, $r(U)\cap s(U)=\emptyset$, $r(U)\subset A$ and 
$\tau_U\in\sD(\cG)$. 
\item If a bisection $U\in\cCO(\cG)$ satisfies 
$r(U)\neq\cG^{(0)}$ and $s(U)\neq\cG^{(0)}$, then 
there exists a full bisection $V$ such that 
$U\subset V$ and $\theta_V\in\sD(\cG)$. 
\end{enumerate}
\end{lemma}

\begin{proof}
(1)\:
Take a bisection $V\in\cCO(\cG)$ such that 
$g\in V$, $r(V)\cap s(V)=\emptyset$, $r(V)\subset A$. 
Apply \cite[Lemma 3.5 or Lemma 4.3]{Ma16Adv} to a nonempty clopen
subset of $r(V)\setminus\{r(g)\}$.
This gives a clopen subset $B\subset r(V)$ such that
$r(g)\notin B$ and $[1_B]=[1_{r(V)}]$ in $H_0(\cG,\Z/2)$. 
Then $U:=(r(V)\setminus B)V$ has the required properties, 
as $[1_{r(U)}]=[1_{r(V)}]-[1_B]=0$ in $H_0(\cG,\Z/2)$. 

(2)\:
Put $A:=\cG^{(0)}\setminus r(U)$ and $B:=\cG^{(0)}\setminus s(U)$.
Since $[1_{r(U)}]=[1_{s(U)}]$, we have $[1_A]=[1_B]$ in $H_0(\cG)$.
By Theorem~\ref{niceproperties} (3), 
there exists a bisection $W\in\cCO(\cG)$ 
such that $A=r(W)$ and $B=s(W)$. 
Then $U\sqcup W$ is a full bisection. 
By Lemma~\ref{localabelianization}, 
there exists $\beta\in\sF(\cG)_B$ such that
$\theta_{U\sqcup W}\beta\in\sD(\cG)$.
Writing $\beta=\theta_{s(U)\sqcup W_0}$
for a full bisection $W_0\in\cCO(\cG|B)$, we conclude that
\[
V:=U\sqcup WW_0\subset\cG
\]
is the required bisection. 
\end{proof}

Under the assumptions of Lemma~\ref{elementofDG}, 
every $A\in\cCO(\cG^{(0)})$ admits 
a finite clopen partition into sets of the form $r(U)$ 
with $U\in\cCO(\cG)$ a bisection satisfying $r(U)\cap s(U)=\emptyset$. 
Hence the classes $[1_{r(U)}]$ generate $H_0(\cG,\Z/2)$.

\subsection{Factor maps}\label{subsec:factormap}

In this subsection, 
we introduce factor maps between ample groupoids and 
describe their induced maps on homology and topological full groups. 

\begin{definition}[{\cite[Setting 4.10]{Ma22DCDS}}]\label{def:factor}
Let $\pi:\cG\to\cH$ be a continuous homomorphism 
between ample groupoids $\cG$ and $\cH$. 
We call $\pi$ a factor map if the following conditions hold. 
\begin{itemize}
\item $\pi$ is surjective. 
\item $\pi$ is proper, 
i.e.  the preimage of every compact subset of $\cH$ is compact in $\cG$. 
\item For any $x\in\cG^{(0)}$, $\pi$ induces a bijection 
from $r^{-1}(x)$ onto $r^{-1}(\pi(x))$. 
\end{itemize}
\end{definition}

We observe that $\pi^{-1}(\cH^{(0)})=\cG^{(0)}$. 
Indeed, suppose that $g\in\cG$ satisfies $\pi(g)\in\cH^{(0)}$. 
Then $\pi(g)=\pi(r(g))$, and both $g$ and $r(g)$ belong to $r^{-1}(r(g))$. 
The injectivity of $\pi:r^{-1}(r(g))\to r^{-1}(\pi(r(g)))$ 
therefore implies $g=r(g)$. 
It follows that the restriction 
\[
\pi^{(0)}:=\pi|{\cG^{(0)}}:\cG^{(0)}\longrightarrow\cH^{(0)}
\]
is a proper surjection. 

A factor map $\pi:\cG\to\cH$ is an isomorphism 
if and only if its restriction to the unit spaces 
$\cG^{(0)}\to\cH^{(0)}$ is a homeomorphism.

\begin{example}
Let $\Gamma$ be a countable discrete group. 
Let $\phi:\Gamma\curvearrowright X$ and $\psi:\Gamma\curvearrowright Y$ 
be minimal topologically free actions of $\Gamma$ 
on Cantor sets $X$ and $Y$. 
Suppose that $\pi:X\to Y$ is a $\Gamma$-equivariant continuous map, 
that is, 
$\pi\circ\phi_\gamma=\psi_\gamma\circ\pi$ holds for all $\gamma\in\Gamma$. 
Minimality of $\psi$ implies that $\pi$ is surjective.
It therefore gives rise to a factor map between the transformation groupoids 
$X\rtimes_\phi\Gamma$ and $Y\rtimes_\psi\Gamma$. 
\end{example}

Let $\pi:\cG\to\cH$ be a factor map. 
Properness and bijectivity on range and source fibers imply that 
$\pi^{-1}(V)$ is a compact open bisection 
whenever $V\in\cCO(\cH)$ is a bisection. 
The next lemma characterizes the sets and bisections that arise in this way: 
they are precisely those 
that are saturated with respect to the fibers of $\pi$. 

\begin{lemma}\label{saturatedness}
\begin{enumerate}
\item Let $A\in\cCO(\cG^{(0)})$. 
There exists a unique $B\in\cCO(\cH^{(0)})$ such that $A=\pi^{-1}(B)$ 
if and only if $\pi^{-1}(\pi(x))\subset A$ holds for every $x\in A$. 
\item Let $U\in\cCO(\cG)$ be a bisection. 
There exists a unique bisection $V\in\cCO(\cH)$ such that $U=\pi^{-1}(V)$ 
if and only if $\pi^{-1}(\pi(g))\subset U$ holds for every $g\in U$. 
\end{enumerate}
\end{lemma}

\begin{proof}
We have seen that $\pi^{(0)}=\pi|{\cG^{(0)}}$ is a proper surjection. 
Since proper maps between locally compact Hausdorff spaces are closed, 
both $\pi$ and $\pi^{(0)}$ are quotient maps. 

(1)\:
The `only if' part is clear. 
Conversely, suppose that 
$\pi^{-1}(\pi(x))\subset A$ holds for every $x\in A$. 
Then $A=(\pi^{(0)})^{-1}(\pi(A))$. 
Put $B:=\pi(A)$. 
Since $A$ is compact, $B$ is compact. 
Moreover, since $\pi^{(0)}$ is a quotient map and 
$(\pi^{(0)})^{-1}(B)=A$ is open, $B$ is open. 
Thus $B\in\cCO(\cH^{(0)})$ and $A=\pi^{-1}(B)$. 
The uniqueness of $B$ follows from the surjectivity of $\pi^{(0)}$. 

(2)\:
Again, the `only if' part is clear. 
Conversely, suppose that 
$\pi^{-1}(\pi(g))\subset U$ holds for every $g\in U$. 
Then $U=\pi^{-1}(\pi(U))$. 
Put $V:=\pi(U)$. 
Since $U$ is compact, $V$ is compact. 
Since $\pi$ is a quotient map and $\pi^{-1}(V)=U$ is open, 
$V$ is open. 

We show that $V$ is a bisection. 
Let $h_1,h_2\in V$ satisfy $r(h_1)=r(h_2)=y$. 
Choose $x\in\cG^{(0)}$ such that $\pi(x)=y$. 
By the defining property of a factor map, 
there exist unique $g_1,g_2\in r^{-1}(x)$ such that 
$\pi(g_1)=h_1$ and $\pi(g_2)=h_2$. 
Since $U=\pi^{-1}(V)$, we have $g_1,g_2\in U$. 
As $U$ is a bisection and $r(g_1)=r(g_2)=x$, 
we obtain $g_1=g_2$, and hence $h_1=h_2$. 
Thus $r|V$ is injective. 

By applying the range-fiber condition to inverses, 
$\pi$ also induces a bijection $s^{-1}(x)\to s^{-1}(\pi(x))$
for every $x\in\cG^{(0)}$. 
The same argument therefore shows that $s|V$ is injective. 
Hence $V$ is a compact open bisection and $U=\pi^{-1}(V)$. 
The uniqueness of $V$ follows from the surjectivity of $\pi$. 
\end{proof}

We next show that a factor map $\pi$ induces 
pullback homomorphisms $H_i^*(\pi)$ on groupoid homology. 
We use the notation of Definition~\ref{def:groupoidhomology}. 
A factor map $\pi:\cG\to\cH$ induces maps 
\[
\pi^{(n)}:\cG^{(n)}\to\cH^{(n)},\qquad 
\pi^{(n)}(g_1,\dots,g_n):=(\pi(g_1),\dots,\pi(g_n)), 
\]
and we write $\pi^{(0)}:=\pi|\cG^{(0)}$. 

\begin{lemma}\label{lem:pullback}
Let $\pi:\cG\to\cH$ be a factor map 
and let $A$ be a topological abelian group. 
For every $n\geq0$, the map $\pi^{(n)}$ is proper, and the maps 
\[
\pi^\#:C_c(\cH^{(n)},A)\to C_c(\cG^{(n)},A),\qquad 
\pi^\#(f):=f\circ\pi^{(n)}, 
\]
define a homomorphism of chain complexes. 
Consequently $\pi$ induces homomorphisms 
\[
H_i^*(\pi):H_i(\cH,A)\to H_i(\cG,A),\qquad i\geq0. 
\]
\end{lemma}

\begin{proof}
We first check that $\pi^{(n)}$ is proper. 
For $n=0$, this was already observed. 
Suppose that $n\geq1$ and let $K\subset\cH^{(n)}$ be compact and 
let $K_i\subset\cH$ be the image of $K$ 
under the $i$-th coordinate projection. 
Then 
\[
(\pi^{(n)})^{-1}(K)\subset
\pi^{-1}(K_1)\times\pi^{-1}(K_2)\times\dots\times\pi^{-1}(K_n), 
\]
and the right-hand side is compact, because $\pi$ is proper. 
Since $\cH^{(n)}$ is Hausdorff and $\cG^{(n)}$ is closed in $\cG^n$, 
the set $(\pi^{(n)})^{-1}(K)$ is closed in $\cG^n$, and hence is compact. 
It follows that, for every $f\in C_c(\cH^{(n)},A)$, 
the function $\pi^\#(f)$ is continuous and 
its support is contained in the compact set 
$(\pi^{(n)})^{-1}(\supp f)$. 
Thus $\pi^\#$ is a well-defined homomorphism. 

Since $\pi$ is a homomorphism, 
we have $d_i\circ\pi^{(n)}=\pi^{(n-1)}\circ d_i$ for every $i$. 
Hence, by the definition of $d_{i*}$, 
the identity $\pi^\#\circ d_{i*}=d_{i*}\circ\pi^\#$ holds 
as soon as $\pi^{(n)}$ restricts to a bijection 
from $d_i^{-1}(\gamma)$ onto $d_i^{-1}(\pi^{(n-1)}(\gamma))$ 
for every $\gamma\in\cG^{(n-1)}$. 
By the definition of a factor map, 
$\pi$ induces a bijection from $r^{-1}(x)$ onto $r^{-1}(\pi(x))$ 
for every $x\in\cG^{(0)}$, and, by applying this to inverses, 
also a bijection from $s^{-1}(x)$ onto $s^{-1}(\pi(x))$. 
When $n=1$, we have $d_0^{-1}(x)=s^{-1}(x)$ and $d_1^{-1}(x)=r^{-1}(x)$, 
and so the claim is precisely these two statements. 
Let $n\geq2$ and $\gamma=(g_1,\dots,g_{n-1})\in\cG^{(n-1)}$. 
The maps 
\[
\begin{aligned}
s^{-1}(r(g_1))\ni g&\longmapsto(g,g_1,\dots,g_{n-1})\in d_0^{-1}(\gamma), \\
r^{-1}(r(g_i))\ni g&\longmapsto
(g_1,\dots,g_{i-1},g,g^{-1}g_i,g_{i+1},\dots,g_{n-1})\in d_i^{-1}(\gamma), \\
r^{-1}(s(g_{n-1}))\ni g&\longmapsto(g_1,\dots,g_{n-1},g)\in d_n^{-1}(\gamma)
\end{aligned}
\]
are bijections, where $1\leq i\leq n{-}1$ in the second line. 
The same formulas describe the fibers of the face maps 
over $\pi^{(n-1)}(\gamma)$, and they are compatible with $\pi$. 
The desired bijectivity is now immediate. 

Therefore $\pi^\#\circ\partial_n=\partial_n\circ\pi^\#$, 
and $\pi^\#$ induces the homomorphisms $H_i^*(\pi)$. 
\end{proof}

\begin{remark}\label{rem:pullback-correction}
In the discussion immediately preceding Definition~3.4 of \cite{Ma12PLMS}, 
the author asserted that a proper homomorphism between \'etale groupoids 
induces a homomorphism on homology by pullback. 
This assertion is incorrect in general. 
Properness ensures that pullback is well-defined on compactly supported 
continuous functions in each degree, 
but does not by itself ensure commutation with the boundary operators. 
Indeed, if $\pi$ denotes the homomorphism from the group $\Z/2$ 
onto the trivial group, both regarded as discrete \'etale groupoids 
with a single unit, then $\pi$ is a proper surjection, 
but every face map of $\Z/2$ has two-point fibers, so that 
$\partial_2\circ\pi^\#=2\,\pi^\#\circ\partial_2\neq\pi^\#\circ\partial_2$ 
with $\Z$ coefficients. 
Lemma~\ref{lem:pullback} establishes the assertion for factor maps: 
the bijectivity on range fibers (and hence on source fibers) guarantees 
the required compatibility with the face-map pushforwards. 
\end{remark}

Let $\pi:\cG\to\cH$ be a factor map. 
Assume further that 
the unit spaces $\cG^{(0)}$ and $\cH^{(0)}$ are Cantor sets. 
Then $\pi$ induces a homomorphism $\pi^*:\sF(\cH)\to\sF(\cG)$ 
between the topological full groups defined 
in Definition~\ref{def:topologicalfullgroup} via 
\[
\pi^*(\theta_V):=\theta_{\pi^{-1}(V)}, 
\]
where $V\in\cCO(\cH)$ is a full bisection. 
Surjectivity of $\pi$ implies that $\pi^*$ is injective.

\subsection{Generalized alternating groups}

We conclude with Nekrashevych's multisections and 
generalized alternating groups \cite{Ne19ETDS}. 
These describe the finite permutation groups inside a full group 
and will be used to generate its commutator subgroup. 

Let $\cG$ be an ample groupoid with compact unit space. 
For $d\geq2$, 
a family $u=(U_{i,j})_{i,j=0}^{d-1}$ of nonempty compact open bisections 
is called a multisection of degree $d$ in $\cG$ 
if the following conditions hold: 
\begin{itemize}
\item The sets $U_{i,i}\subset\cG^{(0)}$ are mutually disjoint. 
\item $U_{i,j}U_{j,k}=U_{i,k}$ holds for all $0\leq i,j,k\leq d{-}1$. 
\end{itemize}
We call the union of the sets $U_{i,i}$ the support of $u$ 
and denote it by $\supp(u)$. 
Suppose that 
we are given bisections $V_1,V_2,\dots,V_{d-1}\in\cCO(\cG)$ such that 
\begin{itemize}
\item $r(V_1)=r(V_2)=\dots=r(V_{d-1})$, 
\item $r(V_1),s(V_1),s(V_2),\dots,s(V_{d-1})$ are mutually disjoint. 
\end{itemize}
Put $V_0:=r(V_1)$. 
Then the bisections 
\[
U_{i,j}:=V_i^{-1}V_j,\ i,j=0,1,2,\dots,d{-}1
\]
form a multisection of degree $d$, 
which we call the multisection generated by $V_1,V_2,\dots,V_{d-1}$. 

When $u=(U_{i,j})_{i,j=0}^{d-1}$ is a multisection of degree $d$ 
and $P\subset U_{0,0}$ is a nonempty clopen subset, 
the family 
\[
U_{i,0}PU_{0,j},\ i,j=0,1,2,\dots,d{-}1
\]
is again a multisection of degree $d$, 
and we write it as $u|P$. 

Let $\mathsf{S}_d$ denote the symmetric group of degree $d$ 
acting on $\{0,1,2,\dots,d{-}1\}$. 
For any $\pi\in\mathsf{S}_d$, 
\[
U_\pi:=U_{\pi(0),0}\sqcup U_{\pi(1),1}\sqcup\dots
\sqcup U_{\pi(d-1),d-1}
\sqcup\left(\cG^{(0)}\setminus\bigcup_{i=0}^{d-1}U_{i,i}\right)
\]
is a full bisection. 
The map $\pi\mapsto\theta_{U_{\pi}}$ is 
an embedding of $\mathsf{S}_d$ into $\sF(\cG)$. 
We denote by $\sA(u)$ the image of the alternating subgroup 
under this embedding.

\begin{definition}[{\cite[Definition 3.7]{Ne19ETDS}}]
Let $\cG$ be an ample groupoid with compact unit space. 
For $d\geq3$, denote by $\sA_d(\cG)$ the subgroup of $\sF(\cG)$ 
generated by the union of the subgroups $\sA(u)$ 
for all multisections $u$ of degree $d$. 
We call $\sA_d(\cG)$ the generalized alternating group of $\cG$. 
\end{definition}

Since the alternating group is the commutator subgroup of the symmetric group, 
$\sA_d(\cG)\subset\sD(\cG)$.
For the two classes considered here, these groups coincide. 

\begin{theorem}\label{A=D}
Let $\cG$ be a minimal ample groupoid with Cantor unit space 
which is either almost finite or purely infinite. 
For any $d\geq3$, the generalized alternating group $\sA_d(\cG)$ 
equals the commutator subgroup $\sD(\cG)$. 
\end{theorem}

\begin{proof}
Conjugation by a full bisection sends a multisection of degree $d$ 
to another such multisection. 
Hence $\sA_d(\cG)$ is normal in $\sF(\cG)$. 
Minimality and the infinitude of the unit space provide 
an orbit with $d$ distinct points, and thus a multisection of degree $d$. 
Therefore $\sA_d(\cG)$ is a nontrivial normal subgroup of $\sD(\cG)$. 
The latter is simple by Theorem~\ref{niceproperties} (1), 
so $\sA_d(\cG)=\sD(\cG)$. 
\end{proof}

\section{Maximal subgroups arising from factor groupoids}\label{sec:factor}

In this section, 
we investigate maximal subgroups of topological full groups 
arising from factor maps between ample groupoids.
Theorems~\ref{thm1:factor} and \ref{thm2:factor} reduce maximality 
to surjectivity and injectivity, respectively, of the map on abelianizations. 
The main step is Proposition~\ref{prop:factor}: 
an element outside the factor full group, 
together with the factor commutator subgroup, 
generates a group containing $\sD(\cG)$. 
We prove this by first using primeness 
to propagate a local alternating group, 
and then constructing such a group near a two-point fiber. 
Throughout this section, the unit spaces of $\cG$ and $\cH$ 
are assumed to be Cantor sets. 
Let $\pi:\cG\to\cH$ be a factor map (Definition~\ref{def:factor}). 
By Section~\ref{subsec:factormap}, 
$\pi$ induces an injective homomorphism $\pi^*:\sF(\cH)\to\sF(\cG)$. 
For a multisection $v=(V_{i,j})_{i,j=0}^{d-1}$ of degree $d$ in $\cH$, 
its pullback is the multisection in $\cG$ defined by 
\[
\pi^*(v):=(\pi^{-1}(V_{i,j}))_{i,j=0}^{d-1}. 
\]

\begin{lemma}\label{strictness}
Suppose that $\cG$ is minimal. 
If $\pi$ is not an isomorphism, 
then $\sD(\cG)$ is not contained in $\pi^*(\sF(\cH))$. 
\end{lemma}

\begin{proof}
Since $\pi$ is not an isomorphism, 
we may find two distinct points $x,z\in\cG^{(0)}$ 
such that $\pi(x)=\pi(z)$. 
The minimality of $\cG$ implies that 
there exists a multisection $u=(U_{i,j})_{i,j=0}^{2}$ 
of degree $3$ in $\cG$ 
such that $x\in U_{0,0}$ and $z\notin\supp(u)$. 
Let $\alpha\in\sA(u)$ be a non-trivial $3$-cycle, and 
let $U_\alpha\in\cCO(\cG)$ be the full bisection implementing $\alpha$. 
Since $z\notin\supp(u)$, the unit $z$ belongs to $U_\alpha$, 
whereas the unit $x$ does not belong to $U_\alpha$. 
As $\pi(x)=\pi(z)$, the bisection $U_\alpha$ is not saturated. 
Hence Lemma~\ref{saturatedness}~(2), 
together with the uniqueness of the full bisection 
implementing an element of $\sF(\cG)$, 
implies that $\alpha\notin\pi^*(\sF(\cH))$. 
Therefore $\sA(u)$ is a subgroup of $\sD(\cG)$ 
which is not contained in $\pi^*(\sF(\cH))$. 
\end{proof}

The first obstruction to maximality is 
the existence of nontrivial intermediate factor groupoids. 

\begin{definition}
Let $\pi:\cG\to\cH$ be a factor map which is not an isomorphism. 
We say that $\pi$ is prime 
if it admits no nontrivial intermediate ample groupoids. 
More precisely, 
if there exist an ample groupoid $\cK$ 
and factor maps $\pi_1:\cG\to\cK$, $\pi_2:\cK\to\cH$ 
such that $\pi=\pi_2\circ\pi_1$, 
then either $\pi_1$ or $\pi_2$ is an isomorphism. 
\end{definition}

Skew products by cyclic groups of prime order provide basic examples. 

\begin{lemma}\label{pto1}
Let $\cH$ be an ample groupoid whose unit space is a Cantor set. 
Let $p\in\N$ be a prime number and 
let $\xi:\cH\to\Z/p$ be a continuous homomorphism. 
Suppose that the skew product $\cG:=\cH\times_\xi\Z/p$ is minimal. 
Then the factor map $\pi:\cG\to\cH$ sending $(h,a)$ to $h$ is prime. 
\end{lemma}

\begin{proof}
Let $\lambda:\Z/p\curvearrowright\cG$ be the action by deck transformations, 
that is, 
$\lambda_b(h,a):=(h,a{+}b)$ for $(h,a)\in\cH\times\Z/p$ and $b\in\Z/p$. 
One has $\pi=\pi\circ\lambda_b$. 
Suppose that $\cK$ is an ample groupoid and 
$\pi_1:\cG\to\cK$, $\pi_2:\cK\to\cH$ are factor maps 
satisfying $\pi=\pi_2\circ\pi_1$. 
Assume that $\pi_1$ is not an isomorphism. 
There exists $x\in\cG^{(0)}$ and $c\neq0$ 
such that $\pi_1(x)=\pi_1(\lambda_c(x))$. 
For any $g\in r^{-1}(x)$, we have 
\[
r(\pi_1(g))=\pi_1(x)=\pi_1(\lambda_c(x))=r(\pi_1(\lambda_c(g)))
\]
and 
\[
\pi_2(\pi_1(g))=\pi(g)=\pi(\lambda_c(g))=\pi_2(\pi_1(\lambda_c(g))). 
\]
Since $\pi_2$ is a factor map, we obtain $\pi_1(g)=\pi_1(\lambda_c(g))$, 
which implies $\pi_1(s(g))=\pi_1(\lambda_c(s(g)))$. 
Hence $\pi_1$ equals $\pi_1\circ\lambda_c$ on $\cG^{(0)}$, 
because $\cG$ is minimal. 
As $p$ is a prime number, this means that 
$\pi_1$ equals $\pi_1\circ\lambda_b$ on $\cG^{(0)}$ for every $b\in\Z/p$. 
Therefore $\pi_2|\cK^{(0)}:\cK^{(0)}\to\cH^{(0)}$ is a homeomorphism, 
and so $\pi_2$ is an isomorphism. 
It follows that $\pi$ is prime. 
\end{proof}

Let $\pi:\cG\to\cH$ be a factor map and suppose that $\cG$ is minimal. 
By Lemma~\ref{strictness}, 
if $\pi^*(\sF(\cH))$ (resp. $\pi^*(\sD(\cH))$) is 
a maximal subgroup of $\sF(\cG)$ (resp. $\sD(\cG)$), 
then $\pi$ is prime. 
Indeed, any intermediate factor groupoid $\cK$ is minimal. 
If neither factor map is an isomorphism, 
applying Lemma~\ref{strictness} to both factor maps 
gives strict intermediate subgroups for both the full groups 
and their commutator subgroups. 
We will obtain converse statements 
under additional hypotheses on the fibers 
and the induced map on abelianizations.

\begin{setting}\label{factorsetting}
In the rest of this section, 
we let $\pi:\cG\to\cH$ be a factor map 
satisfying the following conditions. 
\begin{itemize}
\item The unit spaces $\cG^{(0)}$ and $\cH^{(0)}$ are Cantor sets. 
\item $\cG$ and $\cH$ are minimal and 
either almost finite or purely infinite. 
\item $\pi$ is not an isomorphism. 
\end{itemize}
\end{setting}

We first express intermediate factors in terms of invariant Boolean algebras. 
This will turn primeness 
into a statement about generating the topology of $\cG^{(0)}$. 

Let $\mathcal{B}\subset\cCO(\cG^{(0)})\cup\{\emptyset\}$ be 
a Boolean subalgebra containing $\pi^{-1}(\cCO(\cH^{(0)}))$. 
We say that $\mathcal{B}$ is $\cH$-invariant 
if $\pi^{-1}(V)P\pi^{-1}(V)^{-1}\in\mathcal{B}$ 
for any $P\in\mathcal{B}$ and any bisection $V\in\cCO(\cH)$. 

\begin{lemma}\label{Boolean}
Let $\mathcal{B}\subset\cCO(\cG^{(0)})\cup\{\emptyset\}$ be 
a Boolean subalgebra containing $\pi^{-1}(\cCO(\cH^{(0)}))$, 
which is $\cH$-invariant. 
Then the Stone spectrum of $\mathcal B$ gives 
an intermediate factor groupoid between $\cG$ and $\cH$. 
More precisely, 
there exist an ample groupoid $\cK$ 
and factor maps $\pi_1:\cG\to\cK$, $\pi_2:\cK\to\cH$ 
such that $\pi=\pi_2\circ\pi_1$ and 
$\mathcal{B}=\pi_1^{-1}(\cCO(\cK^{(0)}))\cup\{\emptyset\}$. 
\end{lemma}

\begin{proof}
Let $S(\mathcal{B})$ be the Stone spectrum of $\mathcal{B}$. 
There exist canonical continuous surjections 
$q:\cG^{(0)}\to S(\mathcal{B})$ and 
$p:S(\mathcal{B})\to\cH^{(0)}$ such that 
$\pi|\cG^{(0)}=p\circ q$. 
Notice also that 
$S(\mathcal{B})$ is compact, metrizable and totally disconnected, 
because $\mathcal{B}$ is countable. 

We first define 
an action of $\cH$ on $S(\mathcal{B})$ with anchor map $p$. 
Let $h\in\cH$ and $\omega\in S(\mathcal{B})$ satisfy $s(h)=p(\omega)$. 
Choose $x\in\cG^{(0)}$ such that $q(x)=\omega$. 
Applying the third condition in the definition of a factor map to inverses, 
there exists a unique $g\in\cG$ 
such that $s(g)=x$ and $\pi(g)=h$. 
We set $h\cdot\omega:=q(r(g))$. 
To see that this is well-defined, 
let $x'\in q^{-1}(\omega)$ and let $g'\in\cG$ be the unique element 
such that $s(g')=x'$ and $\pi(g')=h$. 
Choose a compact open bisection $V\subset\cH$ containing $h$. 
For every $P\in\mathcal{B}$, we have 
\[
r(g)\in P
\iff x\in\pi^{-1}(V^{-1})P\pi^{-1}(V)
\iff x'\in\pi^{-1}(V^{-1})P\pi^{-1}(V)
\iff r(g')\in P. 
\]
Indeed, $\pi^{-1}(V^{-1})P\pi^{-1}(V)$ belongs to $\mathcal{B}$ 
by $\cH$-invariance. 
Hence $q(r(g))=q(r(g'))$. 
The groupoid identities follow from the uniqueness of lifts. 
The same argument shows that every compact open bisection $V\subset\cH$ 
induces a homeomorphism from $p^{-1}(s(V))$ onto $p^{-1}(r(V))$. 
It follows that the action is continuous. 

Let 
\[
\cK:=\{(h,\omega)\in\cH\times S(\mathcal{B})\mid r(h)=p(\omega)\}. 
\]
We equip $\cK$ with the groupoid structure determined by 
\[
r(h,\omega)=\omega,\quad s(h,\omega)=h^{-1}\cdot\omega, 
\]
\[
(h,\omega)^{-1}=(h^{-1},h^{-1}\cdot\omega), 
\]
and 
\[
(h,\omega)(k,\eta)=(hk,\omega)
\quad\text{when }\eta=h^{-1}\cdot\omega. 
\]
Since $\cH^{(0)}$ is Hausdorff, its diagonal is closed. 
The map 
\[
\cH\times S(\mathcal{B})\longrightarrow
\cH^{(0)}\times\cH^{(0)},\qquad 
(h,\omega)\longmapsto(r(h),p(\omega))
\]
is continuous, 
and $\cK$ is the inverse image of the diagonal under this map. 
Hence $\cK$ is a closed subset of $\cH\times S(\mathcal{B})$. 
In particular, $\cK$ is locally compact, Hausdorff and second countable. 
If $V\subset\cH$ is a compact open bisection and 
$C\subset p^{-1}(r(V))$ is clopen, then 
\[
[V,C]:=\{(h,\omega)\in\cK\mid h\in V,\ \omega\in C\}
\]
is a compact open bisection of $\cK$. 
Such sets form a basis for the topology of $\cK$. 
Therefore $\cK$ is an ample groupoid 
with unit space naturally identified with $S(\mathcal{B})$, 
with effectiveness still to be verified. 

For $x\in\cG^{(0)}$, the $\cK$-orbit of $q(x)$ contains 
the image under $q$ of the $\cG$-orbit of $x$. 
It follows from the minimality of $\cG$ that $\cK$ is minimal. 
We claim that $p(C)$ has nonempty interior 
for every nonempty clopen subset $C\subset S(\mathcal{B})$. 
Indeed, by minimality and compactness, 
there exist compact open bisections $U_1,\dots,U_n$ of $\cK$ 
such that 
\[
s(U_i)\subset C
\quad\text{and}\quad
S(\mathcal{B})=\bigcup_{i=1}^n r(U_i). 
\]
Refining them if necessary, 
we may assume that $U_i=[V_i,C_i]$ as above. 
If $p(C)$ had empty interior, then, being compact, 
it would be nowhere dense. 
For each $i$, the set $p(r(U_i))$ is the image of 
$p(s(U_i))\subset p(C)$ under the partial homeomorphism 
associated with $V_i$, and hence is nowhere dense. 
This contradicts 
\[
\cH^{(0)}=\bigcup_{i=1}^n p(r(U_i)). 
\]
Suppose now that $\cK$ were not effective. 
Using the basis above, we could find a compact open bisection $[V,C]$ 
such that $C$ is nonempty and
$[V,C]\subset\Iso(\cK)\setminus\cK^{(0)}$. 
For every $z\in p(C)$, the unique element $h\in V$ satisfying 
$r(h)=z$ is a nonunit isotropy element of $\cH$. 
Since $p(C)$ has nonempty interior, 
$V\cap r^{-1}(p(C)^\circ)$ is a nonempty open subset 
of $\Iso(\cH)\setminus\cH^{(0)}$, 
contrary to the effectiveness of $\cH$. 
Thus $\cK$ is effective. 

Define 
\[
\pi_1:\cG\to\cK,
\qquad 
\pi_1(g):=(\pi(g),q(r(g))), 
\]
and 
\[
\pi_2:\cK\to\cH,
\qquad 
\pi_2(h,\omega):=h. 
\]
The definition of the action shows that 
these maps are continuous homomorphisms and that $\pi=\pi_2\circ\pi_1$. 
The map $\pi_2$ is clearly surjective. 
For $(h,\omega)\in\cK$, choose $x\in q^{-1}(\omega)$ 
and lift $h$ uniquely to an element of $r^{-1}(x)$. 
This shows that $\pi_1$ is also surjective. 
The map $\pi_2$ is proper, 
because for every compact subset $L\subset\cH$, 
$\pi_2^{-1}(L)=\cK\cap(L\times S(\mathcal{B}))$ is compact. 
If $L\subset\cK$ is compact, then $\pi_1^{-1}(L)$ is a closed subset of 
the compact set $\pi^{-1}(\pi_2(L))$. 
Hence $\pi_1$ is proper as well. 

For every $x\in\cG^{(0)}$, the restriction of $\pi_1$ gives a bijection 
\[
r_{\cG}^{-1}(x)\longrightarrow r_{\cK}^{-1}(q(x)), 
\]
because $\pi:r_{\cG}^{-1}(x)\to r_{\cH}^{-1}(\pi(x))$ is bijective. 
Likewise, for every $\omega\in S(\mathcal{B})$, 
the restriction of $\pi_2$ gives a bijection 
\[
r_{\cK}^{-1}(\omega)\longrightarrow r_{\cH}^{-1}(p(\omega)). 
\]
Thus both $\pi_1$ and $\pi_2$ are factor maps. 

Finally, Stone duality gives 
\[
\mathcal{B}
=q^{-1}(\cCO(S(\mathcal{B})))\cup\{\emptyset\}
=\pi_1^{-1}(\cCO(\cK^{(0)}))\cup\{\emptyset\}. 
\]
This completes the proof. 
\end{proof}

\begin{lemma}\label{subbase}
Suppose that $P_0,P'_0\in\cCO(\cG^{(0)})$ satisfy 
$P_0\neq\pi^{-1}(\pi(P_0))$, $P_0\cap P'_0=\emptyset$ 
and $P_0\cup P'_0=\pi^{-1}(\pi(P_0\cup P'_0))$. 
If $\pi:\cG\to\cH$ is prime, then
\[
\mathcal{P}
:=\left\{\pi^{-1}(V)P\pi^{-1}(V)^{-1}
\mid P\in\{P_0,P'_0\},\ \text{$V\in\cCO(\cH)$ is a bisection}\right\}
\setminus\{\emptyset\}
\]
is a subbase of $\cG^{(0)}$. 
\end{lemma}

\begin{proof}
For a bisection $V\in\cCO(\cH)$ and 
$P\in\cCO(\cG^{(0)})\cup\{\emptyset\}$, write 
\[
V\cdot P:=\pi^{-1}(V)P\pi^{-1}(V)^{-1}. 
\]
Put $Q_0:=\pi(P_0\cup P'_0)$. 
By the assumption and Lemma~\ref{saturatedness} (1), 
$Q_0\in\cCO(\cH^{(0)})$.

Let $\mathcal{B}$ be the Boolean subalgebra of 
$\cCO(\cG^{(0)})\cup\{\emptyset\}$ generated by $\mathcal{P}$. 
By Lemma~\ref{saturatedness}~(1), 
in order to show that $\mathcal{B}$ contains every saturated clopen subset, 
it suffices to prove that $\pi^{-1}(Q)\in\mathcal{B}$ 
for every $Q\in\cCO(\cH^{(0)})$. 
Let $y\in\cH^{(0)}$ and let $O$ be an open neighborhood of $y$. 
Since $\cH$ is minimal and $Q_0$ is nonempty, 
there exists $h\in\cH$ such that $r(h)=y$ and $s(h)\in Q_0$. 
We can therefore find a compact open bisection $V\subset\cH$ 
containing $h$ such that $r(V)\subset O$ and $s(V)\subset Q_0$. 
For such a bisection, one has 
\[
\pi^{-1}(r(V))=V\cdot(P_0\cup P'_0)=(V\cdot P_0)\cup(V\cdot P'_0). 
\]
It follows that the sets $r(V)$ arising in this way 
form a basis for the topology of $\cH^{(0)}$, 
and that the inverse image of each of them belongs to $\mathcal{B}$ 
and is a finite union of elements of $\mathcal{P}$. 
By compactness, for every
$Q\in\cCO(\cH^{(0)})$, the set $\pi^{-1}(Q)$ is a finite union
of elements of $\mathcal{P}$. 
In particular, $\pi^{-1}(\cCO(\cH^{(0)}))\subset\mathcal{B}$. 

We next verify that $\mathcal{B}$ is $\cH$-invariant.
For compact open bisections $V,W\subset\cH$ and $P\in\{P_0,P'_0\}$, 
one has $W\cdot(V\cdot P)=(WV)\cdot P$, 
where the right-hand side may be empty. 
Moreover, the partial homeomorphism associated with $W$ induces 
a Boolean isomorphism between the clopen subsets of $\pi^{-1}(s(W))$ 
and those of $\pi^{-1}(r(W))$. 
Since both of these saturated clopen subsets belong to $\mathcal{B}$, 
it follows that $W\cdot B\in\mathcal{B}$ 
for every $B\in\mathcal{B}$. 
Thus, $\mathcal{B}$ is $\cH$-invariant. 

By Lemma~\ref{Boolean}, there exist an ample groupoid $\cK$ and
factor maps $\pi_1:\cG\to\cK$ and $\pi_2:\cK\to\cH$ 
such that $\pi=\pi_2\circ\pi_1$ and 
$\mathcal{B}=\pi_1^{-1}(\cCO(\cK^{(0)}))\cup\{\emptyset\}$. 
Since $\pi$ is prime, either $\pi_1$ or $\pi_2$ is an isomorphism.
As $P_0\in\mathcal{P}\subset\mathcal{B}$ 
and $P_0\neq\pi^{-1}(\pi(P_0))$,  
the map $\pi_2$ cannot be an isomorphism. 
Indeed, if $\pi_2$ were an isomorphism, 
then $\pi_1=\pi_2^{-1}\circ\pi$, 
and hence $\mathcal{B}=\pi^{-1}(\cCO(\cH^{(0)}))\cup\{\emptyset\}$, 
which contradicts the choice of $P_0$. 
Hence $\pi_1$ is an isomorphism, 
and therefore 
$\mathcal{B}=\cCO(\cG^{(0)})\cup\{\emptyset\}$. 

It remains to remove complements from Boolean expressions 
in the generators. 
Let $P\in\mathcal{P}$, and suppose first that
$P=V\cdot P_0$ for a compact open bisection $V\subset\cH$. 
Put $V_0:=V\cap s^{-1}(Q_0)$. 
Then $P=V_0\cdot P_0$ and 
\[
\cG^{(0)}\setminus P=(V_0\cdot P'_0)
\cup\pi^{-1}(\cH^{(0)}\setminus r(V_0)).
\]
The second term on the right-hand side is a finite union 
of elements of $\mathcal{P}$ by the first part of the proof. 
Hence $\cG^{(0)}\setminus P$ is a finite union of elements 
of $\mathcal{P}$. 
The case $P=V\cdot P'_0$ is symmetric. 

Thus every element of the Boolean algebra generated by $\mathcal{P}$ 
is a finite union of finite intersections of elements of $\mathcal{P}$. 
Since this Boolean algebra is $\cCO(\cG^{(0)})\cup\{\emptyset\}$, 
the finite intersections of elements of $\mathcal{P}$ 
form a basis for the topology of $\cG^{(0)}$. 
Therefore $\mathcal{P}$ is a subbase. 
\end{proof}

We next track the clopen sets 
on which a subgroup contains a local alternating group. 
Degree five allows us to use the perfectness 
of the alternating group when intersecting such sets. 

Consider pairs $(P,v)$ consisting of $P\in\cCO(\cG^{(0)})$ and 
a multisection $v=(V_{i,j})_{i,j=0}^4$ of degree $5$ in $\cH$ 
satisfying $P\subset\pi^{-1}(V_{0,0})$ and $\supp(v)\neq\cH^{(0)}$. 
We let $L$ denote the set of all such pairs. 
When $(P,v)$ is in $L$, 
$u:=\pi^*(v)|P$ is a multisection in $\cG$. 

\begin{lemma}\label{propertyofP}
Let $H\subset\sF(\cG)$ be a subgroup containing $\pi^*(\sD(\cH))$. 
Set 
\[
\mathcal{P}:=\{P\in\cCO(\cG^{(0)})\mid
\sA(\pi^*(v)|P)\subset H\text{ for some }(P,v)\in L\}. 
\]
\begin{enumerate}
\item If $v=(V_{i,j})_{i,j=0}^4$ is a multisection in $\cH$ 
with $\supp(v)\neq\cH^{(0)}$, then 
$\pi^{-1}(Q)$ is in $\mathcal{P}$ 
for any nonempty clopen subset $Q\subset V_{0,0}$. 
\item If $(P,v)\in L$ and $\sA(\pi^*(v)|P)\subset H$, then 
$\sA(\pi^*(v')|P)\subset H$ for any $(P,v')\in L$. 
\item For any $P,P'\in\mathcal{P}$, 
we have $P\cap P'\in\mathcal{P}$ whenever $P\cap P'$ is nonempty. 
\item For any $P\in\mathcal{P}$ and 
a bisection $V\in\cCO(\cH)$ with $r(V)\neq\cH^{(0)}$, 
the set $\pi^{-1}(V)P\pi^{-1}(V^{-1})$ belongs to $\mathcal{P}$ 
whenever it is nonempty. 
\end{enumerate}
\end{lemma}

\begin{proof}
(1)\:
The group in question is $\pi^*(\sA(v|Q))$,
which is contained in $\pi^*(\sD(\cH))\subset H$. 

(2)\:
Assume $\sA(\pi^*(v)|P)\subset H$, 
where $v=(V_{i,j})_{i,j=0}^4$ is a multisection of degree $5$ in $\cH$ 
and $\pi(P)\subset V_{0,0}$. 
Let $v'=(V'_{i,j})_{i,j=0}^4$ be another multisection of degree $5$ 
such that $\pi(P)\subset V'_{0,0}$. 
Set $Q:=V_{0,0}\cap V'_{0,0}$. 
By replacing $V_{i,j}$ with $V_{i,0}QV_{0,j}$, 
we may assume $V_{0,0}=Q$. 
In the same way, we may also assume $V'_{0,0}=Q$. 
Then 
\[
W:=Q\sqcup V'_{1,0}V_{0,1}\sqcup V'_{2,0}V_{0,2}
\sqcup V'_{3,0}V_{0,3}\sqcup V'_{4,0}V_{0,4}
\]
is a compact open bisection. 
Notice that $r(W)\subset\supp(v')\neq\cH^{(0)}$ and 
$s(W)\subset\supp(v)\neq\cH^{(0)}$. 
It follows from Lemma~\ref{elementofDG} (2) that 
there exists a full bisection $V\in\cCO(\cH)$ 
such that $W\subset V$ and $\theta_V\in\sD(\cH)$. 
Conjugation gives
\[
\pi^*(\theta_V)\sA(\pi^*(v)|P)\pi^*(\theta_V)^{-1}
=\sA(\pi^*(v')|P), 
\]
which implies $\sA(\pi^*(v')|P)\subset H$. 

(3)\:
Suppose that $P\cap P'\neq\emptyset$ and choose $(P,v),(P',v')\in L$
with $\sA(\pi^*(v)|P),\sA(\pi^*(v')|P')\subset H$.
Write $v=(V_{i,j})_{i,j=0}^4$ and $v'=(V'_{i,j})_{i,j=0}^4$, and put
\[
Q:=V_{0,0}\cap V'_{0,0},\qquad
P_1:=P\cap\pi^{-1}(Q),\qquad P_2:=P'\cap\pi^{-1}(Q).
\]
By (1), $\pi^{-1}(Q)\in\mathcal{P}$ and
$\sA(\pi^*(v)|\pi^{-1}(Q))\subset H$.
Applying \cite[Lemma 3.4]{Ne19ETDS} to this group and
$\sA(\pi^*(v)|P)$ gives $\sA(\pi^*(v)|P_1)\subset H$.
Similarly, $\sA(\pi^*(v')|P_2)\subset H$, and (2) then gives
$\sA(\pi^*(v)|P_2)\subset H$.
A second application of \cite[Lemma 3.4]{Ne19ETDS}, using
$P_1\cap P_2=P\cap P'$, yields
$\sA(\pi^*(v)|(P\cap P'))\subset H$.
Thus $P\cap P'\in\mathcal{P}$.

(4)\:
Put $P':=\pi^{-1}(V)P\pi^{-1}(V^{-1})$ and suppose that $P'\neq\emptyset$.
Choose $(P,v)\in L$ with $\sA(\pi^*(v)|P)\subset H$,
and write $v=(V_{i,j})_{i,j=0}^4$.
Set $Q:=s(V)\cap V_{0,0}$.
By (1) and (3), the nonempty set $P\cap\pi^{-1}(Q)$ belongs to $\mathcal{P}$.
Replacing $V$, $P$ and $v$ by $VQ$, $P\cap\pi^{-1}(Q)$ and $v|Q$,
respectively, leaves $P'$ unchanged.
By (2), we may therefore assume $s(V)=V_{0,0}$ and
$\sA(\pi^*(v)|P)\subset H$. 
Since $r(V)\neq\cH^{(0)}$, by Lemma~\ref{elementofDG}, 
we can find a full bisection $W\in\cCO(\cH)$ 
such that $V\subset W$ and $\theta_W\in\sD(\cH)$. 
Then $v':=(WV_{i,j}W^{-1})_{i,j}$ is a multisection of degree $5$ 
with $\supp(v')\neq\cH^{(0)}$, 
and
\[
\pi^*(\theta_W)\sA(\pi^*(v)|P)\pi^*(\theta_W)^{-1}
=\sA(\pi^*(v')|P').
\] 
Hence $\sA(\pi^*(v')|P')$ is contained in $H$. 
Thus $P'$ belongs to $\mathcal{P}$. 
\end{proof}

\begin{lemma}\label{IfHcontainsA0}
Suppose that $\pi:\cG\to\cH$ is prime. 
Let $H\subset\sF(\cG)$ be a subgroup containing $\pi^*(\sD(\cH))$. 
If there exists $(P_0,v)\in L$ such that $P_0\neq \pi^{-1}(\pi(P_0))$ 
and $\sA(\pi^*(v)|P_0)\subset H$, then 
\[
\mathcal{P}:=\{P\in\cCO(\cG^{(0)})\mid
\sA(\pi^*(v)|P)\subset H\text{ for some }(P,v)\in L\}
\]
is a base of $\cG^{(0)}$. 
\end{lemma}

\begin{proof}
We have $P_0\in\mathcal{P}$. 
Let $v=(V_{i,j})_{i,j=0}^4$ be a multisection 
such that $\sA(\pi^*(v)|P_0)\subset H$. 
Put $P'_0:=\pi^{-1}(V_{0,0})\setminus P_0$. 
Then $(P'_0,v)\in L$ and 
\[
\sA(\pi^*(v)|P'_0)
\subset\langle \sA(\pi^*(v)|P_0),\sA(\pi^*(v))\rangle\subset H, 
\]
which implies $P'_0\in\mathcal{P}$. 
To apply Lemma~\ref{subbase}, we note that 
the restriction $r(V)\neq\cH^{(0)}$  in Lemma~\ref{propertyofP} (4) 
causes no problem. 
Indeed, let $P\in\{P_0,P'_0\}$, let $V\in\cCO(\cH)$ be a bisection, 
and take 
\[
x\in\pi^{-1}(V)P\pi^{-1}(V^{-1}). 
\]
Since $\cH^{(0)}$ is a Cantor set, 
there exists $R\in\cCO(\cH^{(0)})$ such that 
$\pi(x)\in R\subset r(V)$ and $R\neq\cH^{(0)}$. 
Put $W:=RV$. 
Then $r(W)=R\neq\cH^{(0)}$ and 
\[
x\in\pi^{-1}(W)P\pi^{-1}(W^{-1})
=\pi^{-1}(R)\cap\pi^{-1}(V)P\pi^{-1}(V^{-1}). 
\]
Since $P\in\mathcal{P}$, Lemma~\ref{propertyofP} (4) implies that 
$\pi^{-1}(W)P\pi^{-1}(W^{-1})$ belongs to $\mathcal{P}$. 
Thus $\mathcal{P}$ locally refines 
the subbase given by Lemma~\ref{subbase}. 
Consequently, the finite intersections of elements of $\mathcal{P}$ 
form a basis for the topology of $\cG^{(0)}$, 
so $\mathcal{P}$ is a subbase. 
Then, by Lemma~\ref{propertyofP} (3), 
$\mathcal{P}$ is a base of $\cG^{(0)}$. 
\end{proof}

To pass from these local groups to arbitrary multisections, 
we will use the following elementary fact about alternating groups. 

\begin{lemma}\label{alternating9}
Let $\sA_9$ be the alternating group acting on $X=\{0,1,2,\dots,8\}$. 
For a subset $E\subset X$, 
let $\sA_9(E)$ denote the alternating group on $E$, 
regarded as a subgroup of $\sA_9$ 
by letting it act trivially on $X\setminus E$. 
Then
\[
\left\langle \sA_9(\{k,5,6,7,8\})\mid k=0,1,\dots,4\right\rangle=\sA_9. 
\]
\end{lemma}

\begin{proof}
Each $3$-cycle $(5\ 6\ i)$, where $i\in X\setminus\{5,6\}$, 
belongs to one of the indicated subgroups, and these cycles generate $\sA_9$. 
\end{proof}

\begin{lemma}\label{IfHcontainsA}
Suppose that $\pi:\cG\to\cH$ is prime. 
Let $H\subset\sF(\cG)$ be a subgroup containing $\pi^*(\sD(\cH))$. 
If there exists $(P_0,v)\in L$ such that $P_0\neq \pi^{-1}(\pi(P_0))$ 
and $\sA(\pi^*(v)|P_0)\subset H$, then one has $\sD(\cG)\subset H$. 
\end{lemma}

\begin{proof}
Let $\mathcal{P}$ be the base of $\cG^{(0)}$ 
obtained in Lemma~\ref{IfHcontainsA0}. 
Let $u=(U_{i,j})_{i,j=0}^4$ be a multisection of degree $5$ in $\cG$. 
It suffices to show that $\sA(u)\subset H$. 
Take $x\in U_{0,0}$. 
We claim that there exists $P_x\in\mathcal{P}$ such that 
$x\in P_x\subset U_{0,0}$ and $\sA(u|P_x)\subset H$. 
Once this is done, by the compactness of $U_{0,0}$, 
we may find finitely many $x_1,x_2,\dots,x_n\in U_{0,0}$ 
such that 
\[
U_{0,0}=\bigcup_{k=1}^nP_{x_k}. 
\]
It follows from \cite[Proposition 3.3]{Ne19ETDS} that 
\[
\sA(u)\subset\left\langle\sA(u|P_{x_k})\mid 1\leq k\leq n\right\rangle
\subset H
\]
as desired. 

For each $i=1,2,3,4$, 
let $g_i\in U_{0,i}$ be the unique arrow with $r(g_i)=x$. 
Put $g_0:=x$. 
First suppose that 
$\pi(x),\pi(s(g_1)),\dots,\pi(s(g_4))$ are all distinct. 
Then, one can find a multisection $v=(V_{i,j})_{i,j=0}^4$ in $\cH$ 
satisfying $\supp(v)\neq\cH^{(0)}$ and 
$\pi(g_i)\in V_{0,i}$ for every $i=1,2,3,4$. 
Since $\pi$ is continuous and $\mathcal{P}$ is a base, 
there exists $P_x\in\mathcal{P}$ 
such that $x\in P_x\subset U_{0,0}$ and 
$P_xU_{0,i}\subset\pi^{-1}(V_{0,i})$, 
which implies $P_xU_{0,i}=P_x\pi^{-1}(V_{0,i})$. 
Therefore we obtain $u|P_x=\pi^*(v)|P_x$, and so 
$\sA(u|P_x)=\sA(\pi^*(v)|P_x)\subset H$ by Lemma~\ref{propertyofP} (2). 

In the general case, we introduce four auxiliary components
whose images avoid those of the original five components. 
Since $\cH$ is minimal, we can find $f_1,\dots,f_4\in\cG$ 
such that the following hold. 
\begin{itemize}
\item $r(f_i)=x$ for all $i=1,2,3,4$. 
\item $\#\{\pi(s(g_k)),\pi(s(f_1)),\dots,\pi(s(f_4))\}=5$ 
for all $k=0,1,2,3,4$. 
\end{itemize}
Choose bisections $W_i\in\cCO(\cG)$, $i=1,2,3,4$, so that the following hold. 
\begin{itemize}
\item $f_i\in W_i$ for all $i=1,2,3,4$. 
\item $A:=r(W_1)=r(W_2)=r(W_3)=r(W_4)\subset U_{0,0}$. 
\item $s(W_1),s(W_2),s(W_3),s(W_4)$ and $\supp(u|A)$ are 
mutually disjoint. 
\end{itemize}
Fix $k\in\{0,1,\dots,4\}$. 
Let $z_k$ be the multisection 
generated by $U_{k,0}W_i$, $i=1,2,3,4$. 
One can find a multisection $v^k=(V^k_{i,j})_{i,j=0}^4$ in $\cH$ 
satisfying $\supp(v^k)\neq\cH^{(0)}$ and 
$\pi(g_k^{-1}f_i)\in V^k_{0,i}$ for every $i=1,2,3,4$. 
Since $\pi$ is continuous, there exists $P_k\in\cCO(\cG^{(0)})$ 
such that $x\in P_k\subset A$ and 
$U_{k,0}P_kW_i\subset\pi^{-1}(V^k_{0,i})$. 
As $\mathcal{P}$ is a base, 
we can take $P_x\in\mathcal{P}$ so that $x\in P_x$ and 
$P_x\subset P_k$ for all $k=0,1,\dots,4$. 
By two applications of Lemma~\ref{propertyofP}~(4), we have
\[
U_{k,0}P_xU_{0,k}
=\pi^{-1}(V^k_{0,1})\left(\pi^{-1}(V^0_{1,0})P_x
\pi^{-1}(V^0_{1,0})^{-1}\right)\pi^{-1}(V^k_{0,1})^{-1}
\in\mathcal{P}. 
\]
Hence Lemma~\ref{propertyofP}~(2) gives
\[
\sA(z_k|U_{k,0}P_xU_{0,k})
=\sA(\pi^*(v^k)|U_{k,0}P_xU_{0,k})
\subset H. 
\]
Let $w$ be the multisection of degree $9$ generated by
\[
P_xU_{0,k}\quad (k=1,2,3,4)
\quad\text{and}\quad
P_xW_i\quad (i=1,2,3,4).
\]
Therefore $u|P_x$ is the sub-multisection of $w$ 
on its first five components, 
while, for every $k=0,1,\dots,4$, 
$z_k|U_{k,0}P_xU_{0,k}$ is the sub-multisection 
on the $k$-th component and the last four components. 
Then Lemma~\ref{alternating9} implies 
\[
\sA(u|P_x)\subset
\left\langle\sA(z_k|U_{k,0}P_xU_{0,k})\mid k=0,1,\dots,4\right\rangle
\subset H. 
\]
Consequently $\sA(u)$ is contained in $H$. 
It follows from Theorem~\ref{A=D} that $H$ contains $\sD(\cG)$. 
\end{proof}

It remains to construct 
a local alternating group on a nonsaturated clopen set. 
The next lemma does this near a two-point fiber with trivial isotropy. 
The key calculation uses a commutator with a lifted transposition. 

\begin{lemma}\label{technical}
Suppose that 
a full bisection $U\subset\cG$ and a point $y\in\cH^{(0)}$ 
satisfy the following. 
\begin{itemize}
\item The isotropy group of $y$ is trivial and 
$\#\pi^{-1}(y)=2$. 
\item Letting $(r|U)^{-1}(\pi^{-1}(y))=\{g,f\}$, 
one has $\pi(g)\neq\pi(f)$. 
\end{itemize}
Let $H\subset\sF(\cG)$ be the subgroup generated 
by $\pi^*(\sD(\cH))$ and $\theta_U$. 
Then, 
there exists $(P_0,v)\in L$ such that $P_0\neq \pi^{-1}(\pi(P_0))$ 
and $\sA(\pi^*(v)|P_0)\subset H$. 
\end{lemma}

\begin{proof}
Let $g,f\in U$ be as in the statement. 
Choose a bisection $V\in\cCO(\cH)$ so that 
$r(V)\neq\cH^{(0)}$, $s(V)\neq\cH^{(0)}$ and $\pi(f)^{-1}\in V$. 
By Lemma~\ref{elementofDG}, 
there exists a full bisection $W\subset\cH$ 
such that $V\subset W$ and $\theta_W\in\sD(\cH)$. 
Put $\bar U:=\pi^{-1}(W)U$ and $\bar y:=\theta_W(y)$. 
Then the arrow of $\bar U$ arising from $f$ is the unit $s(f)$. 
The point $\bar y$ is in the orbit of $y$, so its isotropy group is trivial. 
Moreover, the lifting property of $\pi$ gives $\#\pi^{-1}(\bar y)=2$, 
and the two arrows over this fiber still have distinct images under $\pi$. 
Replacing $U$ by $\bar U$ does not change the subgroup $H$, 
because 
\[
\theta_{\bar U}=\pi^*(\theta_W)\theta_U
\quad\text{and}\quad
\pi^*(\theta_W)\in\pi^*(\sD(\cH)). 
\]
Thus, after replacing $U$ and $y$ with $\bar U$ and $\bar y$ 
and relabelling, we may assume that $f=s(f)$ is in $\cG^{(0)}$. 
The new $g$ is not in $\cG^{(0)}$, and $r(g)\neq s(g)$, 
because the isotropy group of $y$ is trivial. 

Put $x:=r(g)$ and $z:=f=r(f)=s(f)$. 
Thus $\pi^{-1}(y)=\{x,z\}$, $\theta_U(z)=z$ and 
$\theta_U^{-1}(x)=s(g)$. 
Notice that $\pi(s(g))\neq y$. 
Since the isotropy group of $y$ is trivial, 
$\pi$ restricts to a bijection 
from each of the $\cG$-orbits of $x$ and $z$ onto the $\cH$-orbit of $y$. 
For $y_0'$ in the orbit of $y$, write 
\[
\pi^{-1}(y_0')=\{x_0',z_0'\},
\]
where $x_0'$ is in the orbit of $x$ and $z_0'$ is in the orbit of $z$. 
The map $y_0'\mapsto \pi\left(\theta_U^{-1}(z_0')\right)$ is 
a bijection of the orbit of $y$. 
Since this orbit is infinite, we can choose $y_0'$ so that 
\[
y_0'\notin\{y,\pi(s(g))\}
\quad\text{and}\quad
\pi\left(\theta_U^{-1}(z_0')\right)
\notin\{y,\pi(s(g))\}. 
\]
Let $g_0',f_0'\in U$ be the arrows satisfying 
$r(g_0')=x_0'$ and $r(f_0')=z_0'$. 
The three arrows 
\[
y,\quad \pi(s(g)),\quad \pi(f_0')^{-1}
\]
have mutually distinct sources and mutually distinct ranges. 
We can therefore find a bisection $V\in\cCO(\cH)$ satisfying 
$r(V)\neq\cH^{(0)}$, $s(V)\neq\cH^{(0)}$ and 
\[
y,\pi(s(g)),\pi(f_0')^{-1}\in V.
\]
By Lemma~\ref{elementofDG}, 
there exists a full bisection $W\subset\cH$ 
such that $V\subset W$ and $\theta_W\in\sD(\cH)$. 
Replacing $U$ by $\pi^{-1}(W)U$ again does not change $H$. 
Since $W$ contains the units $y$ and $\pi(s(g))$, 
the arrows $g$ and $f$ remain unchanged. 
On the other hand, the arrow arising from $f_0'$ becomes the unit $s(f_0')$. 
Put 
\[
y':=\theta_W(y_0')=\pi(s(f_0')),\quad
x':=\pi^*(\theta_W)(x_0'),\quad
z':=s(f_0'),
\]
and let $g'\in U$ be the arrow arising from $g_0'$. 
Then, after relabelling, our setup is as follows: 
\begin{itemize}
\item $\pi^{-1}(y)=\{x,z\}$ and $\pi^{-1}(y')=\{x',z'\}$. 
\item $x$ and $x'$ are in the same orbit, and 
$z$ and $z'$ are in the same orbit. 
\item $\theta_U(z)=z$ and $\theta_U(z')=z'$. 
\item $g,g'\in U$, $r(g)=x$, $r(g')=x'$ and $s(g)\neq x$. 
\item $y$, $\pi(s(g))$ and $y'$ are mutually distinct. 
\end{itemize}
The isotropy group of $y'$ is also trivial. 
Hence there exists a unique $h\in\cH$ such that $r(h)=y$ and $s(h)=y'$. 
Figure~\ref{fig:case3} illustrates the configuration in Case 3 below. 

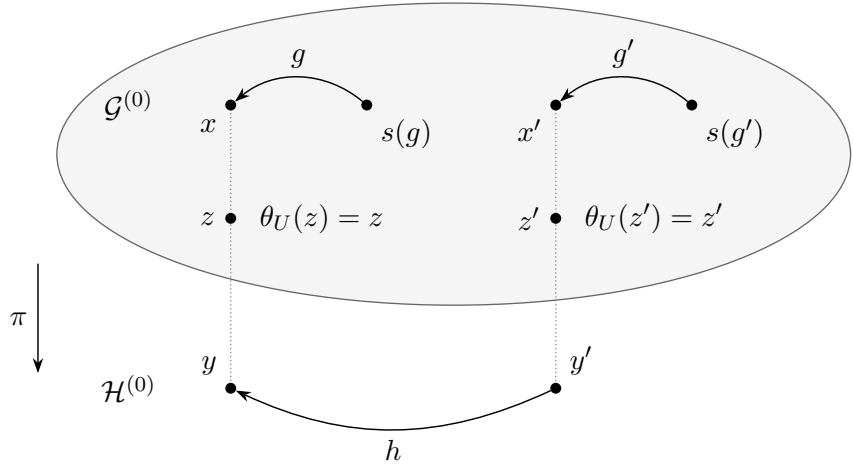
\begin{figure}[ht]
\begin{tikzpicture}[
  x=1cm,y=1cm,
  point/.style={circle,fill=black,inner sep=1.45pt,outer sep=0pt},
  arrow/.style={-{Stealth[length=2mm,width=1.4mm]},line width=.55pt},
  fiber/.style={densely dotted,black!45,line width=.5pt},
  every node/.style={font=\normalsize}
]
  \draw[fill=black!4,draw=black!60,line width=.5pt]
    (5.95,3.10) ellipse [x radius=5.25,y radius=2.00];
  \node at (1.65,3.75) {$\cG^{(0)}$};
  \node at (1.65,0) {$\cH^{(0)}$};

  \draw[fiber] (3,3.75) -- (3,0);
  \draw[fiber] (7.3,3.75) -- (7.3,0);
  \draw[arrow] (.45,1.65) -- node[left] {$\pi$} (.45,.2);

  \node[point,label=below left:{$x$}] (x) at (3,3.75) {};
  \node[point,label=below right:{$s(g)$}] (sg) at (4.8,3.75) {};
  \node[point,label=below left:{$x'$}] (xp) at (7.3,3.75) {};
  \node[point,label=below right:{$s(g')$}] (sgp) at (9.1,3.75) {};
  \draw[arrow] (sg) to[bend right=40] node[above] {$g$} (x);
  \draw[arrow] (sgp) to[bend right=40] node[above] {$g'$} (xp);

  \node[point,label=left:{$z$}] (z) at (3,2.25) {};
  \node[point,label=left:{$z'$}] (zp) at (7.3,2.25) {};
  \node[anchor=west] at (3.25,2.25) {$\theta_U(z)=z$};
  \node[anchor=west] at (7.55,2.25) {$\theta_U(z')=z'$};

  \node[point,label=above left:{$y$}] (y) at (3,0) {};
  \node[point,label=above right:{$y'$}] (yp) at (7.3,0) {};
  \draw[arrow] (yp) to[bend left=25] node[below] {$h$} (y);
\end{tikzpicture}
\caption{The configuration in Case 3 of Lemma~\ref{technical}. 
The dotted lines indicate the fibers over $y$ and $y'$. 
In Cases 1 and 2, $g'=x'$ and $s(g')=x$, respectively. }
\label{fig:case3}
\end{figure}

We will repeatedly use the following consequence of properness: 
if an open set contains a fiber $\pi^{-1}(t)$, 
then it contains $\pi^{-1}(Q)$ for some clopen neighborhood $Q$ of $t$. 
This allows conditions 
on the two lifts to be imposed on saturated neighborhoods. 
We distinguish three cases according to $\theta_U^{-1}(x')$. 
In each case, 
a commutator produces a nontrivial even permutation 
on a restricted multisection of degree five. 
Simplicity of the alternating group then gives the required inclusion. 

\medskip
\noindent
\textbf{Case 1.}
Assume $\theta_U(x')=x'$, or equivalently $g'=x'$. 
By taking sufficiently small bisections around $\pi(g)$ and $h$, 
using Lemma~\ref{elementofDG}~(1) for the bisection containing $h$, 
and then choosing two further bisections, 
we can find $V_1,V_2,V_3,V_4\in\cCO(\cH)$ such that 
\begin{itemize}
\item $\pi(g)\in V_1$, $h\in V_2$ and $\tau_{V_2}\in\sD(\cH)$.
\item $r(V_1)=r(V_2)=r(V_3)=r(V_4)$.
\item $s(V_1)$, $s(V_2)$, $s(V_3)$, $s(V_4)$ and $r(V_1)$ 
are mutually disjoint.
\item $r(V_1)\cup s(V_1)\cup s(V_2)\cup s(V_3)\cup s(V_4)$ 
is a proper subset of $\cH^{(0)}$.
\item $\pi^{-1}(r(V_2))U\subset\cG^{(0)}\sqcup\pi^{-1}(V_1)$.
\item $\pi^{-1}(s(V_2))U\subset\cG^{(0)}$.
\end{itemize}
Let $v=(V_{i,j})_{i,j=0}^4$ be the multisection 
generated by $V_1,V_2,V_3,V_4$. 
Set 
\[
R:=\pi^{-1}(r(V_2)),\quad S:=\pi^{-1}(s(V_2)),
\quad F_R:=U\cap R, 
\]
and 
\[
P_0:=R\setminus F_R
=\left((\pi^{-1}(r(V_2))U)\setminus\cG^{(0)}\right)U^{-1}. 
\]
Then $P_0$ is a nonempty clopen subset of $R$. 
Also put 
\[
P_1:=\pi^{-1}(V_1^{-1})P_0\pi^{-1}(V_1),\quad 
P_2:=\pi^{-1}(V_2^{-1})P_0\pi^{-1}(V_2).
\]
These are the first three diagonal components of $\pi^*(v)|P_0$. 
The last two inclusions above give 
\[
\theta_U^{-1}(R)=F_R\sqcup P_1,\quad 
\theta_U^{-1}(S)=S.
\]
In particular, $\theta_U(P_1)=P_0$, $\theta_U$ fixes $S$ pointwise, 
and $\theta_U(P_0)\cap(R\cup S)=\emptyset$. 
A direct computation now shows that 
\[
\gamma:=\theta_U^{-1}\pi^*(\tau_{V_2})
\theta_U\pi^*(\tau_{V_2})
\]
satisfies 
\[
\gamma(P_0)=P_1,\quad
\gamma(P_1)=P_2,\quad
\gamma(P_2)=P_0, 
\]
and fixes the complement of $P_0\sqcup P_1\sqcup P_2$. 
Thus $\gamma$ is the element corresponding to the $3$-cycle 
$(0\ 1\ 2)$ in $\sA(\pi^*(v)|P_0)$. 
In particular, 
$\gamma$ has order three and cyclically permutes $x,s(g),x'$. 
Since $\tau_{V_2}\in\sD(\cH)$, we have $\gamma\in H$. 

The set $P_0$ contains $x$ and does not contain $z$. 
Hence $P_0\neq\pi^{-1}(\pi(P_0))$. 
Moreover, $(P_0,v)\in L$ by our choice of the bisections. 
The group $\sA(\pi^*(v))$ normalizes 
$\sA(\pi^*(v)|P_0)$, and its conjugation action on this group 
is the usual conjugation action of the alternating group of degree five 
on itself. 
Hence the subgroup generated by the 
$\sA(\pi^*(v))$-conjugates of $\gamma$ is a non-trivial normal subgroup 
of the simple group $\sA(\pi^*(v)|P_0)$. 
As 
\[
\sA(\pi^*(v))\subset\pi^*(\sD(\cH))\subset H,
\]
we conclude that $\sA(\pi^*(v)|P_0)\subset H$. 

\medskip
\noindent
\textbf{Case 2.}
Assume $\theta_U^{-1}(x')=x$, or equivalently $\pi(g')=h^{-1}$. 
At the fibers over $y$ and $y'$, respectively, 
the points $x,x'$ are not fixed by $\theta_U$, whereas $z,z'$ are fixed. 
Since $U\cap\cG^{(0)}$ is clopen, 
we can choose $V_1,V_2,V_3,V_4\in\cCO(\cH)$ so that 
\begin{itemize}
\item $\pi(g)\in V_1$, $h\in V_2$ and $\tau_{V_2}\in\sD(\cH)$. 
\item $r(V_1)=r(V_2)=r(V_3)=r(V_4)$. 
\item $s(V_1)$, $s(V_2)$, $s(V_3)$, $s(V_4)$ and $r(V_1)$ 
are mutually disjoint. 
\item $r(V_1)\cup s(V_1)\cup s(V_2)\cup s(V_3)\cup s(V_4)$ 
is a proper subset of $\cH^{(0)}$. 
\item $\pi^{-1}(r(V_2))U\subset\cG^{(0)}\sqcup\pi^{-1}(V_1)$. 
\item $\pi^{-1}(s(V_2))U\subset\cG^{(0)}\sqcup\pi^{-1}(V_2^{-1})$. 
\item $\pi^*(\tau_{V_2})\left(U\cap\pi^{-1}(r(V_2))\right)
=U\cap\pi^{-1}(s(V_2))$. 
\end{itemize}
The last equality is a local condition around $h$. 
Indeed, $\pi^*(\tau_{V_2})$ sends $x,z$ to $x',z'$, respectively, 
and the unit part $U\cap\cG^{(0)}$ is clopen. 
Thus the equality holds after restricting an initial bisection around $h$, 
and the construction in Lemma~\ref{elementofDG}~(1) can be carried out 
inside this restriction. 
Let $v=(V_{i,j})_{i,j=0}^4$ be the multisection 
generated by $V_1,V_2,V_3,V_4$. 
Set 
\[
R:=\pi^{-1}(r(V_2)),\quad S:=\pi^{-1}(s(V_2)),
\quad F_R:=U\cap R,\quad F_S:=U\cap S, 
\]
and 
\[
P_0:=R\setminus F_R
=\left((\pi^{-1}(r(V_2))U)\setminus\cG^{(0)}\right)U^{-1}. 
\]
Then $P_0$ is a nonempty clopen subset of $R$. 
Put 
\[
P_1:=\pi^{-1}(V_1^{-1})P_0\pi^{-1}(V_1),\quad
P_2:=\pi^{-1}(V_2^{-1})P_0\pi^{-1}(V_2). 
\]
The matching condition gives 
\[
P_2=\pi^*(\tau_{V_2})(P_0)=S\setminus F_S. 
\]
The two preceding inclusions concerning $U$ now give 
\[
\theta_U^{-1}(R)=F_R\sqcup P_1,\quad
\theta_U^{-1}(S)=F_S\sqcup P_0. 
\]
Thus $\theta_U(P_1)=P_0$, $\theta_U(P_0)=P_2$, 
and $\theta_U(P_2)\cap(R\cup S)=\emptyset$. 
A direct computation shows that 
\[
\gamma:=\theta_U^{-1}\pi^*(\tau_{V_2})\theta_U\pi^*(\tau_{V_2})
\]
satisfies 
\[
\gamma(P_0)=P_2,\quad 
\gamma(P_2)=P_1,\quad 
\gamma(P_1)=P_0,
\]
and fixes the complement of $P_0\sqcup P_1\sqcup P_2$. 
Hence $\gamma$ is the element corresponding to the $3$-cycle 
$(0\ 2\ 1)$ in $\sA(\pi^*(v)|P_0)$. 
In particular, 
it has order three and cyclically permutes $x,x',s(g)$. 
Since $\tau_{V_2}\in\sD(\cH)$, we have $\gamma\in H$. 

Again, $P_0$ contains $x$ and does not contain $z$, 
so $P_0\neq\pi^{-1}(\pi(P_0))$, and $(P_0,v)\in L$. 
The same simplicity and normalization argument as in Case~1 yields 
$\sA(\pi^*(v)|P_0)\subset H$. 

\medskip
\noindent
\textbf{Case 3.}
Assume $\theta_U^{-1}(x')\notin\{x,x'\}$. 
We first observe that 
\[
y,\quad \pi(s(g)),\quad y',\quad \pi(s(g'))
\]
are mutually distinct. 
Indeed, 
if $\pi(s(g'))=y'$, then $\pi(g')$ is an isotropy element at $y'$. 
It is therefore the unit at $y'$, 
and the lifting property of $\pi$ gives $g'=x'$, a contradiction. 
If $\pi(s(g'))=y$, then 
the triviality of isotropy gives $\pi(g')=h^{-1}$, 
and the lifting property gives $s(g')=x$, again a contradiction. 
Finally, suppose that $\pi(s(g'))=\pi(s(g))$. 
Let $\tilde h\in\cG$ be the unique lift of $h$ with $r(\tilde h)=x$. 
Then $s(\tilde h)=x'$, and 
\[
\pi(g')=h^{-1}\pi(g)=\pi(\tilde h^{-1}g).
\]
Since both $g'$ and $\tilde h^{-1}g$ have range $x'$, 
the lifting property implies $g'=\tilde h^{-1}g$. 
Thus $s(g')=s(g)$, contradicting the injectivity of $s|U$. 

At both fibers over $y$ and $y'$, 
the points $x,x'$ are not fixed by $\theta_U$, 
whereas $z,z'$ are fixed. 
We can therefore choose $V_1,V_2,V_3,V_4\in\cCO(\cH)$ so that 
\begin{itemize}
\item $\pi(g)\in V_1$, $h\in V_2$, $\pi(g')\in V_3$ and 
$\tau_{V_2}\in\sD(\cH)$.
\item $r(V_1)=r(V_2)=r(V_4)$ and $r(V_3)=s(V_2)$.
\item $s(V_1)$, $s(V_2)$, $s(V_3)$, $s(V_4)$ and $r(V_1)$ 
are mutually disjoint.
\item $r(V_1)\cup s(V_1)\cup s(V_2)\cup s(V_3)\cup s(V_4)$ 
is a proper subset of $\cH^{(0)}$.
\item $\pi^{-1}(r(V_2))U\subset\cG^{(0)}\sqcup\pi^{-1}(V_1)$.
\item $\pi^{-1}(s(V_2))U\subset\cG^{(0)}\sqcup\pi^{-1}(V_3)$.
\item $\pi^*(\tau_{V_2})
\left(U\cap\pi^{-1}(r(V_2))\right)
=U\cap\pi^{-1}(s(V_2))$.
\end{itemize}
As in Case~2, the last equality holds locally around $h$.
We apply Lemma~\ref{elementofDG} (1) inside a restriction
on which it holds. 
Let $v=(V_{i,j})_{i,j=0}^4$ be the multisection 
generated by $V_1,V_2,V_2V_3,V_4$. 
Set 
\[
\begin{aligned}
R&:=\pi^{-1}(r(V_2)), & S&:=\pi^{-1}(s(V_2)),\\
F_R&:=U\cap R, & F_S&:=U\cap S,
\end{aligned}
\]
and 
\[
P_0:=R\setminus F_R
=\left((\pi^{-1}(r(V_2))U)\setminus\cG^{(0)}\right)U^{-1}. 
\]
Then $P_0$ is a nonempty clopen subset of $R$. 
Put 
\[
P_1:=\pi^{-1}(V_1^{-1})P_0\pi^{-1}(V_1),\quad
P_2:=\pi^{-1}(V_2^{-1})P_0\pi^{-1}(V_2), 
\]
and 
\[
P_3:=\pi^{-1}(V_3^{-1})P_2\pi^{-1}(V_3)
=\pi^{-1}(V_3^{-1}V_2^{-1})P_0\pi^{-1}(V_2V_3).
\]
These are the first four diagonal components of $\pi^*(v)|P_0$. 
The matching condition gives 
\[
P_2=\pi^*(\tau_{V_2})(P_0)=S\setminus F_S.
\]
The two inclusions concerning $U$ give 
\[
\theta_U^{-1}(R)=F_R\sqcup P_1,\quad 
\theta_U^{-1}(S)=F_S\sqcup P_3. 
\]
Thus $\theta_U(P_1)=P_0$, $\theta_U(P_3)=P_2$, and 
\[
\theta_U(P_0)\cap(R\cup S)=\emptyset,\quad 
\theta_U(P_2)\cap(R\cup S)=\emptyset. 
\]
A direct computation shows that 
\[
\gamma:=\theta_U^{-1}\pi^*(\tau_{V_2})\theta_U\pi^*(\tau_{V_2})
\]
satisfies 
\[
\gamma(P_0)=P_2,\quad
\gamma(P_2)=P_0,\quad
\gamma(P_1)=P_3,\quad
\gamma(P_3)=P_1, 
\]
and fixes the complement of $P_0\sqcup P_1\sqcup P_2\sqcup P_3$. 
Hence $\gamma$ is the element corresponding to 
$(0\ 2)(1\ 3)$ in $\sA(\pi^*(v)|P_0)$. 
In particular, 
it has order two and interchanges $x$ with $x'$ and $s(g)$ with $s(g')$. 
Since $\tau_{V_2}\in\sD(\cH)$, we have $\gamma\in H$. 

Once again, $P_0$ contains $x$ and does not contain $z$, 
so $P_0\neq\pi^{-1}(\pi(P_0))$, and $(P_0,v)\in L$. 
The same simplicity and normalization argument as in Case~1 
gives $\sA(\pi^*(v)|P_0)\subset H$. 
This completes the proof.
\end{proof}

We now apply the local construction to two classes of factor maps. 
We say that $\pi$ is a covering map of degree $d$ 
if its restriction $\pi^{(0)}:\cG^{(0)}\to\cH^{(0)}$ is 
a covering map of degree $d$. 
Equivalently, $\pi:\cG\to\cH$ itself is a covering map of degree $d$. 

We say that $\pi$ is at most two-to-one 
if $\#(\pi^{(0)})^{-1}(y)\leq2$ for any $y\in\cH^{(0)}$. 

\begin{proposition}\label{prop:factor}
Let $\pi:\cG\to\cH$ be as in Setting~\ref{factorsetting}. 
Suppose that either of the following holds. 
\begin{enumerate}
\item $\pi:\cG\to\cH$ is a covering map of degree two. 
\item $\pi:\cG\to\cH$ is prime, at most two-to-one and $\cH$ is principal. 
\end{enumerate}
Then, for any $\gamma\in\sF(\cG)\setminus\pi^*(\sF(\cH))$, 
the group generated by $\gamma$ and $\pi^*(\sD(\cH))$ contains $\sD(\cG)$. 
\end{proposition}

\begin{proof}
Take a full bisection $U\in\cCO(\cG)$ 
such that $\gamma=\theta_U$. 
Let $H$ be the subgroup generated by $\theta_U$ and $\pi^*(\sD(\cH))$. 
We must show that $\sD(\cG)\subset H$. 
As $\gamma$ is not in $\pi^*(\sF(\cH))$, 
Lemma~\ref{saturatedness}~(2) implies that $\pi^{-1}(\pi(U))\neq U$. 
Therefore $P:=r(\pi^{-1}(\pi(U))\setminus U)\subset\cG^{(0)}$ is 
a nonempty compact subset. 
Notice that, for any $y\in\pi(P)$, 
there exist $g,f\in U$ such that 
$y=\pi(r(g))=\pi(r(f))$ and $\pi(g)\neq\pi(f)$. 

(1)\:
Choose a clopen trivialization of the two-sheeted covering $\pi^{(0)}$. 
Lifting arrows of $\cH$ gives permutations of the two sheets, 
and hence a continuous homomorphism $\xi:\cH\to\Z/2$. 
This identifies $\cG$ with $\cH\times_\xi\Z/2$ over $\cH$. 
In particular, $\pi$ is open. 
Since $U$ is compact open, $\pi(U)$ is compact open, and hence 
$\pi^{-1}(\pi(U))\setminus U$ is open. 
Since both $r$ and $\pi|{\cG^{(0)}}$ are open, 
$P$ and $\pi(P)$ are nonempty open subsets of 
$\cG^{(0)}$ and $\cH^{(0)}$, respectively. 
It follows from Lemma~\ref{pto1} that $\pi$ is prime. 
Since $\cH$ is effective, 
there exists $y\in\pi(P)$ whose isotropy group is trivial. 
The fiber $\pi^{-1}(y)$ has two points. 
It follows from Lemma~\ref{technical} that 
there exists $(P_0,v)\in L$ such that $P_0\neq \pi^{-1}(\pi(P_0))$ 
and $\sA(\pi^*(v)|P_0)\subset H$. 
Hence, by Lemma~\ref{IfHcontainsA}, we get $\sD(\cG)\subset H$. 

(2)\:
Take any $y\in\pi(P)$. 
The isotropy group of $y$ is trivial, because $\cH$ is principal. 
By the observation above, there exist $g,f\in U$ such that 
\[
y=\pi(r(g))=\pi(r(f))\quad\text{and}\quad\pi(g)\neq\pi(f).
\]
Since $U$ is a bisection, $r(g)\neq r(f)$. 
Thus $\pi^{-1}(y)$ contains at least two points, 
and hence $\#\pi^{-1}(y)=2$ by the at most two-to-one assumption. 
Therefore, Lemma~\ref{technical} and Lemma~\ref{IfHcontainsA} apply 
in the same way as above, and we get the conclusion. 
\end{proof}

The remaining obstruction to maximality lies in the abelianization. 
Combining Proposition~\ref{prop:factor} with the commutative diagram 
\[
\xymatrix@M=8pt{
1 \ar[r] & \sD(\cG) \ar[r] & \sF(\cG) \ar[r] & 
H_1(\sF(\cG)) \ar[r] & 1 \\
1 \ar[r] & \sD(\cH) \ar[r] \ar[u]_-{\pi^*} & 
\sF(\cH) \ar[r] \ar[u]_-{\pi^*} & 
H_1(\sF(\cH)) \ar[r] \ar[u]_-{H_1(\pi^*)} & 1, \\
}
\]
we obtain the following two theorems. 

\begin{theorem}\label{thm1:factor}
Let $\pi:\cG\to\cH$ be as in Setting~\ref{factorsetting}. 
Suppose that 
either of the two conditions of Proposition~\ref{prop:factor} holds. 
Then $\pi^*(\sF(\cH))$ is a maximal subgroup of $\sF(\cG)$ 
if and only if 
the induced map $H_1(\pi^*):H_1(\sF(\cH))\to H_1(\sF(\cG))$ 
is surjective. 
\end{theorem}

\begin{proof}
Suppose that 
$H_1(\pi^*):H_1(\sF(\cH))\to H_1(\sF(\cG))$ is surjective. 
Take $\gamma\in\sF(\cG)\setminus\pi^*(\sF(\cH))$ 
and let $H$ be the group generated by $\gamma$ and $\pi^*(\sF(\cH))$. 
It follows from Proposition~\ref{prop:factor} that 
$H$ contains $\sD(\cG)$. 
This, together with the surjectivity of $H_1(\pi^*)$, 
implies that $H$ equals $\sF(\cG)$. 
Therefore, 
$\pi^*(\sF(\cH))$ is a maximal subgroup of $\sF(\cG)$. 

Suppose that $H_1(\pi^*)$ is not surjective. 
By Lemma~\ref{strictness}, 
we can take $\gamma\in\sD(\cG)\setminus\pi^*(\sF(\cH))$. 
Then the group generated by $\gamma$ and $\pi^*(\sF(\cH))$ 
is a proper subgroup of $\sF(\cG)$. 
\end{proof}

\begin{theorem}\label{thm2:factor}
Let $\pi:\cG\to\cH$ be as in Setting~\ref{factorsetting}. 
Suppose that 
either of the two conditions of Proposition~\ref{prop:factor} holds. 
Then $\pi^*(\sD(\cH))$ is a maximal subgroup of $\sD(\cG)$ 
if and only if 
the induced map $H_1(\pi^*):H_1(\sF(\cH))\to H_1(\sF(\cG))$ 
is injective. 
\end{theorem}

\begin{proof}
Suppose that 
$H_1(\pi^*):H_1(\sF(\cH))\to H_1(\sF(\cG))$ is injective. 
Take $\gamma\in\sD(\cG)\setminus\pi^*(\sD(\cH))$ 
and let $H$ be the group generated by $\gamma$ and $\pi^*(\sD(\cH))$. 
If there exists $\gamma_0\in\sF(\cH)$ such that $\pi^*(\gamma_0)=\gamma$, 
then the image of $\gamma_0$ in $H_1(\sF(\cH))$ is 
in the kernel of $H_1(\pi^*)$. 
Hence $\gamma_0$ is in $\sD(\cH)$, which is a contradiction. 
Thus $\gamma$ is in $\sD(\cG)\setminus\pi^*(\sF(\cH))$. 
Proposition~\ref{prop:factor} now gives $H=\sD(\cG)$. 

When $H_1(\pi^*)$ is not injective, 
there exists $\gamma\in\sF(\cH)\setminus\sD(\cH)$ such that 
$\pi^*(\gamma)\in\sD(\cG)$. 
Then the group generated by $\pi^*(\gamma)$ and $\pi^*(\sD(\cH))$ 
is contained in $\pi^*(\sF(\cH))$. 
In particular, Lemma~\ref{strictness} tells us that 
it is not equal to $\sD(\cG)$. 
\end{proof}

The same generation argument also identifies 
the normalizer and the unique maximal subgroup 
containing the factor commutator subgroup. 

\begin{corollary}\label{cor:factornormalizer}
Under the hypotheses of Theorem~\ref{thm2:factor}, one has 
\[
\Nor(\sF(\cG),\pi^*(\sD(\cH)))=\pi^*(\sF(\cH)). 
\]
Moreover, 
\[
M:=\Nor(\sD(\cG),\pi^*(\sD(\cH)))=\pi^*(\sF(\cH))\cap\sD(\cG)
\]
is the unique maximal subgroup of $\sD(\cG)$ containing $\pi^*(\sD(\cH))$, 
and $M/\pi^*(\sD(\cH))\cong\Ker H_1(\pi^*)$. 
\end{corollary}

\begin{proof}
Put $P:=\pi^*(\sF(\cH))$ and $A:=\pi^*(\sD(\cH))$. 
Since $\sD(\cH)$ is characteristic in $\sF(\cH)$, $P$ normalizes $A$. 
If $\gamma\in\Nor(\sF(\cG),A)\setminus P$, then 
Proposition~\ref{prop:factor} gives 
$\sD(\cG)\subset\langle A,\gamma\rangle$. 
Since $A$ is normal in $\langle A,\gamma\rangle$, 
it is a non-trivial normal subgroup of $\sD(\cG)$. 
The simplicity of $\sD(\cG)$ (Theorem~\ref{niceproperties}~(1)) 
would imply $A=\sD(\cG)$, contrary to Lemma~\ref{strictness}. 
Thus $\Nor(\sF(\cG),A)=P$, and intersecting with $\sD(\cG)$ 
gives the formula for $M$. 

By Lemma~\ref{strictness}, $M$ is a proper subgroup of $\sD(\cG)$. 
For every $\gamma\in\sD(\cG)\setminus M$, 
Proposition~\ref{prop:factor} gives $\langle A,\gamma\rangle=\sD(\cG)$. 
Hence every proper subgroup of $\sD(\cG)$ containing $A$ 
is contained in $M$, proving the asserted uniqueness and maximality. 
Finally, under the identification $P/A\cong H_1(\sF(\cH))$ 
induced by $\pi^*$, the subgroup $M/A$ corresponds to $\Ker H_1(\pi^*)$. 
\end{proof}

\begin{remark}\label{rem:Li-naturality}
Li's exact sequence (Theorem~\ref{Li}) helps us to determine 
when the map $H_1(\pi^*):H_1(\sF(\cH))\to H_1(\sF(\cG))$ 
is surjective or injective. 
More precisely, the following diagram is commutative: 
\[
\xymatrix@M=8pt{
H_0(\cG,\Z/2) \ar[r]^-{\zeta_\cG} & 
H_1(\sF(\cG)) \ar[r]^-{\eta_\cG} & H_1(\cG) \ar[r] & 0 \\
H_0(\cH,\Z/2) \ar[r]^-{\zeta_\cH} \ar[u]_-{H_0^*(\pi)} & 
H_1(\sF(\cH)) \ar[r]^-{\eta_\cH} \ar[u]_-{H_1(\pi^*)} & 
H_1(\cH) \ar[r] \ar[u]_-{H_1^*(\pi)} & 0, }
\]
where $H_i^*(\pi)$ are the homomorphisms of Lemma~\ref{lem:pullback}.
The right square commutes by the definition of the index map.
The left square commutes by the description of $\zeta$
in terms of transpositions, since
$\pi^*(\tau_U)=\tau_{\pi^{-1}(U)}$ for every compact open bisection
$U\subset\cH$ with $r(U)\cap s(U)=\emptyset$.

For example, if $\zeta_\cG$ is zero and $H_1^*(\pi)$ is surjective, 
then $H_1(\pi^*)$ is surjective. 
If $\zeta_\cH$ is zero and $H_1^*(\pi)$ is injective, 
then $H_1(\pi^*)$ is injective. 
\end{remark}

\begin{remark}
The author does not know 
if Proposition~\ref{prop:factor} is valid 
for skew product extensions by $\Z/p$ with $p\geq3$ prime. 
Lemma~\ref{pto1} ensures primeness, 
but the construction in Lemma~\ref{technical} uses 
the two-point fiber assumption. 
\end{remark}

\section{Examples I}\label{sec:examplesI}

We apply the factor-map criteria of Section~\ref{sec:factor} 
to minimal $\Z$ actions and SFT groupoids. 
For almost one-to-one extensions, 
the torsion-free quotient of the dimension groups controls maximality. 
For two-to-one coverings, we compute the pullback maps explicitly, 
first for minimal $\Z$ actions 
and then for SFT groupoids in terms of adjacency matrices.

\subsection{Almost one-to-one factors of minimal $\Z$ actions}

The realization theorem of Giordano, Putnam and Skau 
provides almost one-to-one extensions with prescribed dimension groups. 
We show that their construction also yields a prime factor map, 
so that Theorems~\ref{thm1:factor} and \ref{thm2:factor} apply. 

\begin{theorem}[{\cite[Theorem 4.1, Theorem 4.8]{GPS01MathScand}}]
\label{GPS}
Let $\psi:\Z\curvearrowright Y$ be 
a minimal $\Z$ action on a Cantor set $Y$ 
and suppose 
\[
\xymatrix@M=8pt{
0 \ar[r] & 
H_0(Y\rtimes_\psi\Z) \ar[r]^-i & G \ar[r]^-{q} & Q \ar[r] & 0 
}
\]
is a short exact sequence of abelian groups, 
with $Q\neq0$ countable and torsion-free. 
Also suppose that $G$ is a simple dimension group, 
that $i$ is an order embedding and 
that $i(H_0(Y\rtimes_\psi\Z))$ is order dense in $G$. 
Then there exists a minimal $\Z$ action $\phi:\Z\curvearrowright X$ 
on a Cantor set $X$, 
an almost one-to-one equivariant map $\pi:X\to Y$ 
and an order isomorphism $\alpha:H_0(X\rtimes_\phi\Z)\to G$ 
such that $\alpha\circ H_0^*(\pi)=i$ and 
the induced factor map 
$\pi:X\rtimes_\phi\Z\to Y\rtimes_\psi\Z$ 
is prime and at most two-to-one. 
\end{theorem}

\begin{proof}
Only primeness remains to be proved. 
We use the notation of \cite[Section 4]{GPS01MathScand}: 
$Z$ is the path space of the simple subdiagram $D_{\mathcal{Q}}$, 
and $X$ is the spectrum of the completion of $\mathcal{B}$. 
We show that identifying one two-point fiber forces 
every two-point fiber to be identified. 
Let $a$ be a nonzero multiplicative linear functional 
on the completion of $\mathcal{B}$ and 
let $y=(e_1,e_2,\dots)\in Z$ be such that 
\[
a(f)=f(y)\quad\forall f\in\mathcal{A}. 
\]
Put $p_n:=(e_1,e_2,\dots,e_n)$. 
Lemma 4.5 (iii) of \cite{GPS01MathScand} implies 
\[
a\left(1_{G(p_n)}\right)
=a\left(1_{C(p_n)}\right)+\sum_e a\left(1_{G(p_ne)}\right), 
\]
where the sum is over $e$ in $E(D_{\mathcal{Q}})$ 
with $r(p_n)=s(e)$. 
By the construction of $C(p_n)$ (see \cite[Lemma 4.4]{GPS01MathScand}), 
we have 
\[
a\left(1_{C(p_n)}\right)=1_{C(p_n)}(y)=0. 
\]
When $e\neq e_{n+1}$, Lemma 4.5 (ii) of \cite{GPS01MathScand} implies 
\[
a\left(1_{G(p_ne)}\right)
\leq a\left(1_{U(p_ne)}\right)=1_{U(p_ne)}(y)=0. 
\]
Thus, $a(1_{G(p_n)})=a(1_{G(p_{n+1})})$ for all $n\in\N$. 
Each character over $y$ therefore has one of two possible labels: 
either $a(1_{G(p_n)})=0$ for all $n$ or 
$a(1_{G(p_n)})=1$ for all $n$. 

Suppose that 
$\pi:X\rtimes_\phi\Z\to Y\rtimes_\psi\Z$ 
factors as $\pi=\pi_2\circ\pi_1$ through an ample groupoid $\cK$. 
Assume that $\pi_1$ is not an isomorphism. 
One can find $y=(e_1,e_2,\dots)\in Y$ such that 
$\#\pi^{-1}(y)=2$ and $\#\pi_2^{-1}(y)=1$. 
The equivalence relation on $X$ defined by $\pi_1$ is invariant under $\phi$. 
Indeed, the images under $\pi_1$ of the arrows 
from two identified points to their $\phi^k$-translates 
have the same source and the same image under $\pi_2$, 
so they coincide by injectivity on source fibers. 
By \cite[Theorem 4.8 (ii)]{GPS01MathScand}, 
we may therefore move $y$ along its orbit and assume $y\in Z$. 
By \cite[Lemma 4.7]{GPS01MathScand}, 
the two characters over $y$ can be labelled $a_0,a_1$ so that 
\[
a_0(f)=a_1(f)=f(y)\quad\forall f\in\mathcal{A}
\]
and 
\[
a_0\left(1_{G(p_n)}\right)=0,\ 
a_1\left(1_{G(p_n)}\right)=1\quad\forall n, 
\]
where $p_n:=(e_1,e_2,\dots,e_n)$. 
We have $\pi_1(\phi^k(a_0))=\pi_1(\phi^k(a_1))$ for any $k\in\Z$, 
where characters are identified with points in $X$. 
Take any $z=(f_1,f_2,\dots)\in Z$ such that $\#\pi^{-1}(z)=2$. 
Put $q_m:=(f_1,f_2,\dots,f_m)$. 
Label the two characters over $z$ as $b_0,b_1$ in the same way. 
For every $m\in\N$, simplicity gives $n>m$ and 
a path $r=(g_{m+1},\dots,g_n)$ in $D_{\mathcal{Q}}$ 
from $r(q_m)$ to $r(p_n)$. 
There is an integer $k_m$ such that $\psi^{k_m}$ 
replaces the prefix $p_n$ by $q_mr$. 
In particular, 
\[
\psi^{k_m}(y)=(f_1,\dots,f_m,g_{m+1},\dots,g_n,e_{n+1},\dots)\in Z. 
\]
By \cite[Lemma 4.5 (i)]{GPS01MathScand}, 
$\psi^{k_m}(G(p_n))=G(q_mr)$. 
Using the constant labels along paths in $Z$, we obtain 
\[
\phi^{k_m}(a_j)\left(1_{G(q_m)}\right)
=\phi^{k_m}(a_j)\left(1_{G(q_mr)}\right)
=a_j\left(1_{G(q_mr)}\circ\psi^{k_m}\right)
=a_j\left(1_{G(p_n)}\right)=j
\]
for $j=0,1$. 
Since $\psi^{k_m}(y)\to z$ and $\mathcal{B}$ is generated by 
$\mathcal{A}$ and the functions $1_{G(p)}$, 
these equalities imply $\phi^{k_m}(a_j)\to b_j$ as $m\to\infty$ 
for $j=0,1$. 
This, together with 
$\pi_1(\phi^{k_m}(a_0))=\pi_1(\phi^{k_m}(a_1))$, 
implies $\pi_1(b_0)=\pi_1(b_1)$. 
Since $z\in Z$ was arbitrary, 
by \cite[Theorem 4.8 (ii)]{GPS01MathScand}, 
every two-point fiber is identified by $\pi_1$. 
Thus $\pi_2$ is one-to-one on the unit space and hence is an isomorphism. 
This proves primeness. 
\end{proof}

\begin{theorem}\label{almost1to1:Z}
Let $\pi:X\rtimes_\phi\Z\to Y\rtimes_\psi\Z$ be 
the factor map as in Theorem~\ref{GPS}. 
\begin{enumerate}
\item The group $\pi^*(\sF(Y\rtimes_\psi\Z))$ is 
maximal in $\sF(X\rtimes_\phi\Z)$ 
if and only if $Q\otimes\Z/2=0$. 
\item The group $\pi^*(\sD(Y\rtimes_\psi\Z))$ is 
maximal in $\sD(X\rtimes_\phi\Z)$. 
\end{enumerate}
\end{theorem}

\begin{proof}
Both $X\rtimes_\phi\Z$ and $Y\rtimes_\psi\Z$ are principal 
and almost finite. 
By Theorem~\ref{GPS}, 
the factor map $\pi$ is prime and at most two-to-one. 
Moreover, if $\cK$ is the transformation groupoid of 
a minimal $\Z$ action, then $H_1(\cK)\cong\Z$ and 
$H_1(\sF(\cK))\cong H_0(\cK,\Z/2)\oplus\Z$ 
(see \cite[Section 4]{Ma06IJM}). 
The summand $H_0(\cK,\Z/2)$ is the image of 
the injective map $\zeta_\cK$ in Remark~\ref{rem:Li-naturality}. 
Under the canonical identifications 
$H_1(Y\rtimes_\psi\Z)\cong\Z\cong H_1(X\rtimes_\phi\Z)$, 
the homomorphism $H_1^*(\pi)$ is the identity map, 
since $\pi$ is induced by an equivariant map. 
The naturality diagram in Remark~\ref{rem:Li-naturality} therefore 
identifies the kernel and cokernel of $H_1(\pi^*)$ 
with those of $H_0^*(\pi)$ with $\Z/2$ coefficients. 

(1)\: 
By Theorem~\ref{thm1:factor} and the observation above, 
$\pi^*(\sF(Y\rtimes_\psi\Z))$ is 
a maximal subgroup of $\sF(X\rtimes_\phi\Z)$ 
if and only if 
\[
H_0^*(\pi):H_0(Y\rtimes_\psi\Z,\Z/2)\to 
H_0(X\rtimes_\phi\Z,\Z/2)\cong G\otimes\Z/2
\]
is surjective. 
As $Q$ is torsion-free, 
\[
\xymatrix@M=8pt{
0 \ar[r] & 
H_0(Y\rtimes_\psi\Z)\otimes\Z/2 \ar[r]^-i & 
G\otimes\Z/2 \ar[r]^-{q} & Q\otimes\Z/2 \ar[r] & 0 
}
\]
is exact, and so we get the conclusion. 

(2)\:
By Theorem~\ref{thm2:factor} and the observation above, 
$\pi^*(\sD(Y\rtimes_\psi\Z))$ is 
a maximal subgroup of $\sD(X\rtimes_\phi\Z)$ 
if and only if 
\[
H_0^*(\pi):H_0(Y\rtimes_\psi\Z,\Z/2)\to 
H_0(X\rtimes_\phi\Z,\Z/2)\cong G\otimes\Z/2
\]
is injective. 
Since $i:H_0(Y\rtimes_\psi\Z)\to G$ is injective and 
$Q$ is torsion-free, this is always true. 
\end{proof}

Writing $\cG:=X\rtimes_\phi\Z$ and $\cH:=Y\rtimes_\psi\Z$, 
the injectivity above and Corollary~\ref{cor:factornormalizer} give 
\[
\Nor(\sD(\cG),\pi^*(\sD(\cH)))=\pi^*(\sD(\cH)),
\qquad
\Nor(\sF(\cG),\pi^*(\sD(\cH)))=\pi^*(\sF(\cH)). 
\]
Thus the factor commutator subgroup is self-normalizing in $\sD(\cG)$, 
while its normalizer in $\sF(\cG)$ is the factor full group.

\subsection{Two-to-one covering maps of minimal $\Z$ actions}

For two-to-one coverings of minimal $\Z$ actions, 
neither of the two factor subgroups is maximal. 
Corollary~\ref{cor:factornormalizer} nevertheless gives 
a maximal subgroup of the ambient commutator subgroup, 
which we describe as a split extension of the factor commutator subgroup. 

Let $\phi:\Z\curvearrowright X$ and $\psi:\Z\curvearrowright Y$ 
be minimal actions on Cantor sets, and let $\pi:X\to Y$ 
be an equivariant two-to-one covering map. 
Write $\cG:=X\rtimes_\phi\Z$, $\cH:=Y\rtimes_\psi\Z$ 
and also denote the induced factor map by $\pi:\cG\to\cH$. 
Both groupoids are principal and almost finite, and $\pi$ is a covering map 
of degree two. Thus Setting~\ref{factorsetting} and 
the hypothesis of Proposition~\ref{prop:factor} (1) are satisfied. 
Choose a clopen trivialization $X\cong Y\times\Z/2$ and write 
\[
\phi(y,a)=(\psi(y),a+c(y))
\]
for a continuous function $c:Y\to\Z/2$. 
For a transformation groupoid $\cK$ of a minimal $\Z$ action, 
put $\sF_0(\cK):=\Ker I_\cK$ and write 
\[
\operatorname{sgn}_\cK:\sF_0(\cK)\longrightarrow H_0(\cK,\Z/2)
\]
for the signature homomorphism. 
It is surjective with kernel $\sD(\cK)$ 
by \cite[Theorem 4.8]{Ma06IJM}, 
and sends a transposition $\tau_U$ to $[1_{r(U)}]$. 

\begin{proposition}\label{prop:Zcovering}
In the setting above, the following hold. 
\begin{enumerate}
\item The pullback map 
$H_0^*(\pi):H_0(\cH,\Z/2)\to H_0(\cG,\Z/2)$ satisfies 
\[
\Ker H_0^*(\pi)=\langle[c]\rangle\cong\Z/2,\qquad
\Coker H_0^*(\pi)\cong H_0(\cH,\Z/2)\neq0. 
\]
Consequently, $\pi^*(\sF(\cH))$ is not maximal in $\sF(\cG)$ 
and $\pi^*(\sD(\cH))$ is not maximal in $\sD(\cG)$. 
\item The subgroup 
\[
M:=\pi^*\left(\operatorname{sgn}_\cH^{-1}(\langle[c]\rangle)\right)
=\Nor(\sD(\cG),\pi^*(\sD(\cH)))
\]
is the unique maximal subgroup of $\sD(\cG)$ 
containing $\pi^*(\sD(\cH))$. 
Moreover, the exact sequence 
\[
\xymatrix@M=8pt{
1 \ar[r] & \pi^*(\sD(\cH)) \ar[r] & M \ar[r] & \Z/2 \ar[r] & 1
}
\]
splits. 
\end{enumerate}
\end{proposition}

The fact that $\Ker H_0^*(\pi)\cong\Z/2$ also follows from 
the dimension group computation in \cite[Lemma 3.6]{Ma02JMSJ}. 

\begin{proof}
(1)\:
Define $T:C(X,\Z/2)\to C(Y,\Z/2)$ by 
$T(f)(y):=f(y,0)+f(y,1)$. 
Then 
\[
\xymatrix@M=8pt@C=18pt{
0 \ar[r] & C(Y,\Z/2) \ar[r]^-{\pi^*} &
C(X,\Z/2) \ar[r]^-T & C(Y,\Z/2) \ar[r] & 0
}
\]
is exact, and its maps commute with the coboundary operators 
$b\mapsto b\circ\psi-b$ and $f\mapsto f\circ\phi-f$. 
Their kernels consist of constant functions by minimality, 
and their cokernels are $H_0(\cH,\Z/2)$ and $H_0(\cG,\Z/2)$, 
respectively. 
On constants, $\pi^*$ is the identity and $T$ is zero. 
The snake lemma therefore gives an exact sequence 
\[
\xymatrix@M=8pt@C=16pt{
0 \ar[r] & \Z/2 \ar[r]^-\delta &
H_0(\cH,\Z/2) \ar[r]^-{H_0^*(\pi)} &
H_0(\cG,\Z/2) \ar[r]^-T & H_0(\cH,\Z/2) \ar[r] & 0.
}
\]
For $f(y,a):=a$, we have $T(f)=1$ and 
$f\circ\phi-f=c\circ\pi$, so $\delta(1)=[c]$. 
This proves the assertions about the kernel and cokernel; 
in particular, $[c]\neq0$. 
As in the proof of Theorem~\ref{almost1to1:Z}, 
the kernel and cokernel of $H_1(\pi^*)$ agree with those of 
$H_0^*(\pi)$ with $\Z/2$ coefficients. 
Theorems~\ref{thm1:factor} and \ref{thm2:factor} 
now give the two nonmaximality assertions. 

(2)\:
Since $H_1^*(\pi):H_1(\cH)\to H_1(\cG)$ is the identity on $\Z$, 
the pullback preserves the index. 
Naturality of the signature gives 
\[
\pi^*(\sF(\cH))\cap\sD(\cG)
=\pi^*\left(\operatorname{sgn}_\cH^{-1}(\langle[c]\rangle)\right). 
\]
Corollary~\ref{cor:factornormalizer} and (1) therefore 
identify this group with the normalizer $M$, 
establish its uniqueness and maximality, and give 
$M/\pi^*(\sD(\cH))\cong\Z/2$. 

To prove splitting, choose a compact open bisection $U\subset\cH$ 
with nonempty, disjoint source and range. 
By \cite[Lemma 3.5]{Ma16Adv}, there is a clopen subset $B\subset s(U)$ 
such that $[1_B]=[c]$ in $H_0(\cH,\Z/2)$. 
The transposition $\tau:=\tau_{UB}$ has signature $[c]$, 
so $\pi^*(\tau)$ is an element of order two in 
$M\setminus\pi^*(\sD(\cH))$. 
Thus $M=\pi^*(\sD(\cH))\rtimes\langle\pi^*(\tau)\rangle$. 
\end{proof}

\begin{example}[Thue--Morse and Toeplitz subshifts]\label{ex:ThueMorse}
Let $\cG=X\rtimes_\phi\Z$ arise from the minimal subshift 
associated with the square of the Thue--Morse substitution, 
\[
\xi(0):=0110,\qquad\xi(1):=1001. 
\]
Exchanging $0$ and $1$ induces an automorphism $\gamma$ of order two. 
The quotient $\cH:=\cG/\langle\gamma\rangle$ is the transformation 
groupoid of the Toeplitz subshift associated with 
\[
\zeta(0):=0101,\qquad\zeta(1):=0111, 
\]
and the quotient map $\pi:\cG\to\cH$ is a two-to-one covering. 
See \cite[Section 4 (1)]{Ma02JMSJ} for this construction 
and \cite{DHS99ETDS} for background on substitution subshifts. 
There are identifications 
\[
H_0(\cH)\cong\Z[1/2]\oplus\Z\cong H_0(\cG), 
\]
with units $(1,0)$ and $(2,0)$, respectively, under which 
$H_0^*(\pi)$ sends $(r,m)$ to $(2r,2m)$ 
(see \cite[Section 4 (1)]{Ma02JMSJ} 
or \cite[Section 7 (4)]{Ma01CJM}). 
Thus the pullback with $\Z/2$ coefficients is the zero map 
from $\Z/2$ to $\Z/2$. 
In Proposition~\ref{prop:Zcovering}, 
we therefore have $\langle[c]\rangle=H_0(\cH,\Z/2)$ and 
\[
M=\pi^*(\sF_0(\cH)). 
\]
Consequently, in the chain 
\[
\pi^*(\sD(\cH))\subsetneq\pi^*(\sF_0(\cH))\subsetneq\sD(\cG), 
\]
the first inclusion has index two and the second is maximal. 
Both $\sD(\cH)$ and $\sD(\cG)$ are finitely generated simple groups 
by \cite[Theorems 4.9 and 5.4]{Ma06IJM}. 
Hence $M$ is a finitely generated, nonsimple maximal subgroup 
of the finitely generated simple group $\sD(\cG)$. 
We also remark that the three groups are amenable 
thanks to \cite{JM13Annals}. 
\end{example}

\begin{example}[Denjoy system]\label{Denjoy1}
Let $\psi:\Z\curvearrowright Y$ be the Denjoy system of type $(\alpha;0)$, 
where $0<\alpha<1/2$ is irrational 
(see \cite{PSS86JOP} and \cite[Section 7 (2)]{Ma01CJM}). 
For $\cH:=Y\rtimes_\psi\Z$, we have 
$H_0(\cH)=\Z e_0\oplus\Z e_1$, where 
$e_0=[1_Y]$ and $e_1=[1_{[0,\alpha)}]$. 
Here and below, intervals are understood as clopen subsets 
of the corresponding Denjoy Cantor set. 
The two-to-one extension associated with $e_1$ modulo two 
is the Denjoy system $(X,\phi)$ of type $(\alpha/2;0,1/2)$. 
Put $\cG:=X\rtimes_\phi\Z$. 
The covering $\pi:X\to Y$ doubles the circle coordinate. 
With respect to the basis of $H_0(\cG)$ given by the classes of 
$1_X$, $1_{[0,\alpha/2)}$ and $1_{[0,1/2)}$, its pullback is 
\[
H_0^*(\pi):\Z^2\longrightarrow\Z^3,\qquad
(m_0,m_1)\longmapsto(m_0,2m_1,0). 
\]
Indeed, the two lifts of $[0,\alpha)$ have the same homology class, 
as their characteristic functions differ by a coboundary of $1_{[0,1/2)}$. 

Write the signature as 
\[
\operatorname{sgn}_\cH(g)=\ep_0(g)e_0+\ep_1(g)e_1
\quad\text{in }H_0(\cH,\Z/2),\qquad g\in\sF_0(\cH). 
\]
Proposition~\ref{prop:Zcovering} shows that 
$M=\pi^*(\Ker\ep_0)$ is maximal in $\sD(\cG)$. 
It has index two in $\pi^*(\sF_0(\cH))$. 
The splitting can be made explicit: 
let $\tau$ exchange $[0,\alpha)$ and $[\alpha,2\alpha)$ 
by $\psi$ and $\psi^{-1}$, respectively, and fix their complement. 
Then $\operatorname{sgn}_\cH(\tau)=e_1$ modulo two, and 
\[
M=\pi^*(\sD(\cH))\rtimes\langle\pi^*(\tau)\rangle. 
\]
\end{example}

\subsection{Two-to-one covering maps for subshifts of finite type}
\label{subsec:SFTcovering}

For SFT groupoids, the maximality criteria can be expressed 
in terms of finite matrices. 
We construct a two-to-one covering from a decomposition of the edge set 
and compute the induced maps using the snake lemma. 

Let $(V,E)$ be a finite directed graph 
as in Section~\ref{subsec:graphgroupoid}, 
and let $A$ be its adjacency matrix. 
We assume that $A$ is irreducible and is not a permutation matrix. 
We write $\cH:=\cG_{(V,E)}$ for the SFT groupoid of $(V,E)$, 
which is minimal and purely infinite. 

We construct a homomorphism $\xi:\cH\to\Z/2$ as follows. 
Let $E_1\subset E$ be a nonempty subset 
such that the subgraph $(V,E\setminus E_1)$ is still irreducible. 
Set $E_0:=E\setminus E_1$. 
Let $A_i$ be the adjacency matrix of $(V,E_i)$ for each $i=0,1$. 
One has $A=A_0+A_1$. 
For $(x,n,y)\in\cH$, choose $m\geq\max\{0,n\}$ such that 
$\sigma^m(x)=\sigma^{m-n}(y)$. 
The integer 
\[
c:=\#\left\{1\leq j\leq m\mid x_j\in E_1\right\}
-\#\left\{1\leq j\leq m{-}n\mid y_j\in E_1\right\}
\]
is independent of the choice of such $m$, since the tails agree. 
Define $\xi(x,n,y):=c\bmod 2$. 
Additivity of these counts shows that $\xi$ is a homomorphism; 
it is continuous because it is constant on each basic bisection 
$Z(\nu,\mu)$. 
Let $\cG:=\cH\times_\xi\Z/2$ be the skew product 
and let $\pi:\cG\to\cH$ be the factor map, 
which is a covering map of degree two. 
The nontrivial element of $\Z/2$ acts on $\cG$ as 
the deck transformation of this covering. 

The covering groupoid $\cG$ has the following graph model. 
Define a finite directed graph $(W,F)$ 
by $W:=V\times\Z/2$, $F:=E\times\Z/2$, 
\[
i(e,a):=(i(e),a)\quad\text{and}\quad 
t(e,a):=\begin{cases}(t(e),a) & e\in E_0 \\ 
(t(e),a{+}1) & e\in E_1. \end{cases}
\]
The adjacency matrix of the graph $(W,F)$ is 
\[
\begin{bmatrix}A_0&A_1\\A_1&A_0\end{bmatrix}, 
\]
which is irreducible 
because $A_0$ is irreducible and $A_1$ is not zero. 
It is not a permutation matrix, since some row of $A=A_0+A_1$ 
has sum at least two. 
One can define the graph homomorphism $\rho:(W,F)\to(V,E)$ 
by letting $\rho(u,a):=u$ and $\rho(e,a):=e$. 
$\rho$ is right-covering in the sense of Section~\ref{subsec:graphgroupoid}. 
Then, the map $h:X_{(W,F)}\to X_{(V,E)}\times\Z/2$ defined by 
\[
h((x_j,a_j)_j):=\left((x_j)_j,a_1\right)
\]
is a homeomorphism. 
By sending 
\[
\left((x_j,a_j)_j,n,(y_j,b_j)_j\right)
\]
to 
\[
\left(((x_j)_j,n,(y_j)_j),a_1\right)\in\cH\times_\xi\Z/2=\cG, 
\]
the homeomorphism $h$ extends to an isomorphism 
from the SFT groupoid of $(W,F)$ to $\cG=\cH\times_\xi\Z/2$. 
It follows that $\cG$ is also minimal and purely infinite 
by \cite[Lemma 6.1]{Ma15crelle}. 

To compute $H_i^*(\pi):H_i(\cH)\to H_i(\cG)$, 
we compare the standard two-term complexes for these graph groupoids. 
Put 
\[
L:=A^t-I,\quad
M:=\begin{bmatrix}A_0^t-I&A_1^t\\A_1^t&A_0^t-I\end{bmatrix},
\quad
N:=A_0^t-A_1^t-I. 
\]
Define homomorphisms 
\[
\Delta:\Z^V\to\Z^V\oplus\Z^V,\quad 
\Delta(x):=\begin{bmatrix}x\\x\end{bmatrix}, 
\]
and
\[
q:\Z^V\oplus\Z^V\to\Z^V,\quad 
q\begin{bmatrix}x\\y\end{bmatrix}:=x-y. 
\]
The following diagram is commutative and has exact rows: 
\[
\xymatrix@M=8pt{
0 \ar[r] & \Z^V \ar[r]^-\Delta \ar[d]_-L &
\Z^V\oplus\Z^V \ar[r]^-q \ar[d]_-M &
\Z^V \ar[r] \ar[d]_-N & 0 \\
0 \ar[r] & \Z^V \ar[r]^-\Delta &
\Z^V\oplus\Z^V \ar[r]^-q &
\Z^V \ar[r] & 0. 
}
\]
Indeed, one has $M\circ\Delta=\Delta\circ L$ and $q\circ M=N\circ q$. 
The snake lemma therefore gives the exact sequence 
\[
\xymatrix@M=8pt@C=14pt{
0 \ar[r] & \Ker L \ar[r]^-\Delta & \Ker M \ar[r]^-q &
\Ker N \ar[r]^-\partial & \Coker L \ar[r]^-\Delta &
\Coker M \ar[r]^-q & \Coker N \ar[r] & 0. 
}
\]
The connecting homomorphism $\partial$ is given by 
\[
\partial(x)=[A_1^t x]\in\Coker L\quad\forall x\in\Ker N. 
\]
In fact, if $x\in\Ker N$, then 
\[
M\begin{bmatrix}x\\0\end{bmatrix}
=\begin{bmatrix}A_1^t x\\A_1^t x\end{bmatrix}
=\Delta(A_1^t x). 
\]
By \cite[Theorem 4.14]{Ma12PLMS}, 
applied to the SFT groupoids $\cH$ and $\cG$, 
we have the standard identifications 
\[
H_1(\cH)=\Ker L,\quad H_0(\cH)=\Coker L,\qquad 
H_1(\cG)=\Ker M,\quad H_0(\cG)=\Coker M. 
\]
Under these identifications, 
a direct computation of the pullback maps shows that 
the homomorphisms labelled $\Delta$ above are 
precisely $H_i^*(\pi):H_i(\cH)\to H_i(\cG)$ for $i=0,1$. 
In particular, $H_1^*(\pi)$ is injective.

By \cite[Corollary 6.24 (1)]{Ma15crelle} 
(see also \cite[Section 5.5]{Ma16Adv}), 
the homomorphisms $\zeta_\cG$ and $\zeta_\cH$ 
in Theorem~\ref{Li} are injective. 
Hence the naturality of Li's exact sequence gives 
the following commutative diagram with exact rows: 
\[
\xymatrix@M=8pt{
0 \ar[r] & H_0(\cG,\Z/2) \ar[r]^-{\zeta_\cG} &
H_1(\sF(\cG)) \ar[r]^-{\eta_\cG} & H_1(\cG) \ar[r] & 0 \\
0 \ar[r] & H_0(\cH,\Z/2) \ar[r]^-{\zeta_\cH} \ar[u]_-{H_0^*(\pi)} &
H_1(\sF(\cH)) \ar[r]^-{\eta_\cH} \ar[u]_-{H_1(\pi^*)} &
H_1(\cH) \ar[r] \ar[u]_-{H_1^*(\pi)} & 0. 
}
\]
Since $H_1^*(\pi)$ is injective, 
the snake lemma shows $\Ker H_1(\pi^*)\cong\Ker H_0^*(\pi)$
and 
\[
\xymatrix@M=8pt{
0 \ar[r] & \Coker H_0^*(\pi) \ar[r] &
\Coker H_1(\pi^*) \ar[r] &
\Coker H_1^*(\pi) \ar[r] & 0. 
}
\]
Consequently, $H_1(\pi^*)$ is injective if and only if
$H_0^*(\pi):H_0(\cH,\Z/2)\to H_0(\cG,\Z/2)$ is injective, 
and $H_1(\pi^*)$ is surjective if and only if 
both $H_0^*(\pi)$ and $H_1^*(\pi)$ are surjective. 

It remains to express these conditions over $\Z/2$. 
Let $\overline L$ and $\overline M$ denote 
the reductions of $L$ and $M$ modulo two. 
Since the reduction of $N$ modulo two is $\overline L$, 
applying the snake lemma to the same diagram over $\Z/2$ 
gives the exact sequence 
\[
\xymatrix@M=8pt{
\Ker\overline L \ar[r]^-{\partial_2} &
\Coker\overline L \ar[r]^-{H_0^*(\pi)} &
\Coker\overline M \ar[r] &
\Coker\overline L \ar[r] & 0, 
}
\]
where 
\[
\partial_2:\Ker\overline L\to\Coker\overline L,\qquad 
\partial_2(x):=[A_1^t x].
\]
Thus, $H_0^*(\pi)$ is injective if and only if $\partial_2=0$, 
and it is surjective if and only if $\Coker\overline L=0$, 
or equivalently, if and only if $\det(I-A)$ is odd. 

\begin{proposition}\label{SFTcovering}
In the setting above, the following hold. 
\begin{enumerate}
\item The group $\pi^*(\sF(\cH))$ is a maximal subgroup of $\sF(\cG)$ 
if and only if $\det(I-A)$ is odd. 
\item The group $\pi^*(\sD(\cH))$ is a maximal subgroup of $\sD(\cG)$ 
if and only if $\partial_2=0$. 
\end{enumerate}
In particular, if $\pi^*(\sF(\cH))$ is maximal in $\sF(\cG)$, 
then $\pi^*(\sD(\cH))$ is maximal in $\sD(\cG)$.
\end{proposition}

\begin{proof}
(1)\:
By Theorem~\ref{thm1:factor} and the observations above, 
$\pi^*(\sF(\cH))$ is maximal in $\sF(\cG)$ if and only if 
both
\[
H_0^*(\pi):H_0(\cH,\Z/2)\to H_0(\cG,\Z/2)
\]
and
\[
H_1^*(\pi):H_1(\cH)\to H_1(\cG)
\]
are surjective. 
The surjectivity of the first map implies that $\det(I-A)$ is odd. 
Conversely, suppose that $\det(I-A)$ is odd. 
Then $\det L$ is odd, and so is $\det N$, 
because $N$ and $L$ have the same reduction modulo two. 
The matrix $M$ acts as $L$ on the diagonal subspace of 
$\Q^V\oplus\Q^V$ and as $N$ on the anti-diagonal subspace. 
Hence $\det M=\det L\det N$. 
It follows that $L$ and $M$ are non-singular over $\Q$, and therefore 
$H_1(\cH)=\Ker L=0$ and $H_1(\cG)=\Ker M=0$. 
Thus $H_1^*(\pi)$ is surjective, 
while the surjectivity of $H_0^*(\pi)$ follows 
from the observation above. 
Theorem~\ref{thm1:factor} now gives the conclusion. 

(2)\:
By Theorem~\ref{thm2:factor} and the observations above, 
$\pi^*(\sD(\cH))$ is maximal in $\sD(\cG)$ if and only if 
$H_0^*(\pi):H_0(\cH,\Z/2)\to H_0(\cG,\Z/2)$ is injective, 
which is equivalent to $\partial_2=0$. 

Finally, if $\det(I-A)$ is odd, then $\overline L$ is invertible. 
In this case both the domain and the codomain of $\partial_2$ are zero, 
and hence $\partial_2=0$. 
\end{proof}

\begin{example}[One-vertex graphs]\label{SFTonevertex}
Suppose that $V$ consists of one vertex, 
that $E_0$ consists of $r$ loops and that $E_1$ consists of $s$ loops, 
where $r,s\geq1$. 
Put $n:=r{+}s$.  
Then $A_0=[r]$, $A_1=[s]$, $A=[n]$ and 
Proposition~\ref{SFTcovering} gives 
\[
\pi^*(\sF(\cH))\text{ is maximal in }\sF(\cG)
\iff n\text{ is even}, 
\]
whereas 
\[
\pi^*(\sD(\cH))\text{ is maximal in }\sD(\cG)
\iff n\text{ is even or }s\text{ is even}. 
\]
Indeed, when $n$ is odd, 
the homomorphism $\partial_2$ is multiplication by $s$ on $\Z/2$. 

The three parity cases also describe how the factor full group 
meets the ambient commutator subgroup. 
The groupoid $\cH$ is the SFT groupoid of the full $n$-shift, 
and $H_0(\cH)\cong\Z/(n{-}1)$, $H_1(\cH)=0$. 
It follows from \cite[Corollary 6.24 (1)]{Ma15crelle} that 
\[
\sF(\cH)/\sD(\cH)\cong H_0(\cH,\Z/2)
\cong\begin{cases}0 & n\text{ is even},\\
\Z/2 & n\text{ is odd}. \end{cases}
\]

Suppose first that $n$ is even.  
Then $r$ and $s$ have the same parity. 
The reduction modulo two of $M$ is either the identity matrix 
or $\begin{bmatrix}0&1\\1&0\end{bmatrix}$, 
and hence $H_0(\cG,\Z/2)=0$. 
Moreover, $\det M=(n{-}1)(r{-}s{-}1)\neq0$, 
because $r{-}s$ is even. 
Thus $H_1(\cG)=0$, and consequently 
\[
\sF(\cH)=\sD(\cH)\quad\text{and}\quad\sF(\cG)=\sD(\cG). 
\]
In this case, 
the two subgroups in Proposition~\ref{SFTcovering} coincide and 
are maximal in the common group $\sF(\cG)=\sD(\cG)$. 

Suppose next that $n$ is odd and $s$ is even, so that $r$ is odd. 
Then the reduction of $M$ modulo two is zero, and
\[
H_0^*(\pi):\Z/2\longrightarrow(\Z/2)^2,
\qquad
x\mapsto\begin{bmatrix}x\\x\end{bmatrix}, 
\]
is injective. 
Thus the induced map on abelianizations is injective. 
Hence $\pi^*(\sF(\cH))\cap\sD(\cG)=\pi^*(\sD(\cH))$. 
In this case, $\pi^*(\sD(\cH))$ is maximal in $\sD(\cG)$, 
whereas $\pi^*(\sF(\cH))$ is not maximal in $\sF(\cG)$. 

Finally, suppose that $n$ is odd and $s$ is odd, so that $r$ is even. 
The reduction of $M$ modulo two is $\begin{bmatrix}1&1\\1&1\end{bmatrix}$. 
Since $\begin{bmatrix}1\\1\end{bmatrix}$ belongs to its image, 
the homomorphism $H_0^*(\pi):H_0(\cH,\Z/2)\to H_0(\cG,\Z/2)$ is zero. 
It follows that the induced map on abelianizations is zero, 
and hence $\pi^*(\sF(\cH))\subset\sD(\cG)$. 
Since $\sF(\cH)/\sD(\cH)\cong\Z/2$ and $\pi^*$ is injective, 
Lemma~\ref{strictness} gives the strict inclusions 
\[
\pi^*(\sD(\cH))\subsetneq\pi^*(\sF(\cH))\subsetneq\sD(\cG). 
\]
Thus $\pi^*(\sD(\cH))$ is not maximal in $\sD(\cG)$. 
On the other hand, Proposition~\ref{prop:factor} implies that 
$\pi^*(\sF(\cH))$ is maximal in $\sD(\cG)$: 
for every $\gamma\in\sD(\cG)\setminus\pi^*(\sF(\cH))$, 
the group generated by $\gamma$ and $\pi^*(\sD(\cH))$, 
and hence also the group generated by $\gamma$ and $\pi^*(\sF(\cH))$, 
is equal to $\sD(\cG)$. 
Nevertheless, $\pi^*(\sF(\cH))$ is not maximal in $\sF(\cG)$. 

For example, when $(r,s)=(1,1)$, one has 
\[
\pi^*(\sF(\cH))=\pi^*(\sD(\cH)),
\qquad
\sF(\cG)=\sD(\cG), 
\]
and the common subgroup is maximal.  
When $(r,s)=(1,2)$, 
$\pi^*(\sD(\cH))$ is maximal in $\sD(\cG)$, 
but $\pi^*(\sF(\cH))$ is not maximal in $\sF(\cG)$. 
When $(r,s)=(2,1)$, one has 
\[
\pi^*(\sD(\cH))\subsetneq\pi^*(\sF(\cH))\subsetneq\sD(\cG), 
\]
and the middle subgroup is maximal in $\sD(\cG)$. 
Thus the maximality of $\pi^*(\sF(\cH))$ in $\sF(\cG)$ 
depends only on $A=[n]$, 
whereas that of $\pi^*(\sD(\cH))$ in $\sD(\cG)$ 
also detects the decomposition $E=E_0\sqcup E_1$. 
\end{example}

\section{Maximal subgroups arising from subgroupoids}
\label{sec:subgroupoid}

In this section, 
we investigate maximal subgroups of topological full groups 
arising from wide open subgroupoids. 

\begin{definition}\label{def:prime}
Let $\cH$ be an ample groupoid and 
let $\cG\subset\cH$ be a wide open subgroupoid. 
We say that $\cG\subset\cH$ is prime 
if there is no proper intermediate wide open subgroupoid. 
More precisely, 
if $\cK\subset\cH$ is a wide open subgroupoid containing $\cG$, 
then $\cK$ is equal to either $\cH$ or $\cG$. 
\end{definition}

When $\cG\subset\cH$ is a wide open subgroupoid of an ample groupoid $\cH$ 
with compact unit space, 
the topological full group $\sF(\cG)$ (resp. $\sD(\cG)$) is 
a subgroup of $\sF(\cH)$ (resp. $\sD(\cH)$). 
Primeness of $\cG\subset\cH$ alone does not ensure that 
these subgroups are maximal. 
In the setting that we study below, 
$\cH$ is the semidirect product of $\cG$ 
by an action of a cyclic group of prime order $p$, 
and the subgroupoid $\cG\subset\cH$ is prime 
(Lemma~\ref{primesemidirectproduct}), 
whereas $\sF(\cG)$ is not maximal in $\sF(\cH)$ and 
$\sD(\cG)$ is not maximal in $\sD(\cH)$. 

The relevant subgroups are instead the normalizers of $\sD(\cG)$. 
Lemma~\ref{normalizer=semidirectproduct} identifies 
its normalizer in $\sF(\cH)$ with $\sF(\cG)\rtimes_\lambda\Z/p$. 
Its normalizer in $\sD(\cH)$ is the unique maximal subgroup 
of $\sD(\cH)$ containing $\sD(\cG)$ 
(Corollary~\ref{cor:normalizermaximal}), 
while maximality of the former is characterized 
by a homological condition in Theorem~\ref{thm1:subgroupoid}. 
The key step is Proposition~\ref{prop:subgroupoid}, 
which gives the following dichotomy: 
a subgroup of $\sF(\cH)$ containing $\sD(\cG)$ 
either normalizes $\sD(\cG)$ or contains the whole of $\sD(\cH)$. 

\begin{setting}\label{subgroupoidsetting}
Let $p\in\N$ be a prime number. 
In the rest of this section, 
we let $\lambda:\Z/p\curvearrowright\cG$ be an action 
on an ample groupoid $\cG$ and 
let $\cH:=\cG\rtimes_\lambda\Z/p$ be the semidirect product. 
We assume the following conditions. 
\begin{itemize}
\item The unit space $\cG^{(0)}$ is a Cantor set. 
\item $\cH$ and $\cG$ are effective, minimal and 
either almost finite or purely infinite. 
\item $\lambda$ induces a free action on $\cG^{(0)}$. 
\end{itemize}
We remark that if $\cH$ is effective, then so is $\cG$; 
if $\cG$ is minimal, then so is $\cH$; 
if $\cG$ is purely infinite, then so is $\cH$. 
\end{setting}

In what follows, we identify $\cG$ with $\cG\times\{0\}\subset\cH$. 
In the skew-product examples considered later, 
the automorphisms $\lambda_a$ are the deck transformations 
of the corresponding cyclic covering. 

\begin{lemma}\label{primesemidirectproduct}
The wide open subgroupoid $\cG=\cG\times\{0\}\subset\cH$ is prime. 
\end{lemma}

\begin{proof}
Suppose that $\cK\subset\cH$ is a wide open subgroupoid 
containing $\cG\times\{0\}$ properly. 
There exists a bisection $V\in\cCO(\cG)$ and $c\neq0$ 
such that $V\times\{c\}\subset\cK$. 
Since $\cG$ is minimal, for any $g\in\cG$, 
we can find $g_1,g_2\in\cG$ such that $g_1gg_2\in V$. 
Then 
\[
(g_1^{-1},0)\cdot(g_1gg_2,c)\cdot(\lambda_{-c}(g_2^{-1}),0)
=(g,c)
\]
belongs to $\cK$. 
Thus $\cG\times\{c\}\subset\cK$. 
The element $c$ generates $\Z/p$, because $p$ is prime. 
Therefore we obtain $\cH=\cK$. 
\end{proof}

The restriction of each $\lambda_a$ to $\cG^{(0)}$ belongs to $\sF(\cH)$ 
and acts on $\sF(\cG)$ by conjugation. 
We use the same notation $\lambda$ for these induced actions. 
Since the action of $\Z/p$ on the Cantor set $\cG^{(0)}$ is free, 
one can find a clopen fundamental domain $F\subset\cG^{(0)}$, i.e. 
\[
\cG^{(0)}=\bigsqcup_{a\in\Z/p}\lambda_a(F). 
\]
The compact open bisections 
\[
U_{a,b}:=\lambda_a(F)\times\{a{-}b\}\subset\cH,
\qquad a,b\in\Z/p, 
\]
form a multisection $u$ of degree $p$ in $\cH$ with $\supp(u)=\cH^{(0)}$. 
For $c\in\Z/p$ with $c\neq0$, 
the cyclic permutation $\pi:b\mapsto b{+}c$ satisfies 
\[
U_\pi=\bigsqcup_{b\in\Z/p}\lambda_{b+c}(F)\times\{c\}
=\cG^{(0)}\times\{c\}, 
\]
and hence $\theta_{U_\pi}=\lambda_c$. 
In particular, $\lambda_c$ is a product of $p{-}1$ transpositions 
of the form $\tau_{U_{a,b}}$. 
When $p$ is an odd prime number, 
$\pi$ is an even permutation, and so 
$\lambda_c$ belongs to $\sA(u)\subset\sA_p(\cH)\subset\sD(\cH)$. 

Since $\cH$ is effective, 
a full bisection implementing an element of $\sF(\cH)$ is unique, 
and hence $\lambda_a$ is not in $\sF(\cG)$ for $a\neq0$. 
Thus the subgroup of $\sF(\cH)$ generated 
by $\sF(\cG)$ and $\{\lambda_a\mid a\in\Z/p\}$ is 
the internal semidirect product $\sF(\cG)\rtimes_\lambda\Z/p$. 

\begin{lemma}\label{normalizer=semidirectproduct}
We have 
$\sF(\cG)\rtimes_\lambda\Z/p
=\Nor(\sF(\cH),\sF(\cG))=\Nor(\sF(\cH),\sD(\cG))$. 
This group is also the stabilizer in $\sF(\cH)$ 
of the partition of $\cG^{(0)}$ into $\cG$-orbits. 
\end{lemma}

\begin{proof}
The inclusions 
\[
\sF(\cG)\rtimes_\lambda\Z/p\subset\Nor(\sF(\cH),\sF(\cG))
\quad\text{and}\quad 
\Nor(\sF(\cH),\sF(\cG))\subset\Nor(\sF(\cH),\sD(\cG))
\]
follow because each $\lambda_a$ preserves $\cG$ and 
$\sD(\cG)$ is characteristic in $\sF(\cG)$. 
It remains to prove 
$\Nor(\sF(\cH),\sD(\cG))\subset\sF(\cG)\rtimes_\lambda\Z/p$. 

By Lemma~\ref{elementofDG}~(1), 
any two distinct points in a $\cG$-orbit are exchanged 
by a transposition in $\sD(\cG)$. 
Thus the $\sD(\cG)$-orbits coincide with the $\cG$-orbits. 

Let $\alpha\in\Nor(\sF(\cH),\sD(\cG))$. 
For every $x\in\cG^{(0)}$ we have 
\[
\alpha\left(\sD(\cG)x\right)
=\left(\alpha\sD(\cG)\alpha^{-1}\right)\alpha(x)
=\sD(\cG)\alpha(x), 
\]
and hence $\alpha$ maps each $\cG$-orbit onto a $\cG$-orbit. 
Conversely, suppose that $\alpha\in\sF(\cH)$ 
maps each $\cG$-orbit onto a $\cG$-orbit. 
We show that $\alpha\in\sF(\cG)\rtimes_\lambda\Z/p$. 
Let $U\in\cCO(\cH)$ be the full bisection with $\alpha=\theta_U$ and write 
\[
U=\bigsqcup_{a\in\Z/p}U_a\times\{a\}, 
\]
where each $U_a\subset\cG$ is a compact open bisection, possibly empty. 
Since $U$ is a full bisection, the clopen sets 
\[
A_a:=s\left(U_a\times\{a\}\right)=\lambda_{-a}(s(U_a)),
\qquad a\in\Z/p, 
\]
form a partition of $\cG^{(0)}$. 
For $x\in A_a$, 
the unique $g\in U_a$ with $s(g)=\lambda_a(x)$ satisfies $\alpha(x)=r(g)$, 
and so $\alpha(x)$ is in the $\cG$-orbit of $\lambda_a(x)$. 

Assume that $A_a$ and $A_b$ are both nonempty for some $a\neq b$. 
Since $\cH$ is effective, 
the points with trivial isotropy in $\cH$ are dense in $\cG^{(0)}$. 
We may therefore choose $x\in A_a$ with $\cH_x=\{x\}$. 
By minimality, choose $y\in A_b$ in the $\cG$-orbit of $x$. 
As $\lambda_b$ is an automorphism of $\cG$, 
the point $\lambda_b(y)$ is in the $\cG$-orbit of $\lambda_b(x)$. 
Also, $\alpha(x)$ and $\alpha(y)$ are in the same $\cG$-orbit, 
because so are $x$ and $y$. 
It follows that 
$\lambda_a(x)$ and $\lambda_b(x)$ are in the same $\cG$-orbit, 
i.e.\ there exists $g\in\cG$ such that 
$r(g)=\lambda_a(x)$ and $s(g)=\lambda_b(x)$. 
Then 
\[
\left(\lambda_{-b}(g^{-1}),a-b\right)\in\cH_x
\]
is a nontrivial isotropy element, because $a\neq b$. 
This contradicts the choice of $x$. 

Therefore there exists a unique $a\in\Z/p$ with $A_a=\cG^{(0)}$, 
that is, $U=U_a\times\{a\}$. 
Then $U_a$ is a full bisection of $\cG$ and 
$\alpha=\theta_{U_a}\lambda_a$ is in $\sF(\cG)\rtimes_\lambda\Z/p$. 
\end{proof}

Following \cite[Section 1]{GV26ETDS}, 
an action on a Cantor set $X$ is said to be topologically primitive 
if it is minimal and preserves no partition of $X$ into closed sets 
other than $\{X\}$ and the partition into singletons. 
These two partitions are called trivial. 
Preserving a partition means permuting its classes. 

\begin{proposition}\label{prop:topologicalprimitivity}
In Setting~\ref{subgroupoidsetting}, 
the groups $\sF(\cG)\rtimes_\lambda\Z/p$ and $\Nor(\sD(\cH),\sD(\cG))$ 
act topologically primitively on $\cG^{(0)}$. 
In fact, $\sD(\cG)$ itself acts topologically primitively. 
\end{proposition}

\begin{proof}
Both groups contain $\sD(\cG)$, 
so it suffices to prove that $\sD(\cG)$ acts topologically primitively. 
By Lemma~\ref{elementofDG}~(1), 
its orbits coincide with the $\cG$-orbits, and hence its action is minimal. 

Put $X:=\cG^{(0)}$ and let $\mathcal P$ be a partition of $X$ 
into closed sets preserved by $\sD(\cG)$. 
If every class is a singleton, then $\mathcal P$ is trivial. 
Otherwise, choose a class $C\in\mathcal P$ containing distinct points $x,y$. 
For any point $z$ in the $\cG$-orbit of $x$ 
with $z\notin\{x,y\}$, choose a further point $w$ 
in that orbit outside $\{x,y,z\}$. 
This is possible because every $\cG$-orbit is infinite. 
Choose compact open bisections $U,V\subset\cG$ containing arrows 
from $x$ to $z$ and from $x$ to $w$, respectively, 
with $s(U)=s(V)$ and with $s(U)$, $r(U)$ and $r(V)$ 
mutually disjoint and avoiding $y$. 
Lemma~\ref{threecycle} gives an element $\alpha\in\sD(\cG)$ 
with $\alpha(x)=z$ and $\alpha(y)=y$. 
Since $\alpha(C)$ and $C$ are classes containing $y$, 
they coincide, and hence $z\in C$. 
Thus $C$ contains the $\cG$-orbit of $x$. 
This orbit is dense, so closedness gives $C=X$. 
Therefore $\mathcal P=\{X\}$ is also trivial. 
\end{proof}

\begin{lemma}\label{strictness:subgroupoid}
We have $\sD(\cH)\not\subset\sF(\cG)\rtimes_\lambda\Z/p$. 
\end{lemma}

\begin{proof}
Take $b\in\Z/p$ with $b\neq0$. 
Since $\lambda$ is free and $\cG^{(0)}$ is a Cantor set, 
we can choose $A\in\cCO(\cG^{(0)})$ with $A\cap\lambda_{-b}(A)=\emptyset$. 
Put $V_1:=A\times\{b\}$, so that $r(V_1)=A$ and $s(V_1)=\lambda_{-b}(A)$. 
Since $\cG$ is minimal, 
shrinking $A$ if necessary, 
we can find a bisection $V_2\in\cCO(\cG)$ such that $r(V_2)=A$ and 
the three sets $A$, $\lambda_{-b}(A)$, $s(V_2)$ are mutually disjoint and 
do not cover $\cG^{(0)}$. 
Let $u$ be the multisection of degree $3$ in $\cH$ 
generated by $V_1$ and $V_2$, and 
let $\alpha\in\sA(u)$ be a non-trivial element. 
The full bisection implementing $\alpha$ contains 
the nonempty unit set $\cG^{(0)}\setminus\supp(u)\subset\cG\times\{0\}$ 
as well as $V_1$ or $V_1^{-1}$, 
which is contained in $\cG\times\{b\}$ or $\cG\times\{-b\}$. 
Hence it is not contained in $\cG\times\{a\}$ for any single $a\in\Z/p$, 
and so $\alpha\in\sA_3(\cH)\subset\sD(\cH)$ is 
not in $\sF(\cG)\rtimes_\lambda\Z/p$. 
\end{proof}

The next three lemmas show how $\sD(\cG)$, together with 
one alternating group linking two translates of a small multisection, 
generates $\sD(\cH)$. 
We first introduce the required multisections. 
A multisection $v=(V_{i,j})_{i,j=0}^2$ of degree $3$ in $\cG$ 
is said to be small 
if the clopen sets 
\[
s(V_{0,j}\times\{a\})=\lambda_{-a}(s(V_{0,j}))\times\{0\}\qquad
j=0,1,2,\ a\in\Z/p
\]
are mutually disjoint. 
When $v$ is a small multisection in $\cG$, 
the compact open bisections $V_{0,j}\times\{a\}$ generate 
a multisection $\bar v$ of degree $3p$ in $\cH$. 
Namely, the index set of $\bar v$ is $\{0,1,2\}\times\Z/p$ and 
each element in $\bar v$ is of the form 
\[
(V_{0,i}\times\{a\})^{-1}(V_{0,j}\times\{b\})
=\lambda_{-a}(V_{i,j})\times\{-a{+}b\}. 
\]
We call $\bar v$ the equivariant extension of $v$. 
Its sub-multisection on $\{0,1,2\}\times\{a\}$ is $\lambda_{-a}(v)$, 
which we call the $a$-th layer. 

For $b\in\Z/p\setminus\{0\}$, 
we write $\bar v^{(b)}$ for the sub-multisection of 
$\bar v$ on $\{0,1,2\}\times\{0,b\}$, 
with distinguished component $(0,0)$. 

The first lemma transports the alternating group on two layers 
and then joins all the layers using the primality of $p$. 

\begin{lemma}\label{AbaruPinH}
Let $H\subset\sF(\cH)$ be a subgroup containing $\sD(\cG)$. 
Suppose that $\sA(\bar v^{(b)})\subset H$ for some small multisection 
$v=(V_{i,j})_{i,j=0}^2$ in $\cG$ and some $b\neq0$. 
Then, for any small multisection $u=(U_{i,j})_{i,j=0}^2$ in $\cG$ 
and any $x\in U_{0,0}$, 
there exists $P\in\cCO(\cG^{(0)})$ such that $x\in P\subset U_{0,0}$ 
and $\sA(\bar u|Q)\subset H$ holds for any nonempty clopen $Q\subset P$. 
\end{lemma}

\begin{proof}
We first prove the conclusion with $\bar u^{(b)}$ in place of $\bar u$. 
Since $\cG$ is minimal, one can find a bisection $W_0\in\cCO(\cG)$ 
such that $x\in r(W_0)\subsetneq U_{0,0}$ and $s(W_0)\subsetneq V_{0,0}$. 
Then 
\[
W_1:=\bigsqcup_{j=0}^2\bigsqcup_{a\in\{0,b\}}
(U_{0,j}\times\{a\})^{-1}(W_0\times\{0\})(V_{0,j}\times\{a\})
\]
is a compact open bisection in $\cG$: 
each summand lies in $\cG\times\{0\}$, and smallness ensures 
that their sources are pairwise disjoint, as are their ranges. 
Both $r(W_1)$ and $s(W_1)$ are proper subsets of $\cG^{(0)}$. 
By Lemma~\ref{elementofDG}~(2), there exists a full bisection 
$W\subset\cG$ such that $W_1\subset W$ and $\theta_W\in\sD(\cG)$. 
Put $P:=r(W_0)$ and take a nonempty clopen set $Q\subset P$. 
Set $Q':=W_0^{-1}QW_0$. 
The conjugates of $\sA(v|Q')\subset\sD(\cG)$ 
by elements of $\sA(\bar v^{(b)})$ generate $\sA(\bar v^{(b)}|Q')$, 
since these conjugates contain all the $3$-cycles on the six components. 
Therefore $\sA(\bar v^{(b)}|Q')\subset H$, and hence 
\[
\sA(\bar u^{(b)}|Q)
=\theta_W\sA(\bar v^{(b)}|Q')\theta_W^{-1}\subset H. 
\]

For each $a\in\Z/p$, apply this conclusion to the small multisection 
$\lambda_{-a}(u)$ at $\lambda_{-a}(x)$. 
We obtain a clopen neighborhood $P_a\subset\lambda_{-a}(U_{0,0})$ 
of $\lambda_{-a}(x)$ such that 
$\sA(\overline{\lambda_{-a}(u)}^{(b)}|R)\subset H$ 
for every nonempty clopen $R\subset P_a$. 
Put $P:=\bigcap_{a\in\Z/p}\lambda_a(P_a)$. 
For a nonempty clopen $Q\subset P$, take $R=\lambda_{-a}(Q)$. 
After relabelling, the multisection 
$\overline{\lambda_{-a}(u)}^{(b)}|R$ is the sub-multisection of 
$\bar u|Q$ on $\{0,1,2\}\times\{a,a+b\}$. 
Thus $H$ contains the alternating group on each of these six-component sets. 
The graph on $\Z/p$ with edges $\{a,a+b\}$ is connected, 
because $p$ is prime and $b\neq0$. 
Alternating groups on finite sets of at least three elements 
with at least two common elements 
generate the alternating group on their union. 
Applying this observation successively along the graph, 
we obtain $\sA(\bar u|Q)\subset H$. 
\end{proof}

We next pass to arbitrary multisections of degree $3$ in $\cH$, 
using two auxiliary components when some of the chosen points 
lie in the same $\Z/p$-orbit. 

\begin{lemma}\label{AuPinH}
Let $H\subset\sF(\cH)$ be a subgroup 
satisfying the conditions in Lemma~\ref{AbaruPinH}. 
Suppose that $u=(U_{i,j})_{i,j=0}^2$ is a multisection in $\cH$ 
and $x\in U_{0,0}$. 
Then there exists $P\in\cCO(\cG^{(0)})$ such that 
$x\in P\subset U_{0,0}$ and 
$\sA(u|Q)\subset H$ holds for any nonempty clopen $Q\subset P$. 
\end{lemma}

\begin{proof}
For $j=1,2$, 
let $h_j=(g_j,a_j)\in U_{0,j}$ be the unique element such that $r(h_j)=x$. 
Put $h_0:=x$ and $x_j:=s(h_j)$ for $j=0,1,2$. 

First suppose that $x_0,x_1,x_2$ belong to distinct $\Z/p$-orbits. 
Since $s(g_j)=\lambda_{a_j}(x_j)$ and 
$\lambda$ is free on $\cG^{(0)}$, the $3p$ points 
\[
\lambda_a(x),\quad \lambda_a(s(g_1)),\quad \lambda_a(s(g_2))
\qquad(a\in\Z/p)
\]
are all distinct. 
By taking sufficiently small compact open bisections 
containing $g_1$ and $g_2$, 
we can find a small multisection $w=(W_{i,j})_{i,j=0}^2$ in $\cG$ 
such that $x\in W_{0,0}\subset U_{0,0}$ and 
\[
W_{0,j}\times\{a_j\}=W_{0,0}U_{0,j}
\quad\text{for $j=1,2$}. 
\]
By Lemma~\ref{AbaruPinH}, 
there exists $P\in\cCO(\cG^{(0)})$ such that $x\in P\subset W_{0,0}$ 
and $\sA(\bar w|Q)\subset H$ for every nonempty clopen $Q\subset P$. 
For each such $Q$, 
the multisection $u|Q$ is, after relabelling, 
the sub-multisection of $\bar w|Q$ on the components $(0,0),(1,a_1),(2,a_2)$. 
Hence 
\[
\sA(u|Q)\subset\sA(\bar w|Q)\subset H. 
\]

We now consider the general case. 
Since $\cG$ is minimal and $\cG^{(0)}$ is a Cantor set, 
the $\cG$-orbit of $x$ is infinite. 
We may therefore choose $f_3,f_4\in\cG$ with $r(f_3)=r(f_4)=x$ such that 
$s(f_3)$ and $s(f_4)$ belong to distinct $\Z/p$-orbits, 
neither of which contains any of $x_0,x_1,x_2$. 
Indeed, only finitely many points need to be avoided. 
Choose bisections $W_3,W_4\in\cCO(\cG)$ containing $f_3,f_4$, respectively, 
such that 
\[
A:=r(W_3)=r(W_4)\subset U_{0,0}
\]
and $s(W_3)$, $s(W_4)$ and $\supp(u|A)$ are mutually disjoint. 

For each $k=0,1,2$, let $z_k$ be the multisection of degree $3$ 
generated by $U_{k,0}W_3$ and $U_{k,0}W_4$. 
The elements $h_k^{-1}f_3$ and $h_k^{-1}f_4$ have common range $x_k$ and 
sources $s(f_3)$ and $s(f_4)$, respectively. 
These three points belong to distinct $\Z/p$-orbits. 
Applying the preceding argument to $z_k$ at $x_k$, 
we obtain $P_k\in\cCO(\cG^{(0)})$ such that 
\[
x_k\in P_k\subset U_{k,0}AU_{0,k}
\]
and $\sA(z_k|R)\subset H$ for every nonempty clopen $R\subset P_k$. 
Put 
\[
P:=\bigcap_{k=0}^2U_{0,k}P_kU_{k,0}. 
\]
Then $P\in\cCO(\cG^{(0)})$ and $x\in P\subset A$. 

Let $Q\subset P$ be a nonempty clopen set. 
For every $k=0,1,2$, we have 
$U_{k,0}QU_{0,k}\subset P_k$, and hence 
\[
\sA(z_k|U_{k,0}QU_{0,k})\subset H.
\]
Let $w$ be the multisection of degree $5$ 
generated by $QU_{0,1},QU_{0,2},QW_3,QW_4$. 
Then $u|Q$ is the sub-multisection of $w$ on the components $\{0,1,2\}$, 
while $z_k|U_{k,0}QU_{0,k}$ is, after relabelling, 
the sub-multisection on the components $\{k,3,4\}$. 
Since the sets $\{k,3,4\}$ share two components, 
their alternating groups generate $\sA(w)$, 
as in the proof of Lemma~\ref{AbaruPinH}. 
Hence $\sA(u|Q)\subset\sA(w)\subset H$. 
\end{proof}

\begin{lemma}\label{AuinH}
Let $H\subset\sF(\cH)$ be a subgroup 
satisfying the conditions in Lemma~\ref{AbaruPinH}. 
Then $\sD(\cH)\subset H$. 
\end{lemma}

\begin{proof}
Suppose that $u=(U_{i,j})_{i,j=0}^2$ is a multisection in $\cH$. 
By Theorem~\ref{A=D}, it suffices to show $\sA(u)\subset H$. 
For any $x\in U_{0,0}$, 
it follows from Lemma~\ref{AuPinH} that 
there exists $P\in\cCO(\cG^{(0)})$ such that $x\in P\subset U_{0,0}$ 
and $\sA(u|Q)\subset H$ holds for any nonempty clopen $Q\subset P$. 
By the compactness of $U_{0,0}$, we can find a partition 
\[
U_{0,0}=\bigsqcup_{i=1}^nQ_i
\]
into nonempty clopen sets $Q_i$ such that $\sA(u|Q_i)\subset H$. 
Each element of $\sA(u)$ is the product of its restrictions 
to the disjoint sets $\supp(u|Q_i)$, so $\sA(u)\subset H$. 
\end{proof}

We can now prove the dichotomy announced above. 
The main task is to obtain an alternating group on two layers 
from an element outside the normalizer. 

\begin{proposition}\label{prop:subgroupoid}
Let $\cG\subset\cH=\cG\rtimes_\lambda\Z/p$ be 
as in Setting~\ref{subgroupoidsetting}. 
Let $H\subset\sF(\cH)$ be a subgroup containing $\sD(\cG)$. 
Then exactly one of the following holds. 
\begin{enumerate}
\item $H$ is contained in $\Nor(\sF(\cH),\sD(\cG))$, that is, 
$H$ normalizes $\sD(\cG)$. 
\item $H$ contains $\sD(\cH)$. 
\end{enumerate}
\end{proposition}

\begin{proof}
By Lemmas~\ref{normalizer=semidirectproduct} and 
\ref{strictness:subgroupoid}, the two alternatives are mutually exclusive. 
Suppose that $H$ is not contained in $\Nor(\sF(\cH),\sD(\cG))$, 
and choose $\gamma\in H\setminus(\sF(\cG)\rtimes_\lambda\Z/p)$. 

We first find an element of $H\setminus\sF(\cG)$ with proper support. 
Write the full bisection implementing $\gamma$ as 
$\bigsqcup_a T_a\times\{a\}$ and put 
$A_a:=\lambda_{-a}(s(T_a))$. 
There exist distinct $a,c\in\Z/p$ with $A_a,A_c\neq\emptyset$. 
Choose $z\in A_a$ with trivial isotropy in $\cH$ and 
$z'\in A_c$ in the $\cG$-orbit of $z$. 
As in the proof of Lemma~\ref{normalizer=semidirectproduct}, 
$\lambda_a(z)$ and $\lambda_c(z)$ lie in different $\cG$-orbits. 
Hence so do $\gamma(z)$ and $\gamma(z')$. 
By Lemma~\ref{threecycle}, a sufficiently small multisection of degree $3$ 
in $\cG$ gives 
$\delta\in\sD(\cG)$ with $\delta(z)=z'$ and 
$\supp(\delta)\neq\cG^{(0)}$. 
Then $\beta:=\gamma\delta\gamma^{-1}\in H$ has proper support and 
does not belong to $\sF(\cG)$. 
Write its full bisection as 
\[
U=\bigsqcup_{a\in\Z/p}U_a\times\{a\}. 
\]
Since $\beta$ fixes a nonempty clopen set, $U_0\neq\emptyset$; 
since $\beta\notin\sF(\cG)$, also $U_b\neq\emptyset$ for some $b\neq0$. 

By minimality, choose $g\in U_b$ and $g'\in U_0$ such that 
$x:=s(g)$ and $x':=s(g')$ lie in the same $\cG$-orbit 
but in distinct $\Z/p$-orbits. 
Their ranges are distinct, since the sets $r(U_a)$ are disjoint. 
Small compact open neighborhoods of $g$ and $g'$ form a bisection 
with proper source and range. 
By Lemma~\ref{elementofDG}~(2), this bisection extends to a full 
bisection $V\subset\cG$ with $\theta_V\in\sD(\cG)$. 
Replacing $U$ by $V^{-1}U$ and $\beta$ by $\theta_V^{-1}\beta$, 
we may therefore assume that $x\in U_b$ and $x'\in U_0$ are units. 

Choose $g_1,g_2\in\cG$ with common range $x$ such that 
$s(g_1)=x'$ and $s(g_2)$ belongs to neither 
the $\Z/p$-orbit of $x$ nor that of $x'$. 
Choose sufficiently small compact open bisections $B_1,B_2\subset\cG$ 
containing $g_1,g_2$, respectively, with 
$r(B_1)=r(B_2)\subset U_b$ and $s(B_1)\subset U_0$, 
so that the multisection they generate is small. 
Applying the construction in the proof of Lemma~\ref{elementofDG}~(1) 
inside $B_1$, choose a compact open bisection $W_1\subset B_1$ 
containing $g_1$ with $\tau_{W_1}\in\sD(\cG)$, and put 
$W_2:=r(W_1)B_2$. 
Let $v$ be the small multisection generated by $W_1,W_2$. 
The element 
\[
\alpha:=\tau_{W_1}\beta^{-1}\tau_{W_1}\beta\in H 
\]
is the $3$-cycle on the components $(0,0)$, $(1,0)$ and $(0,b)$ 
of $\bar v^{(b)}$. 
Indeed, $\beta$ agrees with $\lambda_b$ on $\lambda_{-b}(r(W_1))$ 
and with the identity on $s(W_1)$. 
Thus $\beta^{-1}\tau_{W_1}\beta$ exchanges 
$s(W_1)$ and $\lambda_{-b}(r(W_1))$, 
while $\tau_{W_1}$ exchanges $r(W_1)$ and $s(W_1)$. 

The alternating groups on each of the two layers are 
$\sA(v)$ and $\sA(\lambda_{-b}(v))$, both contained in $\sD(\cG)$. 
Together with $\alpha$ they generate $\sA(\bar v^{(b)})$. 
To see this, label the layers by $\{0,1,2\}$ and $\{3,4,5\}$. 
The cycles $(0\ 1\ 2)$ and $(0\ 1\ 3)$ generate 
the alternating group on $\{0,1,2,3\}$; its conjugates by 
$(3\ 4\ 5)$ then generate the alternating group on all six points. 
Thus $\sA(\bar v^{(b)})\subset H$, and Lemma~\ref{AuinH} 
gives $\sD(\cH)\subset H$. 
\end{proof}

\begin{corollary}\label{cor:normalizermaximal}
The normalizer $\Nor(\sD(\cH),\sD(\cG))$ is 
the unique maximal subgroup of $\sD(\cH)$ containing $\sD(\cG)$. 
\end{corollary}

\begin{proof}
By Lemmas~\ref{normalizer=semidirectproduct} and 
\ref{strictness:subgroupoid}, 
$\Nor(\sD(\cH),\sD(\cG))$ is a proper subgroup of $\sD(\cH)$. 
By Proposition~\ref{prop:subgroupoid}, every proper subgroup 
of $\sD(\cH)$ containing $\sD(\cG)$ is contained in this normalizer. 
This proves both uniqueness and maximality. 
\end{proof}

We now determine when the normalizer in $\sF(\cH)$ is maximal, 
and describe the maximal subgroup 
in Corollary~\ref{cor:normalizermaximal} more explicitly. 
Let $\iota:\cG\to\cH$ and $\iota_*:\sF(\cG)\to\sF(\cH)$ denote 
the inclusion maps. 
Applying Theorem~\ref{Li} to $\cG$ and to $\cH$, 
we obtain the exact rows of the following diagram: 
\[
\xymatrix@M=8pt{
H_0(\cH,\Z/2) \ar[r]^-{\zeta_\cH} &
H_1(\sF(\cH)) \ar[r]^-{\eta_\cH} & H_1(\cH) \ar[r] & 0 \\
H_0(\cG,\Z/2) \ar[r]^-{\zeta_\cG} \ar[u]_-{H_0(\iota)} &
H_1(\sF(\cG)) \ar[r]^-{\eta_\cG} \ar[u]_-{H_1(\iota_*)} &
H_1(\cG) \ar[r] \ar[u]_-{H_1(\iota)} & 0.
}
\]
Since $\cG$ and $\cH$ have the same unit space, 
the homomorphism $H_0(\iota)$ is surjective. 
The diagram is commutative by the same argument 
as in Remark~\ref{rem:Li-naturality}. 
A diagram chase therefore shows that 
$H_1(\iota_*)$ is surjective if and only if $H_1(\iota)$ is surjective. 

As observed above,
each $\lambda_a$ is a product of transpositions $\tau_U$ with $U\in\cCO(\cH)$, 
and the class of such $\tau_U$ in $H_1(\sF(\cH))$ is $\zeta_\cH([1_{r(U)}])$. 
Hence $[\lambda_a]$ belongs to $\Ima\zeta_\cH$. 
Since $H_0(\iota)$ is surjective and the left square commutes, 
$\Ima\zeta_\cH$ is contained in $\Ima H_1(\iota_*)$, and we obtain 
\begin{equation}\label{eq:lambdaabelianization:subgroupoid}
[\lambda_a]\in\Ima H_1(\iota_*)
\qquad(a\in\Z/p). 
\end{equation}
In particular, $\lambda_a\in\sF(\cG)\sD(\cH)$ for every $a\in\Z/p$. 

\begin{theorem}\label{thm1:subgroupoid}
Let $\cG\subset\cH=\cG\rtimes_\lambda\Z/p$ be 
as in Setting~\ref{subgroupoidsetting}. 
The following conditions are equivalent. 
\begin{enumerate}
\item $\Nor(\sF(\cH),\sD(\cG))$ $(=\sF(\cG)\rtimes_\lambda\Z/p)$ 
is a maximal subgroup of $\sF(\cH)$. 
\item $H_1(\iota_*):H_1(\sF(\cG))\to H_1(\sF(\cH))$ is surjective. 
\item $H_1(\iota):H_1(\cG)\to H_1(\cH)$ is surjective. 
\end{enumerate}
\end{theorem}

\begin{proof}
The equivalence of (2) and (3) follows from the diagram above. 

(1)$\implies$(2)\:
By Lemma~\ref{normalizer=semidirectproduct} and 
\eqref{eq:lambdaabelianization:subgroupoid}, 
\[
\Nor(\sF(\cH),\sD(\cG))
\subset\sF(\cG)\sD(\cH)\subset\sF(\cH). 
\]
The first inclusion is strict by Lemma~\ref{strictness:subgroupoid}. 
Thus (1) implies $\sF(\cG)\sD(\cH)=\sF(\cH)$, which is equivalent to (2). 

(2)$\implies$(1)\:
Let $H\subset\sF(\cH)$ be a subgroup 
properly containing $\Nor(\sF(\cH),\sD(\cG))$. 
By Proposition~\ref{prop:subgroupoid}, $H$ contains $\sD(\cH)$. 
Since $\sF(\cG)\subset H$ and $H_1(\iota_*)$ is surjective, 
we obtain $H=\sF(\cH)$.
The normalizer is a proper subgroup 
by Lemmas~\ref{normalizer=semidirectproduct} and
\ref{strictness:subgroupoid}, and so it is maximal. 
\end{proof}

By Corollary~\ref{cor:normalizermaximal}, 
$\Nor(\sD(\cH),\sD(\cG))$ is maximal 
without any additional homological assumption. 
We next identify this normalizer in terms of $\sD(\cG)$, 
using the canonical isomorphism 
\begin{equation}\label{eq:abelianizationkernel:subgroupoid}
\left(\sF(\cG)\cap\sD(\cH)\right)/\sD(\cG)\cong\Ker H_1(\iota_*). 
\end{equation}

\begin{theorem}\label{thm2:subgroupoid}
Let $\cG\subset\cH=\cG\rtimes_\lambda\Z/p$ be 
as in Setting~\ref{subgroupoidsetting}. 
Suppose that $\lambda_a\in\sD(\cH)$ for every $a\in\Z/p$ 
(this holds automatically when $p\geq3$). 
Then 
\[
\Nor(\sD(\cH),\sD(\cG))
=\left(\sF(\cG)\cap\sD(\cH)\right)\rtimes_\lambda\Z/p. 
\]
Moreover, the following conditions are equivalent. 
\begin{enumerate}
\item $H_1(\iota_*):H_1(\sF(\cG))\to H_1(\sF(\cH))$ is injective. 
\item $\Nor(\sD(\cH),\sD(\cG))=\sD(\cG)\rtimes_\lambda\Z/p$. 
\item $\sD(\cG)\rtimes_\lambda\Z/p$ is a maximal subgroup of $\sD(\cH)$. 
\end{enumerate}
\end{theorem}

\begin{proof}
The displayed identity follows 
from Lemma~\ref{normalizer=semidirectproduct}, 
since $\gamma\lambda_a\in\sD(\cH)$ if and only if 
$\gamma\in\sD(\cH)$ for $\gamma\in\sF(\cG)$. 
Together with \eqref{eq:abelianizationkernel:subgroupoid}, 
it gives the equivalence of (1) and (2). 
Corollary~\ref{cor:normalizermaximal} gives (2)$\implies$(3). 
Conversely, if (3) holds, then 
\[
\sD(\cG)\rtimes_\lambda\Z/p
\subset\Nor(\sD(\cH),\sD(\cG))\subsetneq\sD(\cH)
\]
forces equality in the first inclusion, proving (2). 
\end{proof}

It remains to describe the normalizer when $p=2$ and 
the nontrivial element of the action does not belong to $\sD(\cH)$. 
In this case, it can be corrected by an element of $\sF(\cG)$. 

\begin{theorem}\label{thm3:subgroupoid}
Let $\cG\subset\cH=\cG\rtimes_\lambda\Z/2$ be 
as in Setting~\ref{subgroupoidsetting}, 
and let $a\in\Z/2$ be the nontrivial element. 
Suppose that $\lambda_a\notin\sD(\cH)$. 
Then there exists $\beta\in\sF(\cG)$ such that $\beta\lambda_a\in\sD(\cH)$. 
If
\[
H_1(\iota_*):H_1(\sF(\cG))\longrightarrow H_1(\sF(\cH))
\]
is injective, then, for any such $\beta$, the set 
\[
M:=\sD(\cG)\sqcup\sD(\cG)\beta\lambda_a
\]
coincides with $\Nor(\sD(\cH),\sD(\cG))$. 
In particular, $M$ is independent of the choice of $\beta$, 
is a maximal subgroup of $\sD(\cH)$, 
and contains $\sD(\cG)$ as a normal subgroup of index $2$. 
\end{theorem}

\begin{proof}
By \eqref{eq:lambdaabelianization:subgroupoid}, 
we may choose $\beta\in\sF(\cG)$ 
whose class in $H_1(\sF(\cH))$ is $-[\lambda_a]$. 
Then $\beta\lambda_a\in\sD(\cH)$. 

Assume that $H_1(\iota_*)$ is injective. 
By \eqref{eq:abelianizationkernel:subgroupoid}, 
$\sF(\cG)\cap\sD(\cH)=\sD(\cG)$. 
For $\gamma\in\sF(\cG)$, we have 
\[
\gamma\lambda_a\in\sD(\cH)
\iff\gamma\beta^{-1}\in\sD(\cH)
\iff\gamma\beta^{-1}\in\sD(\cG). 
\]
Taking the intersection of the equality 
in Lemma~\ref{normalizer=semidirectproduct} with $\sD(\cH)$ 
therefore gives 
\[
\Nor(\sD(\cH),\sD(\cG))=\sD(\cG)\sqcup\sD(\cG)\beta\lambda_a=M. 
\]
This also proves that $M$ is independent of $\beta$ 
and that $\sD(\cG)$ is normal in $M$. 
The two cosets are distinct, 
since $\beta\lambda_a\notin\sF(\cG)$, so the index is $2$. 
Finally, maximality follows from Corollary~\ref{cor:normalizermaximal}. 
\end{proof}

\begin{remark}\label{rem:injectivitycriterion}
Suppose that $\zeta_\cG$ is zero and 
$H_1(\iota):H_1(\cG)\to H_1(\cH)$ is injective. 
The commutative diagram above shows that $H_1(\iota_*)$ is injective. 
Moreover, surjectivity of $H_0(\iota)$ gives $\zeta_\cH=0$, 
so $\lambda_a\in\sD(\cH)$ for every $a\in\Z/p$. 
Thus Theorem~\ref{thm2:subgroupoid} applies, even when $p=2$, and 
the normalizer is the maximal subgroup $\sD(\cG)\rtimes_\lambda\Z/p$. 
\end{remark}

\section{Examples II}\label{sec:examplesII}

We first apply Section~\ref{sec:subgroupoid} to minimal $\Z$ actions, 
where maximal subgroups of derived groups can arise from nonsplit extensions. 
For SFT groupoids, we express the criteria in terms of finite matrices. 
We conclude with an orbit-breaking example 
outside the semidirect product setting.

\subsection{Finite order automorphisms of minimal $\Z$ actions}

The following realization theorem provides finite order automorphisms 
that act trivially on zeroth homology \cite[Theorem 3.7]{Ma02JMSJ}. 

\begin{theorem}
Let $(D,D^+,u)$ be a countable unital simple dimension group 
with $D\not\cong\Z$, and let $p\geq2$ be an integer. 
Suppose that $D/pD$ is isomorphic to $\Z/p$ as an abelian group 
and $u$ is divisible by $p$. 
There exist an ample groupoid $\cG:=X\rtimes_\phi\Z$ of 
a minimal $\Z$ action on a Cantor set 
and an action $\lambda:\Z/p\curvearrowright\cG$ 
such that the following hold. 
\begin{enumerate}
\item $H_0(\cG)$ is isomorphic to $D$ as a unital dimension group. 
\item $\lambda$ arises from 
a homeomorphism on $\cG^{(0)}=X$ of order $p$ commuting with $\phi$. 
\item The induced action 
$H_0(\lambda):\Z/p\curvearrowright H_0(\cG)$ is trivial. 
\end{enumerate}
\end{theorem}

For every nonzero $a\in\Z/p$ and every $x\in X$, 
the point $\lambda_a(x)$ does not belong to the $\phi$-orbit of $x$. 
Indeed, if $\lambda_a(x)=\phi^n(x)$ for some $n\in\Z$, then 
$\phi^{-n}\circ\lambda_a$ has a fixed point. 
Since it commutes with $\phi$, 
its fixed point set is a nonempty closed $\phi$-invariant subset of $X$, 
and hence equals $X$ by minimality. 
Thus $\lambda_a=\phi^n$. 
Since $\lambda_a$ is nontrivial, we have $n\neq0$. 
The finite order of $\lambda_a$ would then imply $\phi^{np}=\id$, 
which contradicts the minimality of $\phi$ on the Cantor set $X$. 
In particular, the action of $\Z/p$ on $X$ is free. 

Let $\lambda:\Z/p\curvearrowright\cG$ be as above 
and let $\cH:=\cG\rtimes_\lambda\Z/p$. 
Let $Y:=X/(\Z/p)$ and let $q:X\to Y$ be the quotient map. 
Let $\psi:\Z\curvearrowright Y$ be the minimal $\Z$ action induced by $\phi$. 
The maps $\lambda_a$ act as deck transformations of the covering $q:X\to Y$. 
Choose a clopen fundamental domain $A\subset X$ for the action of $\Z/p$. 
Then $\cH|A$ is naturally isomorphic to $Y\rtimes_\psi\Z$, 
and hence $\cH$ is Kakutani equivalent to the quotient system. 
To see directly that $\cH$ itself is a groupoid of a minimal $\Z$ action, 
identify $A$ with $Y$ via $q|A$ and let 
$\bar\phi:A\to A$ be the homeomorphism corresponding to $\psi$. 
Writing 
\[
X=\bigsqcup_{j=0}^{p-1}\lambda_j(A), 
\]
define a homeomorphism $T:X\to X$ by 
\[
T(\lambda_j(x))=\lambda_{j+1}(x)\quad (0\leq j<p{-}1),
\qquad T(\lambda_{p-1}(x))=\bar\phi(x) 
\]
for $x\in A$. 
Then $T$ is minimal and one checks that 
$X\rtimes_T\Z$ is naturally isomorphic to $\cH$. 
Thus both $\cG$ and $\cH$ are principal and almost finite. 

The group $H_0(\cH)$ is obtained from $H_0(\cG)$ by imposing 
the relations $v=H_0(\lambda_a)(v)$. 
Since $H_0(\lambda)$ is trivial, 
$H_0(\iota)$ is an isomorphism, both over $\Z$ and over $\Z/2$. 
The quotient homomorphism $\cH\to Y\rtimes_\psi\Z$ 
induces the Kakutani isomorphism on homology and maps 
the full bisection implementing $\phi$ $p$-to-one 
onto that implementing $\psi$. 
Thus $H_1(\iota)$ sends $1\in\Z\cong H_1(\cG)$ 
to $p\in\Z\cong H_1(\cH)$. 

For groupoids of minimal $\Z$ actions, the map $\zeta$ 
in Li's exact sequence is injective by \cite[Theorem 4.8]{Ma06IJM}. 
In the commutative diagram preceding Theorem~\ref{thm1:subgroupoid}, 
$H_0(\iota)$ is an isomorphism and $H_1(\iota)$ is injective. 
A diagram chase therefore shows that $H_1(\iota_*)$ is injective. 

Suppose now that $p=2$, and let $a\in\Z/2$ be the nontrivial element. 
Retain the clopen fundamental domain $A\subset X$ chosen above, 
so that $X=A\sqcup\lambda_a(A)$. 
Let $U\in\cCO(\cH)$ be the bisection implementing $\lambda_a$ 
from $\lambda_a(A)$ onto $A$. 
Then $\lambda_a=\tau_U$. 
Under the canonical identification of the subgroup generated by 
the classes of transpositions with $H_0(\cH,\Z/2)$, 
the class of $\lambda_a$ corresponds to $[1_A]$. 
On the other hand, under the identification $H_0(\cH)\cong D$ above, 
we have $2[1_A]=[1_X]=u$. 
Since the dimension group $D$ is torsion-free, it follows that 
\[
\lambda_a\in\sD(\cH)
\iff [1_A]=0\text{ in }H_0(\cH,\Z/2)
\iff u\in4D. 
\]

When $p$ is prime, $\cG\subset\cH$ satisfies Setting 
\ref{subgroupoidsetting}. 
Since $H_1(\iota)$ is multiplication by $p$, it is not surjective. 
Thus Theorem~\ref{thm1:subgroupoid} shows that 
$\sF(\cG)\rtimes_\lambda\Z/p$ is not maximal in $\sF(\cH)$. 
Together with Theorems~\ref{thm2:subgroupoid} and~\ref{thm3:subgroupoid}, 
the preceding computations give the following. 

\begin{corollary}\label{cor:finiteorderZ}
Let $p$ be prime and let $\cG\subset\cH=\cG\rtimes_\lambda\Z/p$ be as above. 
\begin{enumerate}
\item The group $\sF(\cG)\rtimes_\lambda\Z/p$ is not 
a maximal subgroup of $\sF(\cH)$. 
\item When $p$ is an odd prime number, 
$\sD(\cG)\rtimes_\lambda\Z/p$ is a maximal subgroup of $\sD(\cH)$. 
\item When $p=2$ and $u\in D$ is divisible by $4$, 
$\sD(\cG)\rtimes_\lambda\Z/p$ is a maximal subgroup of $\sD(\cH)$. 
\item When $p=2$ and $u\in D$ is not divisible by $4$, 
the group $M$ in Theorem~\ref{thm3:subgroupoid} is 
a maximal subgroup of $\sD(\cH)$. 
\end{enumerate}
\end{corollary}

\begin{remark}\label{rem:nonsplit:subgroupoid}
Suppose that $p=2$ and $u\notin4D$. 
The extension 
\[
\xymatrix@M=8pt{
1 \ar[r] & \sD(\cG) \ar[r] & M \ar[r] & \Z/2 \ar[r] & 1
}
\]
in Theorem~\ref{thm3:subgroupoid} does not split. 
Indeed, every element of $M\setminus\sD(\cG)$ has the form 
$\eta\lambda_a$ with $\eta\in\sF(\cG)$, and is fixed-point-free 
since $\lambda_a(x)$ is never in the $\cG$-orbit of $x$. 
If such an element $\sigma$ had order two, it would admit a clopen 
fundamental domain $B\subset X$. 
Under the identification $H_0(\cH)\cong D$, we would have 
$2[1_B]=u$, and hence $[1_B]=u/2\notin2D$. 
Since $\sigma$ is the transposition exchanging $B$ and $\sigma(B)$, 
its class in $H_1(\sF(\cH))$ would correspond to the nonzero class 
of $[1_B]$ in $D/2D$. 
This contradicts $\sigma\in M\subset\sD(\cH)$. 
Thus the nontrivial coset contains no element of order two, 
so the extension does not split. 
\end{remark}

\subsection{Finite order automorphisms of SFT groupoids}

We next consider a canonical class of finite order automorphisms 
of SFT groupoids. 
Let $(V,E)$ be a finite directed graph 
as in Section~\ref{subsec:graphgroupoid}, 
and let $A$ be its adjacency matrix. 
Let $p$ be an odd prime number and suppose that 
\[
E=\bigsqcup_{a\in\Z/p}E_a. 
\]
We assume that the subgraph $(V,E_0)$ is irreducible and 
that $E_a$ is nonempty for some $a\neq0$. 
These assumptions imply that 
$A$ is irreducible and is not a permutation matrix. 
Write $\cK:=\cG_{(V,E)}$. 
Let $A_a$ denote the adjacency matrix of $(V,E_a)$, so that 
\[
A=\sum_{a\in\Z/p}A_a. 
\]
Define $\ell:E\to\Z/p$ by $\ell(e)=a$ for $e\in E_a$. 
For a finite path $\mu=\mu_1\cdots\mu_n$, put 
\[
\ell(\mu):=\sum_{j=1}^n\ell(\mu_j)\in\Z/p. 
\]
If $(x,n,y)\in\cK$ and $n=k-l$ with $\sigma^k(x)=\sigma^l(y)$, then 
\[
\xi(x,n,y):=\ell(x_1\cdots x_k)-\ell(y_1\cdots y_l)
\]
is independent of the choice of sufficiently large $k,l$. 
Thus $\xi:\cK\to\Z/p$ is a continuous homomorphism. 
Let 
\[
\cG:=\cK\times_\xi\Z/p. 
\]
As in Section~\ref{subsec:SFTcovering}, $\cG$ is itself an SFT groupoid. 
Indeed, define a finite directed graph $(W,F)$ by 
\[
W:=V\times\Z/p,\qquad F:=E\times\Z/p, 
\]
and 
\[
i(e,b):=(i(e),b),\qquad 
t(e,b):=(t(e),b+\ell(e)). 
\]
The SFT groupoid of $(W,F)$ is naturally isomorphic to $\cG$. 
The graph $(W,F)$ is irreducible: 
the edges in $E_0$ allow us to move between vertices in each layer, 
and an edge in some $E_a$, $a\neq0$, allows us to move between all layers, 
because $a$ generates $\Z/p$. 
It is not a permutation graph. 
Hence $\cG$ is minimal and purely infinite. 

There is a canonical action $\lambda:\Z/p\curvearrowright\cG$ by 
deck transformations, given by 
\[
\lambda_c(g,b):=(g,b+c). 
\]
Its action on the unit space is free. 
Put 
\[
\cH:=\cG\rtimes_\lambda\Z/p. 
\]
Let $\mathcal R_p$ denote the full equivalence relation on $\Z/p$. 
There exists a natural isomorphism $\cH\cong\cK\times\mathcal R_p$ 
sending $((g,b),c)$ to $(g,b,b+\xi(g)-c)$. 
In particular, the reduction of $\cH$ to 
$\cK^{(0)}\times\{0\}$ is isomorphic to $\cK$, 
and this clopen subset is full. 
Thus $\cH$ is Kakutani equivalent to $\cK$, 
and hence $\cH$ is minimal, effective and purely infinite. 
Consequently, $\cG\subset\cH$ satisfies Setting~\ref{subgroupoidsetting}. 

We now compute the homomorphisms on homology 
induced by the inclusion $\iota:\cG\to\cH$. 
Put 
\[
\Lambda:=\bigoplus_{b\in\Z/p}\Z^V,\qquad 
\Sigma:\Lambda\to\Z^V,\qquad 
\Sigma((x_b)_b):=\sum_{b\in\Z/p}x_b, 
\]
and let $\Lambda_0:=\Ker\Sigma$. 
Define 
\[
L:=A^t-I:\Z^V\to\Z^V
\]
and $M:\Lambda\to\Lambda$ by 
\[
(Mx)_b:=\sum_{a\in\Z/p}A_a^t x_{b-a}-x_b
\qquad(b\in\Z/p). 
\]
The map $M$ is $B^t-I$, where $B$ is the adjacency matrix of $(W,F)$. 
Moreover, $\Sigma\circ M=L\circ\Sigma$. 
Therefore $M$ leaves $\Lambda_0$ invariant. 
We write 
\[
N:=M|\Lambda_0:\Lambda_0\to\Lambda_0. 
\]
Notice that $N$ is an explicitly determined integer matrix of size 
$(p{-}1)\#V$ after choosing a basis of $\Lambda_0$. 

We have a commutative diagram with exact rows 
\[
\xymatrix@M=8pt{
0 \ar[r] & \Lambda_0 \ar[r] \ar[d]_-N & 
\Lambda \ar[r]^-\Sigma \ar[d]_-M &
\Z^V \ar[r] \ar[d]_-L & 0 \\
0 \ar[r] & \Lambda_0 \ar[r] &
\Lambda \ar[r]^-\Sigma &
\Z^V \ar[r] & 0. 
}
\]
The standard computation of homology for SFT groupoids, together with 
the Kakutani equivalence between $\cH$ and $\cK$, gives 
\[
H_1(\cG)=\Ker M,\quad H_0(\cG)=\Coker M,
\qquad
H_1(\cH)=\Ker L,\quad H_0(\cH)=\Coker L. 
\]
Under these identifications, the homomorphisms induced by $\iota$ 
are the maps induced by $\Sigma$. 
Indeed, composing $\iota$ with the projection 
$\cH\cong\cK\times\mathcal R_p\to\cK$ forgets the layer coordinate. 
On the standard two-term complexes for SFT homology, 
this gives $\Sigma$ in both degrees: 
the basis vector corresponding to $(v,b)$ is sent to 
that corresponding to $v$. 
Hence the snake lemma gives an exact sequence 
\[
\xymatrix@M=8pt@C=14pt{
0 \ar[r] & \Ker N \ar[r] & H_1(\cG) 
\ar[r]^-{H_1(\iota)} & H_1(\cH) \ar[r]^-\partial &
\Coker N \ar[r] & H_0(\cG) 
\ar[r]^-{H_0(\iota)} & H_0(\cH) \ar[r] & 0. 
}
\]
Thus all the homomorphisms relevant to the inclusion can be computed 
from the finite matrices $L,M,N$. 

For the maximality question, 
we only need a particularly simple consequence modulo two. 
Let $\overline L,\overline M,\overline N$ denote the reductions modulo two. 
Since $p$ is odd, the diagonal homomorphism 
\[
\Delta:(\Z/2)^V\to((\Z/2)^V)^{\Z/p},\qquad
\Delta(x):=(x)_b, 
\]
satisfies $\Sigma\circ\Delta=\id$. 
Moreover, $\overline M\circ\Delta=\Delta\circ\overline L$. 
It follows that 
\[
((\Z/2)^V)^{\Z/p}=\Delta((\Z/2)^V)\oplus(\Lambda_0\otimes\Z/2)
\]
and, with respect to this decomposition, 
\[
\overline M=\overline L\oplus\overline N. 
\]
Consequently, 
\[
H_0(\cG,\Z/2)\cong H_0(\cH,\Z/2)\oplus\Coker\overline N, 
\]
and $H_0(\iota)$ is the projection onto the first summand. 
In particular, 
\[
H_0(\iota):H_0(\cG,\Z/2)\to H_0(\cH,\Z/2)
\quad\text{is injective}
\iff\det N\text{ is odd}. 
\]

The abelianization of the topological full group of an SFT groupoid 
was computed in \cite[Corollary 6.24]{Ma15crelle}. 
By pure infiniteness, 
$\cK$ contains $p$ disjoint clopen copies of its unit space, 
each implemented by a bisection. 
Thus 
$\cH\cong\cK\times\mathcal R_p$ is isomorphic to a clopen reduction of $\cK$, 
to which the same result applies. 
We obtain short exact sequences 
\[
\xymatrix@M=8pt{
0 \ar[r] & H_0(\cH,\Z/2) \ar[r]^-{\zeta_\cH} &
H_1(\sF(\cH)) \ar[r]^-{\eta_\cH} & H_1(\cH) \ar[r] & 0 \\
0 \ar[r] & H_0(\cG,\Z/2) \ar[r]^-{\zeta_\cG} \ar[u]_-{H_0(\iota)} &
H_1(\sF(\cG)) \ar[r]^-{\eta_\cG} \ar[u]_-{H_1(\iota_*)} &
H_1(\cG) \ar[r] \ar[u]_-{H_1(\iota)} & 0. 
}
\]
The diagram commutes by the same argument 
as in Remark~\ref{rem:Li-naturality}. 
If $\det N$ is odd, then $\det N\neq0$, 
so $N$ is injective and $\Ker N=0$. 
It follows from the preceding exact sequences that 
both outer vertical maps are injective, and therefore so is $H_1(\iota_*)$. 
Conversely, if $H_1(\iota_*)$ is injective, then the injectivity of 
$\zeta_\cG$ implies that $H_0(\iota)$ is injective. 
Hence $\det N$ is odd. 

Combining these observations with Theorem~\ref{thm1:subgroupoid} 
and Theorem~\ref{thm2:subgroupoid}, we obtain the following. 

\begin{proposition}\label{SFTsemidirect}
In the setting above, the following hold. 
\begin{enumerate}
\item The group $\sF(\cG)\rtimes_\lambda\Z/p$ is 
a maximal subgroup of $\sF(\cH)$ if and only if 
$\partial:\Ker L\to\Coker N$ is the zero map. 
\item The group $\sD(\cG)\rtimes_\lambda\Z/p$ is 
a maximal subgroup of $\sD(\cH)$ if and only if $\det N$ is odd. 
\end{enumerate}
In particular, if $\det N=\pm1$, then both groups above are maximal. 
\end{proposition}

\begin{proof}
(1)\:
The exact sequence above shows that $H_1(\iota)$ is surjective 
if and only if $\partial=0$. 
The conclusion follows from Theorem~\ref{thm1:subgroupoid}. 

(2)\:
As observed above, $H_1(\iota_*)$ is injective 
if and only if $\det N$ is odd. 
Since $p$ is odd, every $\lambda_a$ belongs to $\sD(\cH)$, 
and hence Theorem~\ref{thm2:subgroupoid} gives the conclusion. 

Finally, if $\det N=\pm1$, then $\Coker N=0$, so $\partial=0$, 
and $\det N$ is odd. 
\end{proof}

\begin{example}[Three-fold extensions]\label{SFTsemidirect3}
Let $V$ consist of one vertex, let $p=3$, and suppose that 
$E_0$ and $E_1$ consist of one loop each and $E_2$ is empty. 
Then $\cK$ is the SFT groupoid of the full two-shift. 
In this case $\cK\cong\cK\times\mathcal R_3$. 
Indeed, the three cylinder sets corresponding to the complete prefix set 
$\{0,10,11\}$ are mutually equivalent in $\cK$, and the reduction of $\cK$ 
to any one of them is again isomorphic to $\cK$. 
Since $\cH\cong\cK\times\mathcal R_3$, 
$\cH$ is therefore also isomorphic to the SFT groupoid of the full two-shift. 
The graph $(W,F)$ has three vertices, with one loop and one edge to 
the next vertex at each vertex. 
In this case $M$ is the cyclic permutation matrix on $\Z^3$, 
and its restriction $N$ to $\Lambda_0$ has determinant one. 
Moreover, if $B$ denotes the adjacency matrix of $(W,F)$, then 
$H_0(\cG)=\Coker(I-B^t)=0$ and $\det(I-B^t)=-1$. 
Since also $H_0(\cK)=0$ and $\det(I-A^t)=-1$, 
Matsumoto's classification theorem 
\cite[Theorem 1.1]{Ma13PAMS} (see also \cite[Theorem 6.2]{Ma15crelle}) 
implies $\cG\cong\cK$. 
Consequently, $\cG\cong\cH\cong\cK$. 
Since $H_0(\cG)=0$ and $H_1(\cG)=\Ker M=0$, 
the abelianization of $\sF(\cG)$ is trivial 
by \cite[Corollary 6.24]{Ma15crelle}, and hence 
$\sF(\cG)=\sD(\cG)$. 
The same holds for $\cH$. 
As $\sF(\cK)$ is isomorphic to the Higman--Thompson group $V_2$ 
(see \cite[Remark 6.3]{Ma15crelle}), 
Proposition~\ref{SFTsemidirect} shows that 
$V_2\rtimes_\lambda\Z/3$ is a maximal subgroup of $V_2$, 
where $\lambda$ is induced by the deck transformation of order three. 

This example also arises from the type systems of 
Belk--Bleak--Quick--Skipper. 
Let $\mathcal P$ be the type system on finite binary words 
in which the type of a word is the number of occurrences of $1$ modulo $3$. 
Its child rules are $P_i0=P_i$ and $P_i1=P_{i+1}$, with indices in $\Z/3$. 
This type system is nuclear and simple. 
Under the identification $(z,i)\mapsto\rho_i z$, 
where $(\rho_0,\rho_1,\rho_2)=(0,10,11)$, 
the subgroup $\sF(\cG)$ corresponds to $\Fix_{V_2}(\mathcal P)$. 
The prefix cycle $(0\ 10\ 11)$ implements the deck transformation 
and cyclically permutes the three types. 
Since every permutation of the types preserving the child rules is 
a cyclic permutation, 
$\sF(\cG)\rtimes_\lambda\Z/3$ corresponds to $\St_{V_2}(\mathcal P)$. 
Thus the maximality of this subgroup also follows 
from \cite[Theorem 5.12]{BBQS25TAMS}. 
\end{example}

\begin{example}
Let $V$ again consist of one vertex and let $p=3$, but now suppose that 
$E_0$ consists of one loop, $E_1$ consists of two loops and $E_2$ is empty. 
Then $\cK$ is the SFT groupoid of the full three-shift. 
If $P$ denotes the cyclic permutation matrix on $\Z^3$, then, 
up to replacing $P$ by its inverse, we have $L=[2]$ and $M=2P$. 
Thus $\Ker L=0$, so the connecting homomorphism $\partial$ is zero, 
whereas the restriction $N$ of $M$ to $\Lambda_0$ has determinant $4$. 
Here $\overline M=0$, and hence 
\[
H_0(\cG,\Z/2)\cong(\Z/2)^3,
\qquad
H_0(\cH,\Z/2)\cong\Z/2. 
\]
Since $H_1(\cG)=H_1(\cH)=0$, 
the abelianizations of $\sF(\cG)$ and $\sF(\cH)$ are therefore 
isomorphic to $(\Z/2)^3$ and $\Z/2$, respectively. 
In particular, $\sD(\cG)\subsetneq\sF(\cG)$ and 
$\sD(\cH)\subsetneq\sF(\cH)$. 
Proposition~\ref{SFTsemidirect} shows that 
$\sF(\cG)\rtimes_\lambda\Z/3$ is a maximal subgroup of $\sF(\cH)$, 
while $\sD(\cG)\rtimes_\lambda\Z/3$ is not a maximal subgroup of $\sD(\cH)$. 
Thus the maximality phenomena for the full group and the derived group 
can behave differently even in this simple one-vertex situation. 
\end{example}

\begin{remark}
The construction, the integral exact sequence and 
Proposition~\ref{SFTsemidirect}~(1) remain valid for $p=2$. 
However, the diagonal map no longer splits $\Sigma$ modulo two, 
so the argument for part~(2) does not apply. 
Moreover, the deck involution need not belong to $\sD(\cH)$, 
and Theorems~\ref{thm2:subgroupoid} and~\ref{thm3:subgroupoid} 
must be distinguished. 
We do not pursue this case here. 
\end{remark}

\subsection{Orbit-breaking subgroupoids of minimal $\Z$ actions}
\label{subsec:orbitbreaking}

We conclude this section with a basic example of a prime wide open subgroupoid 
obtained by breaking a single orbit of a minimal $\Z$ action. 
This is not a direct application of Section~\ref{sec:subgroupoid}, 
since the ambient groupoid is not obtained from the subgroupoid 
by the finite cyclic semidirect product construction considered there. 
Nevertheless, primeness is immediate, 
and a direct argument gives maximal subgroups 
both in the kernel of the index map and in the derived group.

Let $\phi:X\to X$ be a minimal homeomorphism of a Cantor set, 
and let $\cG:=X\rtimes_\phi\Z$. 
Fix $x\in X$ and put 
\[
A:=\{\phi^n(x)\mid n\geq1\},\qquad 
B:=\{\phi^n(x)\mid n\leq0\}. 
\]
We let $\cK$ denote the orbit-breaking subgroupoid 
\[
\cK:=\{g\in\cG\mid r(g)\in A\iff s(g)\in A\}. 
\]
Thus $\cK$ is obtained by removing precisely the arrows between $A$ and $B$; 
our convention is to break the orbit between $x$ and $\phi(x)$. 
For each $n\in\Z$, only finitely many arrows in $X\times\{n\}$ are removed, 
so $\cK$ is a wide open subgroupoid of $\cG$. 
Its orbits are $A$, $B$ and the $\cG$-orbits other than that of $x$. 
In particular, $\cK$ is minimal, since both half-orbits are dense. 
Moreover, $\cK$ is AF (see \cite[Theorem 4.3]{GPS04ETDS}), 
and hence $\sF(\cK)$ is locally finite 
(see \cite[Proposition 3.2]{Ma06IJM}). 
Since AF groupoids are almost finite, 
Theorem~\ref{niceproperties} (1) also implies that $\sD(\cK)$ 
is simple; it is locally finite as a subgroup of $\sF(\cK)$. 
The inclusion $\cK\subset\cG$ is prime 
in the sense of Definition~\ref{def:prime}: 
any wide open subgroupoid of $\cG$ properly containing $\cK$ contains 
an arrow between $A$ and $B$, 
and multiplying it, or its inverse, on both sides by arrows of $\cK$ 
produces all the removed arrows. 

Let $I:\sF(\cG)\to H_1(\cG)\cong\Z$ be the index map. 
Since $\sF(\cK)$ is locally finite, it is contained in $\Ker I$. 
By the definition of $\cK$, we have 
\begin{equation}\label{eq:orbitbreakingstabilizer}
\sF(\cK)=\{\alpha\in\sF(\cG)\mid\alpha(A)=A\}. 
\end{equation}
Every $\alpha\in\sF(\cG)$ has 
a continuous, hence bounded, orbit cocycle $n_\alpha:X\to\Z$ 
satisfying $\alpha(y)=\phi^{n_\alpha(y)}(y)$. 
It follows that $A\setminus\alpha(A)$ and $\alpha(A)\setminus A$ are finite. 
If $\alpha\in\Ker I$, these two sets have the same cardinality 
(see \cite[Remark 5.6]{GPS99Israel} and \cite[proof of Lemma 4.1]{Ma06IJM}). 
For such an $\alpha$, define its crossing number by 
\[
d(\alpha):=\#\left(A\setminus\alpha(A)\right)
=\#\left(\alpha(A)\setminus A\right). 
\]
Then $d(\alpha)=0$ if and only if $\alpha\in\sF(\cK)$. 
We will first produce an element with crossing number one 
and then reduce arbitrary crossing numbers by induction. 

For the local constructions, we use the following consequence of 
Lemma~\ref{threecycle} (1) and (2). 
Given three distinct points in a single $\cK$-orbit, 
there exists a $3$-cycle in $\sD(\cK)$ 
which permutes these points in any prescribed cyclic order. 
The support can be chosen to avoid any prescribed finite set 
disjoint from the three points, by choosing the bisections sufficiently small. 
In particular, $\sD(\cK)$ acts transitively on $A\times B$: 
one first moves the $A$-coordinate by such a $3$-cycle 
fixing the $B$-coordinate, 
and then moves the $B$-coordinate 
while fixing the new $A$-coordinate.

\begin{lemma}\label{lem:orbitbreakingonecrossing}
For every $\gamma\in\Ker I\setminus\sF(\cK)$, 
there exists $k\in\sD(\cK)$ such that $d(\gamma^{-1}k\gamma)=1$. 
In particular, the group $\langle\sD(\cK),\gamma\rangle$ 
contains an element of $\sD(\cG)$ with crossing number one. 
\end{lemma}

\begin{proof}
Put $R:=A\setminus\gamma(A)$ and $E:=\gamma(A)\setminus A$. 
These are nonempty finite sets of the same cardinality. 
Since $R$ is finite, $A\cap\gamma(A)$ is infinite. 
Choose $a\in R$ and two distinct points $b,c\in A\cap\gamma(A)$. 
By the preceding observation, 
there exists $k\in\sD(\cK)$ which cyclically permutes $a,b,c$ 
in this order and fixes $(R\setminus\{a\})\cup E$ pointwise. 
We have $\gamma(A)=(A\setminus R)\cup E$, 
$k(R)=(R\setminus\{a\})\cup\{b\}$ and $k(E)=E$. 
Together with $k(A)=A$, this gives 
\[
k\gamma(A)=(A\setminus k(R))\cup E
=\left(\gamma(A)\setminus\{b\}\right)\cup\{a\}. 
\]
Consequently, for $h:=\gamma^{-1}k\gamma$, we have 
\[
h(A)=\left(A\setminus\{\gamma^{-1}(b)\}\right)\cup\{\gamma^{-1}(a)\}. 
\]
Here $\gamma^{-1}(b)\in A$ and $\gamma^{-1}(a)\in B$. 
Thus $d(h)=1$. 
Finally, $h$ belongs to $\sD(\cG)$, 
since $k\in\sD(\cK)\subset\sD(\cG)$ and $\sD(\cG)$ is normal in $\sF(\cG)$. 
\end{proof}

\begin{theorem}\label{thm:orbitbreakingmaximal}
In the setting above, the following hold. 
\begin{enumerate}
\item The group $\sF(\cK)$ is a maximal subgroup of $\Ker I$. 
\item The group $\sD(\cK)$ is a maximal subgroup of $\sD(\cG)$. 
\end{enumerate}
\end{theorem}

\begin{proof}
We first verify the equality 
\begin{equation}\label{eq:orbitbreakingintersection}
\sF(\cK)\cap\sD(\cG)=\sD(\cK). 
\end{equation}
The inclusion $\iota:\cK\to\cG$ induces an isomorphism 
$H_0(\iota):H_0(\cK)\to H_0(\cG)$ 
(as recalled in \cite[Section 4]{Ma06IJM}). 
By \cite[Lemma 3.5 and Theorem 4.8]{Ma06IJM}, 
the signature homomorphisms 
\[
\sF(\cK)\longrightarrow H_0(\cK,\Z/2),\qquad 
\Ker I\longrightarrow H_0(\cG,\Z/2)
\]
are surjective and have kernels $\sD(\cK)$ and $\sD(\cG)$, respectively. 
Each sends a transposition $\tau_U$ to $[1_{r(U)}]$ 
in the corresponding zeroth homology group. 
Since $\cK$ is AF, $\sF(\cK)$ is generated by transpositions. 
Thus the two signature homomorphisms agree on $\sF(\cK)$ 
under $H_0(\iota)$, which is also an isomorphism with $\Z/2$ coefficients. 
Consequently, inclusion induces an injective homomorphism 
\[
\sF(\cK)/\sD(\cK)\longrightarrow\Ker I/\sD(\cG),
\]
which proves \eqref{eq:orbitbreakingintersection}. 

We now prove both assertions simultaneously. 
Let $(\Gamma,S)$ be either 
\[
(\Ker I,\sF(\cK))
\quad\text{or}\quad
(\sD(\cG),\sD(\cK)). 
\]
By \eqref{eq:orbitbreakingstabilizer} and 
\eqref{eq:orbitbreakingintersection}, 
$S$ is the setwise stabilizer of $A$ in $\Gamma$, and $\sD(\cK)\subset S$. 
Both inclusions $S\subset\Gamma$ are proper. 
In fact, 
Lemma~\ref{threecycle} (1) and (2) provides a $3$-cycle 
in $\sD(\cG)\subset\Gamma$ which cyclically permutes 
one point of $B$ and two points of $A$, and hence does not preserve $A$. 

Suppose that $S\subsetneq H\subset\Gamma$, and take $\gamma\in H\setminus S$. 
Then $\gamma\in\Ker I\setminus\sF(\cK)$, and so 
Lemma~\ref{lem:orbitbreakingonecrossing} provides $h\in H$ with $d(h)=1$. 
Write 
\[
h(A)=(A\setminus\{a_0\})\cup\{b_0\},\qquad a_0\in A,\ b_0\in B. 
\]
Since $\sD(\cK)$ acts transitively on $A\times B$, 
for any $a\in A$ and $b\in B$ there exists $k\in\sD(\cK)$ 
with $k(a_0)=a$ and $k(b_0)=b$. 
Thus $h_{a,b}:=kh\in H$ satisfies 
\[
h_{a,b}(A)=(A\setminus\{a\})\cup\{b\}. 
\]

For $\alpha\in\Gamma$, we prove $\alpha\in H$ by induction on $d(\alpha)$. 
If $d(\alpha)=0$, then $\alpha\in S\subset H$. 
Otherwise, choose $a\in A\setminus\alpha(A)$ and $b\in\alpha(A)\setminus A$. 
Then $h_{a,b}^{-1}\alpha\in\Gamma$ and 
\[
d(h_{a,b}^{-1}\alpha)
=\#\left(h_{a,b}(A)\setminus\alpha(A)\right)
=\#\left((A\setminus\alpha(A))\setminus\{a\}\right)
=d(\alpha)-1. 
\]
The induction hypothesis gives $h_{a,b}^{-1}\alpha\in H$, 
and hence $\alpha\in H$. 
Therefore $H=\Gamma$, proving both assertions. 
\end{proof}

We finish by computing the corresponding normalizers. 
Recall the notation $\Nor(\Gamma,S)$ introduced 
in Section~\ref{subsec:homology}. 

\begin{lemma}\label{lem:orbitbreakingnormalizer}
We have 
\[
\begin{aligned}
\Nor(\sF(\cG),\sF(\cK))
&=\Nor(\sF(\cG),\sD(\cK))=\sF(\cK),\\
\Nor(\sD(\cG),\sD(\cK))&=\sD(\cK). 
\end{aligned}
\]
\end{lemma}

\begin{proof}
Let $\alpha\in\Nor(\sF(\cG),\sD(\cK))$. 
Since $n_\alpha$ is bounded, there exists $y=\phi^m(x)\in A$ 
with $m$ sufficiently large that $\alpha(y)\in A$. 
The group $\sD(\cK)$ acts transitively on $A$, and hence 
\[
\alpha(A)=\alpha(\sD(\cK)y)=(\alpha\sD(\cK)\alpha^{-1})\alpha(y)
=\sD(\cK)\alpha(y)=A. 
\]
By \eqref{eq:orbitbreakingstabilizer}, $\alpha\in\sF(\cK)$. 
The reverse inclusion follows 
because $\sD(\cK)$ is the derived subgroup of $\sF(\cK)$. 
As $\sD(\cK)$ is characteristic in $\sF(\cK)$, we also have 
\[
\sF(\cK)\subset\Nor(\sF(\cG),\sF(\cK))
\subset\Nor(\sF(\cG),\sD(\cK))=\sF(\cK). 
\]
Finally, intersecting $\Nor(\sF(\cG),\sD(\cK))=\sF(\cK)$ 
with $\sD(\cG)$ and using \eqref{eq:orbitbreakingintersection} 
gives the equality on the second line. 
\end{proof}

In Section~\ref{sec:subgroupoid}, the finite cyclic action supplies 
additional normalizing elements, producing a semidirect product 
strictly larger than the full group of the subgroupoid. 
Here a normalizing element $\alpha\in\sF(\cG)$ must permute 
the two $\cK$-orbits $A$ and $B$ inside the $\cG$-orbit of $x$, 
but the boundedness of $n_\alpha$ prevents it 
from interchanging these half-orbits. 
There is consequently no analogous enlargement by a finite cyclic group. 
In particular, $\sF(\cK)$ is self-normalizing in $\sF(\cG)$, 
and $\sD(\cK)$ is self-normalizing in $\sD(\cG)$. 
Nevertheless, $\sF(\cK)$ is not a maximal subgroup of $\sF(\cG)$, 
since $\sF(\cK)\subsetneq\Ker I\subsetneq\sF(\cG)$.

\section*{Acknowledgements}
The author is grateful to K.~Matsumoto for helpful correspondence
on factor maps between shifts of finite type and for pointing out
relevant results in symbolic dynamics at an early stage of this work.
The author is grateful to T.~Takeishi for pointing out the error
in \cite{Ma12PLMS} discussed in Remark~\ref{rem:pullback-correction}.

The author used ChatGPT (OpenAI, including GPT-6 Astra)
and Claude (Anthropic, including Opus 5)
to assist with mathematical discussions, checking arguments,
literature searches, and the drafting and revision of the manuscript.
The mathematical arguments and references were independently
verified by the author, who is responsible for the final content.

\bibliography{../zzz}
\bibliographystyle{amsplain}

\end{document}